\documentclass[12pt,reqno]{amsart}

\usepackage{amsmath}
\usepackage{amssymb}
\usepackage{amsthm}
\usepackage{amscd, amsfonts, mathrsfs}
\usepackage{amsaddr}

\usepackage{cases}
\usepackage{xfrac}

\usepackage[dvips]{epsfig}
\usepackage{verbatim} 

\usepackage{epsf}
\usepackage[bookmarksnumbered,pdfpagelabels=true,plainpages=false,colorlinks=true,
            linkcolor=black,citecolor=black,urlcolor=blue]{hyperref}

\theoremstyle{plain}
\newtheorem{theorem}{Theorem}[section]
\newtheorem{theo}{Theorem}[section]
\newtheorem{lemma}[theorem]{Lemma}
\newtheorem{prop}[theorem]{Proposition}

\theoremstyle{definition}
\newtheorem{rema}[theorem]{Remark}

\newtheorem{assump}{Assumption}[section]

\numberwithin{equation}{section}

\newenvironment{ sys_eq }{\ left \ lbrace \ begin { array }{ @ {} l@ {}}}{\ end { array }\ right .}

\newcommand{\al}{\alpha}

\newcommand{\de}{\delta}
\newcommand{\ep}{\epsilon}
\newcommand{\la}{\lambda}

\newcommand{\Om}{\Omega}

\newcommand{\TT}{\mathbb T}
\newcommand{\RR}{\mathbb R}

\newcommand{\ZZ}{\mathbb Z}

\newcommand{\rar}{\rightarrow}

\newcommand{\dive}{\operatorname{div}}

\newcommand{\p}{\parallel}
\newcommand{\n}{\parallel}
\newcommand{\tl}{\Xi}

\newcommand{\half}{\sfrac{1}{2}}

\title[Slightly compressible fluids II]{Motion of slightly compressible fluids in a bounded 
domain. II}

\author[Disconzi]{Marcelo M. Disconzi}
\address{\vspace{-0.5cm}  Department of Mathematics\\
Vanderbilt University\\ Nashville, TN 37240, USA}
\email{marcelo.disconzi@vanderbilt.edu}
\thanks{Marcelo M. Disconzi is partially supported by NSF grant
1305705.}

\author[Ebin]{David G. Ebin}
\address{\vspace{-0.5cm} Department of Mathematics\\
Stony Brook University\\ Stony Brook, NY 11794-3651, USA}
\email{ebin@math.sunysb.edu}

\begin{document}

\begin{abstract}
We study the problem of inviscid slightly compressible fluids in a bounded domain. 
We find a unique solution to the initial-boundary value problem and show that it is near
 the analogous solution for an incompressible fluid provided the initial 
conditions for the two problems are close. 
In particular, the divergence of the initial velocity of the compressible flow at time zero is assumed to be small.
Furthermore we find that solutions to the compressible motion problem
in Lagrangian coordinates depend differentiably on their initial data,
an unexpected result for this type of non-linear equations.  

\end{abstract}

\maketitle

\tableofcontents

\section{Introduction \label{intro}}

This paper is a continuation of \cite{E4} in which we prove the statements that were announced but not proven there. In that paper we studied the initial value problem for the motion of inviscid fluids in a bounded domain.  We looked at both compressible and incompressible fluids and showed that the flow of the former is close to that of the latter if the compressibility is low 
(or, what is the same, if the sound speed is large)
and the initial conditions for two flows are close. Since the divergence of the velocity of the incompressible fluid is zero, this requires that the divergence of the initial velocity of the  compressible fluid be small.
For the reader's convenience, we comment on the differences between this work and \cite{E4} in section
\ref{review}, after stating the main theorems. A discussion of the main results and a brief review of the literature is also 
carried out in section \ref{review}.

In \cite{E4} we announced but did not prove that for compressible fluids the fluid density has one more spatial derivative than the velocity.  This fact is curious because as is well known the fluid density and velocity together satisfy a quasilinear symmetric hyperbolic system and thus one would not expect either to be smoother than the other.  Furthermore, as a consequence of the extra smoothness we find that the solution of the system, when expressed in Lagrangian coordinates is differentiably dependent on the initial data.  This is unexpected because for quasilinear systems the dependence is usually continuous, but not differentiable. See for example \cite{K2},
and section \ref{review}.

 We note that in a previous paper \cite{E2} similar results were proven for fluid motions which are periodic in 
space, but here we prove them in a domain with boundary. The boundary case is more difficult because it 
involves more complicated estimates (cf. section \ref{estimates}), but it is also more important because most fluid problems do 
involve boundaries. Furthermore, our proof is more direct than of \cite{E2} in that it avoids operator theoretic methods.

Finally, (section \ref{assymp}) following a suggestion of Professor C. S. Morawetz, we explain how the construction used to show 
that slightly compressible motion is near to incompressible actually gives us a sequence of successive 
approximations to the compressible motion. The first term in the sequence is the incompressible motion, 
and the second is derived by studying the propagation of sound caused by the first motion.

We will use the same notation as in \cite{E4} and for the reader's convenience repeat some of the formulas and estimates written there.
Thus throughout the paper, $u,\rho$ and $p$ will denote the velocity, density, and pressure of compressible motion. They are 
all functions of time $t$ and of position $x=(x^1,\dots,x^n)$; that is, $u(t,x)$ is the velocity of that fluid particle which 
is at position $x$ at time $t$; $x$ is known as the Euler coordinate of the particle.

Given the velocity $u(t,x)$, one can find the flow of the fluid, $\zeta(t,x)$, as the solution of the ordinary 
differential equation
\begin{gather}
 \partial_t \zeta(t,x) = u(t,\zeta(t,x)), ~~~ \zeta(0,x) = x.
\nonumber
\end{gather}
Then $\zeta(t,x)$ is the position at time $t$ of the fluid particle which was at $x$ at time zero. $x$ is known as the 
Lagrange coordinate of the particle; its Euler coordinate at time $t$ is $\zeta(t,x)$.

For incompressible fluid motion the velocity will be called $v=v(t,x)$, and will also be the velocity of the particle with Euler 
coordinate $x$. Its flow will be called $\eta(t,x)$ so $\eta$ will satisfy
\begin{gather}
 \partial_t \eta(t,x) = v(t,\eta(t,x)), ~~~ \eta(0,x)=x.
\nonumber
\end{gather}
The density of an incompressible fluid is always taken to be $1$, and the pressure, if used, will also be called $p$. However, 
strictly speaking, there is no notion of pressure for the idealized incompressible fluid. This is because of the following:
physical fluids are of course at least slightly compressible, so the solutions of the equations of idealized incompressible fluids have physical meaning only if they are close to the solutions of the equations of compressible fluids.  We prove that when the compressibility is small, the positions and velocities of the two types of flows are close, but we  
do not prove anything about the accelerations.  In fact one can find examples of incompressible flows for which the accelerations are not close.  Thus the pressures, whose gradients are the accelerations, are also not close so $\nabla p$ for the incompressible flow is not close to $\nabla p$ for the compressible flow and hence has no direct physical 
meaning (see, however, remark \ref{remark_pressure_dist}).
 
We can in fact consider the following concrete example: consider a fluid filling a disk, and let it move by spinning the disk as a rigid body at a velocity of $\frac{\partial}{\partial \theta}$ where $\theta$ is the angular coordinate.  This is a steady state incompressible motion whose ``pressure" is $-r^2 /2$ where $r$ is the radial coordinate.  Now take a family of compressible fluids governed by the relation $p_k = k \log \rho$ with $k$ large.  These fluids will oscillate towards the edge of the disk with an amplitude proportional to $1/k$ and a frequency like $\sqrt{k}.$  The radial velocity will have an amplitude like $1/\sqrt{k}$ and the acceleration will oscillate with an amplitude like 1.  Hence the pressure will also oscillate   
with an amplitude like 1 so it will not be close to the ``pressure" of the incompressible fluid.  Thus we find that the motion in the incompressible case is physically relevant, but the pressure is not (this phenomenon of convergence of the zeroth and first time derivatives, but not higher ones is typical of motions with strong constraining force; see \cite{E2}). 
Note that in the incompressible case, the pressure is not an independent unknown. At any instant it is determined entirely by the velocity at that instant (see equation (\ref{eq_v_no_p})).

Our analysis of fluid motion will be partly in Eulerian and partly in Lagrangian coordinates. Spatial derivatives and 
many $L^2$ estimates involving them will be done using Eulerian coordinates. Time derivatives will usually be written in Lagrangian
coordinates or (what is essentially the same) as material derivatives in Eulerian coordinates. This seems to be the best way 
to study the situation because the $L^2$ estimates are simpler in Eulerian coordinates, but the Lagrangian coordinates 
make it possible to write some of the equations as ordinary differential equations on a function space. In fact 
one result, the $C^1$ dependence on initial data, is true in Lagrangian coordinates, but not in Eulerian coordinates.

\section{Equations of Fluid Motion \label{equations_fluid}}
We consider the motion of a fluid which fills a fixed bounded domain $\Om \subset \RR^n$, and we assume that $\Om$ has a 
smooth boundary $\partial \Om$, which is a compact $(n-1)$-dimensional submanifold of $\RR^n$.

In the compressible case, we restrict our attention to barotropic fluids, (fluids whose pressure depends only on the density). Their equations of motion are:

\begin{subnumcases}{\label{eq_comp}}
\frac{\partial u}{\partial t} + \nabla_u u = -\frac{1}{\rho}\nabla p, \label{eq_comp_u} \\
\frac{\partial \rho}{\partial t} + \operatorname{div}(\rho u) = 0 ,
\label{eq_comp_rho} 
\end{subnumcases}

For the incompressible case, the equations are:
\begin{subnumcases}{\label{eq_incomp}}
 \frac{\partial v}{\partial t} + \nabla_v v = -\nabla p, \label{eq_incomp_v}\\
\operatorname{div}(v) = 0 .
\label{eq_incomp_rho}
\end{subnumcases}
Here, $\nabla$ means gradient, $\operatorname{div}$ means divergence and $\nabla_w$ means derivative in 
direction $w$, $w$ a vector.  Thus in coordinates, where all indices run from $1$ to $n$ and repeated indices are summed:
\begin{gather}
(\nabla_v v)^j = v^i \frac{\partial v^j}{\partial x^i}\hspace{.1in} {\rm and}\hspace{.1in} \operatorname{div}(\rho u) = \frac{\partial(\rho u^i)}{\partial x^i}
\nonumber
\end{gather}

The equations of motion are accompanied by the following boundary conditions which express the physical 
condition that the fluid does not pass through the boundary. To write them we let $\nu$ be the unit outward normal vector 
field on $\partial \Om$. Then, for compressible motion,
\begin{gather}
\langle u, \nu \rangle = 0, \,\, \text{ which implies } \, \,  \langle \nabla_u u, \nu \rangle = -\frac{1}{\rho} \nabla_\nu p ,
\label{bc_comp}
\end{gather}
and for incompressible motion,
\begin{gather}
\langle v, \nu \rangle = 0, \,\, \text{ which implies } \,\, \langle \nabla_v v, \nu \rangle = - \nabla_\nu p .
\label{bc_incomp}
\end{gather}
Thus, equations (\ref{eq_comp_u})-(\ref{eq_comp_rho}) with boundary condition (\ref{bc_comp}) shall be known as the compressible
motion problem, while (\ref{eq_incomp_v})-(\ref{eq_incomp_rho}) with boundary condition
(\ref{bc_incomp}) shall be the incompressible motion problem.

For the compressible case the pressure can be written as a function of the density so in addition to  (\ref{eq_comp_u})-(\ref{eq_comp_rho}) we specify this function, writing
\begin{gather}
 p=p(\rho).
\label{p_of_rho}
\end{gather}
The core of this paper will be an analysis of the density for a compressible motion. We shall show that it satisfies a 
hyperbolic equation and use this equation to make estimates. As density is positive we can take its logarithm which turns out to be more convenient to work with. Thus we define the function 
\begin{gather}
 f = \log \rho .
\nonumber
\end{gather}
In addition to partial derivatives of $f$ we shall use the material derivative which we denote by `` $\dot{}$ ''. This derivative 
depends on the fluid velocity which we here label $u=u(t,x)$. Specifically,
 for any function  $h$ depending on $x$ and $t$ we 
define the material derivative of $h$ with respect to a fluid moving with velocity $u$ to be
\begin{gather}
 \dot{h} = \frac{\partial h}{\partial t} + \nabla_u h .
\nonumber
\end{gather}
Using (\ref{eq_comp_rho}) we find:
\begin{gather}
 \dot{f} = \frac{1}{\rho}\dot{\rho} = -\operatorname{div}(u) .
\label{f_dot_div_u}
\end{gather}
Taking another material derivative and using (\ref{eq_comp_u}) we find
\begin{gather}
 \ddot{f} = \operatorname{div}(p^\prime(\rho)\frac{1}{\rho}\nabla \rho ) + u^i_j u^j_i
\label{eq_f_1}
\end{gather}
where $u^i_j$ means $\frac{\partial u^i}{\partial x^j}$ and we sum over repeated 
indices. We shall rewrite (\ref{eq_f_1}) as a hyperbolic equation in $f$.

In order to rewrite (\ref{eq_f_1}) we shall need additional notation. We note that for physical reasons $p^\prime(\rho)$ is 
non-negative (pressure cannot decrease as density increases). Thus we can define a function $c$ by
\begin{gather}
 c = \sqrt{p^\prime(\rho)} .
\label{definition_of_c}
\end{gather}
This ``$c$'' is the sound speed of the fluid, and, like $\rho$, it depends on $x$ and $t$. 

Also, we introduce the operator $\de = -\operatorname{div}$ which is the formal adjoint of $\nabla$.

With this notation (\ref{eq_f_1}) can be rewritten as
\begin{gather}
 \ddot f = - \de c^2 \nabla f + u^i_j u^j_i
\label{eq_f_convected}
\end{gather}
and can be regarded as a linear hyperbolic equation in $f$ with inhomogeneous term $u^i_j u^j_i$ (although the 
equation is not quite linear since $c$ depends on $f$). It is sometimes referred to as a convected wave equation --- convected because the usual time derivative is 
replaced by the material derivative.

The boundary condition (\ref{bc_comp}) gives a boundary condition for (\ref{eq_f_convected}) as well. From our definitions 
we find that
\begin{gather}
 \frac{1}{\rho} \nabla p = c^2 \nabla f,
\nonumber
\end{gather}
so (\ref{bc_comp}) implies
\begin{gather}
 \nabla_\nu f = -\frac{1}{c^2}\langle \nabla_u u, \nu \rangle .
\label{normal_der_f}
\end{gather}
This is an inhomogeneous Neumann boundary condition for the equation (\ref{eq_f_convected}).

For incompressible fluid motion, the 
density is constant (in (\ref{eq_incomp_v}) we have taken it to be equal to
 $1$), and $p$ does not correspond to any physical quantity (as explained in the introduction) and can in fact be written in terms of
 $v$, as we now show. 
In fact for any given time $t$, $p(t)$ is determined by $v(t)$ (up to a constant), although the determination is not local (in $x$). 
This is easily shown as follows: (\ref{eq_incomp_rho}) implies that
\begin{gather}
 \de \big ( \frac{\partial v}{\partial t} \big ) = 0, \,\, \text{ so } \, \,
  \de \big ( \nabla_v v \big ) = -\delta \nabla p.
\nonumber
\end{gather}
But since 
\begin{gather}
 -\de \nabla p = \Delta p 
\nonumber
\end{gather}
we get
\begin{gather}
 \Delta p = \de \big( \nabla_v v \big ) .
\label{lap_p}
\end{gather}
Also (\ref{bc_incomp}) gives a Neumann boundary condition for $p$, so given $v$ and (\ref{bc_incomp}), (\ref{lap_p}) has a 
solution $p$ which is unique modulo an additive constant. $\nabla p$ is thus determined uniquely by $v$.

Given a vector field $w$ in $\Om$, if we define an operator $Q$ by 
\begin{gather}
 Q(w) = \nabla g,
\nonumber
\end{gather}
where $g$ solves
\begin{gather}
\begin{cases}
 \Delta g = -\de w ~~ \text{ in } \Om ,\\
\nabla_\nu g = \langle w, \nu \rangle ~~\text{ on } \partial \Om,
\end{cases}
\nonumber
\end{gather}
then the system 
(\ref{eq_incomp_v})-(\ref{eq_incomp_rho}) with (\ref{bc_incomp}) is equivalent to
\begin{gather}
 \frac{\partial v}{\partial t} + \nabla_v v = Q(\nabla_v v)
\label{eq_v_no_p}
\end{gather}
so the artificial $p$ is eliminated.

The operator $Q$ is useful in the compressible motion as well. We let $P=I-Q$, where $I$ is the identity operator, and 
decompose the velocity as
\begin{gather}
 u = P(u) + Q(u) = w + \nabla g.
\nonumber
\end{gather}
This is the well known decomposition of a vector field into its solenoidal (or divergence free) and gradient parts. We shall 
apply $P$ and $Q$ to (\ref{eq_comp_u}) to get equations for $w$ and $\nabla g$.

First note that $Q^2 = Q$, so $P^2 = P$. Also, for any function $g$, $Q(\nabla g) = \nabla g$, and $P(\nabla g)=0$. Furthermore 
since $p$ is a function of $\rho$
\begin{gather}
 \frac{1}{\rho}\nabla p = \frac{1}{\rho}p^\prime(\rho)\nabla \rho = \nabla \int_1^\rho \frac{1}{\la}p^\prime(\la)d\la
\label{eq_grad_pressure}
\end{gather}
so
\begin{gather}
 P(\frac{\partial u}{\partial t} + \nabla_u u ) = 0 .
\label{P_applied_eq_u}
\end{gather}
Also
\begin{gather}
 \nabla_{\nabla g}\nabla g = \frac{1}{2} \nabla \langle \nabla g, \nabla g \rangle .
\label{double_grad_g}
\end{gather}
From (\ref{P_applied_eq_u}) and (\ref{double_grad_g}) we get
\begin{gather}
 \frac{\partial w}{\partial t} + P(\nabla_u w) = -P(\nabla_w \nabla g) .
\label{eq_w}
\end{gather}
Applying $Q$ to (\ref{eq_comp_u}) we can get a similar equation for $\nabla g$. However, since $\de \nabla g = \de u$ and 
since by (\ref{f_dot_div_u}) we have
\begin{gather}
 \dot{f} = \de u,
\nonumber
\end{gather}
we can investigate $\nabla g$ by studying $\dot{f}$.

We are concerned with the initial-boundary value problem for our fluid motion, so we prescribe $u_0(x) = u(0,x)$ and 
$\rho_0(x) = \rho(0,x)$, the initial velocity and density. As we showed in \cite{E3}, the system (\ref{eq_comp_u})-(\ref{eq_comp_rho}) and (\ref{p_of_rho}) with initial data $u_0$, $\rho_0$ and with boundary condition (\ref{bc_comp}) has a 
unique solution on some time interval provided $u_0$ and $\rho_0$ satisfy the following conditions:

a) For each $x \in \Om$, $|u_0(x)| < c(0,x) = \sqrt{p^\prime(\rho_0(x))}$;

b) $\rho_0$ is close to a constant;

c) $u_0$, $\rho_0$ satisfy compatibility conditions on $\partial \Om$. That is,
\begin{gather}
 \langle u_0, \nu \rangle = 0, \nonumber \\
 \langle \nabla_{u_0} u_0 + \frac{1}{\rho_0}p^\prime(\rho_0) \nabla \rho_0, \nu \rangle = 0, 
\nonumber
\end{gather}
etc.

Since we want to compare the solutions of (\ref{eq_comp_u})-(\ref{eq_comp_rho}) to that of (\ref{eq_incomp_v})-(\ref{eq_incomp_rho}) 
when the initial data are close we will assume that $v_0$, the initial datum for 
(\ref{eq_incomp_v})-(\ref{eq_incomp_rho}) is near $u_0$. Since $v$ must satisfy (\ref{eq_incomp_rho}) we have
$\operatorname{div}(v_0) = 0$, and therefore $\operatorname{div}(u_0)$ must be small. Then from (\ref{eq_f_1}) we find that
$\dot{f}(0)$ is small.

Also since $\rho$ is taken to be identically one in the incompressible fluid problem, we want $\rho_0$ to be nearly one as well. 
This means that $f(0) = \log \rho_0$ must also be small.

In section \ref{estimates}, we shall show that $f(0)$ and $\dot{f}(0)$ being small implies that $f(t)$ and $\dot{f}(t)$ are 
small as well.

\section{The Equation of State, Main Theorems, and discussion}

In this section we fix our notation and state the main results. We also give a brief discussion 
of similar results in the literature. 

\subsection{Main results\label{func_space}}

We begin by fixing our notation. We denote by $\nabla^\ell g$ the vector valued function consisting of all $\ell^{\text{th}}$ order partial derivatives of a function $g$.  $C^k(\Om,\RR^m)$ denotes the spaces of  $k$-times continuously differentiable 
functions on $\Om$ with values in $\RR^m$, and 
 $H^s(\Om,\RR^m)$ denotes the Sobolev space of $s$-times (weakly) differentiable functions whose derivatives belong
 to $L^2$, with the understanding that Fourier transform is employed to interpret non-integer or negative values of $s$, so that
 $s$ can be any real number \cite{P}. We sometimes write simply $C^k(\Om)$ and $H^s(\Om)$ or even $C^k$ and $H^s$.
 The norms on these spaces are denoted $\parallel \cdot \parallel_{C^k} $ and $\parallel \cdot \parallel_s$.
 The Sobolev norm on $H^s(\partial \Om)$ is denoted $\parallel \cdot \parallel_{\partial,s}$.
For  $s > \frac{n}{2} + 1$, we denote by $  \mathscr{D}^s(\Om) \equiv \mathscr{D}^s$ 
the group of $H^s$ diffeomorphisms of the domain $\Om$. All these definitions, along with some known results
that will be used in the paper, are recalled in the appendix. 
In particular, we point out  
inequality (\ref{product_estimate}) which will be used throughout the paper to estimate
the several products involved.

Our estimates on $f$ and $\dot{f}$ will be made by considering $f(t),~\dot{f}(t): \Om \rar \RR$. We shall 
estimate for each $t$ the $L^2$ norms of these functions. Our estimates for $f$ and $\dot{f}$ will depend on the fluid motion $\zeta(t)$, but also on the equation of 
state (\ref{p_of_rho}); i.e., on the relationship between $p$ and $\rho$. We are interested in slightly compressible 
fluids and their motion is governed by functions $p(\rho)$ where $p^\prime(\rho) = \frac{dp}{d\rho}$ is large, which is equivalent to the sound speed's being large.
Thus we shall 
consider a family $\{ p_k(\rho) \}$ parameterized by $k \in \RR_+$, and we shall assume that the parameterization is chosen such that:
\begin{gather}
p^\prime_k(\rho)_{| \rho=1} = k. 
\nonumber
\end{gather}
We are concerned with the fluid motion when $k$ is large and in its limit as $k \rar \infty$.

We shall need conditions on $p_k(\rho)$ for $\rho$ near $1$ so we make the following assumption:
\begin{assump} There exist positive constants $a_0$ and $a_1$ such that
\begin{gather}
 \parallel p_k^{(\ell)}(\rho)\parallel_{s+1} \leq k a_1 ~~~ \ell=1,2,\dots,s+2
\nonumber
\end{gather}
provided that $\parallel \rho -1 \parallel_{s+1} \leq \frac{a_0}{\sqrt{k}}$.
\label{assumption_on_rho}
\end{assump}
Assumption \ref{assumption_on_rho} involves the norm of a composite function; that is $\rho$ and $p^{(\ell)}(\rho)$ are 
considered as functions on $\Om$, the latter being $x \mapsto p^{(\ell)}(\rho(t,x))$. The $\parallel~\parallel_s$ norms of
such functions are estimated in a straightforward manner using the derivatives of $p$ and $\parallel \rho \parallel_s$ (see 
for example Lemma A.2 of \cite{BB}).

Assumption \ref{assumption_on_rho} is actually more restrictive than necessary since it bounds several derivatives by a 
multiple of $k$. It is possible to use a more general assumption, as was done in 
equations (5.3)-(5.4) of \cite{E2}, but this would make 
the argument more complicated. Hence we content ourselves to the less general case and leave extensions to the interested reader.

Before starting our theorems, we mention the compatibility conditions required for an $H^s$ compressible motion. As is 
stated in \cite{E3}, the equations (\ref{eq_comp_u})-(\ref{eq_comp_rho}) admit an $H^s$ solution $u$, $\rho$ only if 
the initial data $u_0$, $\rho_0$ satisfy compatibility conditions up to order $s-1$. The zeroth order condition is simply
\begin{gather}
 \langle u_0,\nu\rangle = 0 ~\text{ on }~ \partial \Om .
\label{zeroth_ord_comp_cond}
\end{gather}
All time derivatives of $u$ must also have zero normal component and using this and the equations of motion we get the 
first order compatibility condition:
\begin{gather}
 \langle \nabla_{u_0} u_0 + c^2_0 \nabla f_0,\nu\rangle = 0 ~\text{ on }~ \partial \Om,
\label{first_ord_comp_cond}
\end{gather}
where $c_0 = \sqrt{p^\prime(\rho_0)}$ is the sound speed of the fluid with density $\rho_0$ and $f_0 = \log\rho_0$.

Taking another derivative we get
\begin{gather}
 \langle \nabla_{u} \partial_t u + \nabla_{\partial_t u} u + \partial_t(c^2)\nabla f + c^2 \nabla \partial_t f, \nu \rangle = 0
\text{ on } \partial \Om .
\label{another_der_first_ord}
\end{gather}
To analyze (\ref{another_der_first_ord}) we use the so called second fundamental form of $\partial \Om$ which we call $S_2$. 
Given any two vector fields $z_1$ and $z_2$ on $\partial \Om$ we define 
\begin{gather}
 S_2(z_1,z_2) = \langle \nabla_{z_1}z_2,\nu\rangle = - \langle z_1,\nabla_{z_2} \nu \rangle 
\nonumber
\end{gather}
which is a function on $\partial \Om$. $S_2$ is easily seen to be symmetric in $z_1$ and $z_2$, and at 
each $p \in \partial \Om$ it depends only on $z_1(p)$ and $z_2(p)$ and not on their derivatives. Using 
it, (\ref{another_der_first_ord}) can be written
\begin{gather}
 2 S_2(u,\partial_t u) + \partial_t(c^2)\nabla_\nu f + c^2 \nabla_\nu \partial_t f=0 \, \text{ on } \, \partial \Om.
\nonumber
\end{gather}
Substituting for the time derivatives using (\ref{eq_comp_u}) and (\ref{eq_comp_rho}) we get the second order 
compatibility condition
\begin{align}
\begin{split}
& -2 S_2(u_0,\nabla_{u_0}u_0 + c_0^2 \nabla f_0) + p^{\prime\prime}(\rho_0)\rho_0(-\nabla_{u_0}f_0 
 + \de u_0)\nabla_\nu f_0 \\
 & + c_0^2\nabla_\nu(-\nabla_{u_0}f_0 + \de u_0 ) = 0 \, \text{ on } \, \partial \Om.
\end{split}
\label{second_ord_comp_cond}
\end{align}
Higher order conditions are computed in the same way. One sees inductively that the condition of order $\ell$ 
depends on $\ell$ derivatives of $f_0$ or $\rho_0$ but only $\ell-1$ derivatives of $u_0$ and $\de u_0$. That is, 
although there are $\ell^{\text{th}}$ order derivatives of $u_0$, they can all be expressed as $(\ell-1)^{\text{st}}$ 
derivatives of $u_0$ and $\de u_0$. In (\ref{second_ord_comp_cond}), for example, we have only first order 
derivatives of $u_0$ and $\de u_0$. With these compatibility conditions we are ready to state our theorems.
\begin{theo}
 Let $s > \frac{n}{2} + 1$ and let $u_{0k} \in H^s(\Om, \RR^n)$ and $\rho_{0k} \in H^{s+1}(\Om, \RR)$ be families of 
functions parameterized by $k$. Also let $p_k$ be a family of smooth functions satisfying the 
Assumption \ref{assumption_on_rho}. Assume that

1) For each $k$, $u_{0k},~\rho_{0k}$ satisfy compatibility conditions up to order $s$.

2) There exists a constant $a_2$ such that

\hskip 1cm a) $\parallel u_{0k} \parallel_s \leq a_2$

\hskip 1cm b) $\de u_{0k} \in H^s(\Om, \RR)$ and $\parallel \de u_{0k} \parallel_s \leq \frac{a_2}{\sqrt{k}}$

\hskip 1cm c) $\parallel f_{0k} \parallel_{s+1} \leq \frac{a_2}{k}$

where $f_{0k} = \log\rho_{0k}$.

Then there exists a number $k_0$ and a positive function $T(k)$ such that if $k > k_0$ the 
system (\ref{eq_comp_u})-(\ref{eq_comp_rho}) with initial conditions $\rho_{0k},~u_{0k}$ and with boundary condition 
(\ref{bc_comp}) has a unique $H^s$ solution $u_k(t),~\rho_k(t)$ defined on a time interval $[0,T(k))$. 
Furthermore $\rho_k(t) \in H^{s+1}$ and $\de u_k(t) \in H^s$.
\label{first_big_theo}
\end{theo}
\begin{rema}
 Most of this theorem was proven in \cite{E3}. It is only the last statement, $\rho_k \in H^{s+1}$ and $\de u_k \in H^s$ 
that is new, for in \cite{E3} we only got $\rho_k \in H^s$. We shall see the significance of the additional differentiability
in section \ref{diff_dep}.
\end{rema}
\begin{theo}
Assume that $u_{0k},~\rho_{0k}$ and $p_k$ are as above and satisfy all hypothesis of Theorem \ref{first_big_theo}. Assume 
also that there exists a $v_0 \in H^s(\Om,\RR^n)$ such that $u_{0k} \rar v_0$ in $H^s(\Om,\RR^n)$ as $k\rar \infty$. Then 
there exists an interval $[0,T]$ and a unique curve $v:[0,T] \rar H^s(\Om,\RR^n)$ 
satisfying (\ref{eq_incomp_v})-(\ref{eq_incomp_rho}) with initial condition $v_0$ and boundary 
condition (\ref{bc_incomp}). Furthermore if $u_k(t)$, $\rho_k(t)$ are the solutions of 
(\ref{eq_comp_u})-(\ref{eq_comp_rho}) from theorem \ref{first_big_theo}, then if $T(k)$ is maximal, we 
find that $T(k) > T$ for large $k$ and $u_k(t) \rar v(t)$ in $H^s(\Om,\RR^n)$ as $k \rar \infty$. 
Also $\rho_k(t) \rar 1$ in $H^{s+1}(\Om,\RR)$.
\label{second_big_theo}
\end{theo}

A consequence of this theorem is that the solution to the slightly compressible motion (in Lagrangian coordinates) depends differentiably on the initial data. To state this precisely requires additional constructions so we delay both the detailed statement and proof of this fact until sections \ref{diff_dep} and \ref{diff_dep_quasi}. We point out that Theorem \ref{second_big_theo}
can be given a simple geometric interpretation in terms of convergence
of curves in appropriate Banach manifolds (see Theorem \ref{equiv_second_big_theo}).
This fits with the general study of motions with a strong constraining force
\cite{E2}. Along these lines, recently the authors have proven similar
convergence results  in the context of the free boundary Euler equations
in two \cite{DE} and three \cite{DE2} dimensions, where the relevant parameter is the coefficient of 
surface tension.

\begin{rema}
 All of Theorem \ref{second_big_theo} except the last statement about $\rho$ was proven in \cite{E4}.  The solution $v(t)$ was constructed in \cite{EM} and has also been dealt with by other 
authors (cf. \cite{K1}, \cite{K3} or \cite{G} for example).
\end{rema}

\begin{rema}
As discussed in the introduction, the pressure of an incompressible fluid is determined
entirely by the velocity and will in general not be close to the pressure of the (more realistic) slightly compressible
fluid. However, convergence in an averaged sense does hold. 
 This becomes clear when one notices that the convergence $u_k \rar v$ and
$\rho_k \rar 1$ stated in Theorem \ref{second_big_theo} corresponds in Lagrangian coordinates to
the convergence of $(\zeta_k, \dot{\zeta}_k)$ to $(\eta, \dot{\eta})$,
where $\zeta_k$ is the flow of $u_k$ and $\eta$ the flow of $v$ (see equations (\ref{eq_eta}), (\ref{eq_zeta}),
and Theorem \ref{equiv_second_big_theo}). The convergence of the time derivatives of the flows thus ensures
(using the equations of motion) that for $t_1<t_2$
\begin{gather}
\int_{t_1}^{t_2} \nabla p_k \circ \zeta_k \rar \int_{t_1}^{t_2} \nabla p_0 \circ \eta.
\nonumber
\end{gather}
\label{remark_pressure_dist}
\end{rema}

Below we discuss the hypotheses and conclusions of Theorems \ref{first_big_theo} and \ref{second_big_theo}, in particular 
the differences with the results of \cite{E4}. A brief review of the literature is also presented.
In  section \ref{estimates} we will derive the basic estimates required for the proofs of these theorems, and the proofs 
themselves will follow in section \ref{proofs}.

\subsection{Discussion and brief review of the literature\label{review}}

Since, as mentioned, this work is a continuation of \cite{E4},
a comparison with those results seems appropriate. We start pointing out the differences in the assumptions
of Theorems \ref{first_big_theo} and \ref{second_big_theo} to the theorems in \cite{E4}.

In \cite{E4}, it was assumed that: 

(a) $u_0$ and $\rho_0$ are in $H^s$;

(b) the compatibility conditions hold up to order $s-1$;

(c) $\dive ( u_0 )$ is small in $H^{s-1}$, but nothing is said about $\dive ( u_0 )$
being in $H^s$;

(d) $f_0$ is small in $H^s$.

\noindent In the present work we have: 

(A) $u_0$ is in $H^s$ and $\rho_0$ is in $H^{s+1}$;

(B) the compatibility conditions hold up to order $s$;

(C) $\dive ( u_0 )$ is in $H^s$ and is small in $H^s$;

(D) $f_0$ is small in $H^{s+1}$. 

\noindent The reasons for the differences (a)-(A) and (d)-(D) are clear: our goal is to prove that
$\rho(t)$ has one extra derivative if $\rho_0 = \rho(0)$ does. Because we are assuming an additional derivative
on $\rho_0$, we require an extra compatibility condition, which explains the difference 
(b)-(B). 
(C) is required rather than (c) because in order to estimate $f$ in $H^{s+1},$ we need to have $\dot{f}(0) \in H^s$ and $\dot{f}(0)= \dive(u_0)$. 

Note that our assumptions for $f$ are rather standard, since $f$ satisfies a convected wave equation with inhomogeneous term, i.e., 
equation (\ref{eq_f_convected}). Since we want $f$ to be in $H^{s+1}$, the standard theory of wave equations
suggests that the initial data $(f(0),\dot{f}(0))$ should be taken in $H^{s+1} \times H^s$ which is what we are requiring.

The smallness conditions
in  (C) and (D) are needed because we need to know that the initial data for $f$ are small in order to guarantee that (\ref{eq_f_convected}) has a small solution. With this and the assumption that $u_{0k}$ is close to $v_0,$ we show that the  solutions to
(\ref{eq_comp}) are close to those of (\ref{eq_incomp}). 

In the introduction we mentioned that the differentiability of the data-to-solution map in Lagrangian coordinates
(which is presented in sections \ref{diff_dep} and \ref{diff_dep_quasi}) is unexpected in light 
of known results on the theory of quasi-linear hyperbolic differential equations. (In fact, the 
result is false in Eulerian coordinates).
At first sight, one might think that the result would not be that unexpected since the system in 
Lagrangian coordinates can be loosely thought as an ordinary differential equation. However, note that in Lagrangian coordinates the equation for the flow can be written
$$\ddot{\zeta} = (c^2 \nabla f) \circ \zeta$$
which is a second order ODE on $\mathscr{D}^s$
provided that $\nabla f$ is in $H^s.$ But this requires $f$ (and thus $\rho)$ to be in $H^{s+1}.$

In order to obtain the additional smoothness of $f$, we make use of the extra regularity
of the density function obtained in Theorem \ref{first_big_theo}. We notice that this extra regularity
holds in Eulerian coordinates, and it probably fails in Lagrangian coordinates; indeed,  
to say that this extra regularity holds in Lagrangian coordinates means that  $\rho \circ \zeta$ is in $H^{s+1}$,
but this is unlikely since $\zeta$ is only in $H^s$. Summarizing, we obtain first a new result in Eulerian
coordinates, namely, the extra smoothness of $\rho$ (which is probably false in Lagrangian coordinates). We then
use this new result in Eulerian coordinates to establish a second new result, but now in Lagrangian coordinates,
namely, the $C^1$ dependence on the initial data (and this second new result is false in Eulerian coordinates).
This illustrates the interplay between Lagrangian and Eulerian coordinates that is one of the key elements of the paper.

Thus our results explore good properties of both Eulerian and Lagrangian coordinates, getting additional smoothness in the former and differentiable dependence on initial data in the latter.

Regarding the methods employed in the proofs in the present paper, in comparison with \cite{E4}, 
they fall basically into two types: (i) Those that are entirely new, i.e.,
addressing the results announced but not proven in \cite{E4}; and (ii) those that are similar to \cite{E4}. 
Naturally, (i) is the main interest of our paper, but we note that because of the stronger results 
proven here (namely the extra regularity of the density
and the differentiability of the data-to-solution map), even those parts based on \cite{E4} had to be re-worked.
In particular, the estimates of section \ref{estimates} are more delicate than those for the corresponding function $f$
in \cite{E4}, even if reminiscent of \cite{E4}. 

We now make some brief comments on previous works that have dealt with the incompressible limit
of the Euler equations. There have been many works on this subject and a complete review of the literature
is quite impractical. We refer the reader to the monograph \cite{Ma} for a longer discussion of the problem, 
and here mention only a few of the extensive list of works on the topic.

Some works that we are aware of that treat the incompressible limit of the
Euler equations in a domain with boundary in $\RR^n$  are  \cite{AlazardIncompWall, ChengLowMach, ChengImprovedInc, S}.
When the equations in the whole of $\RR^n$ 
or $\TT^n$ are considered,
the incompressible limit has also been investigated in 
\cite{KM1, KM2, MS1} (the results in \cite{KM1, KM2} are reproved in a simpler fashion in \cite{Ma}; our second author
also treats the $\TT^n$  case in \cite{E2}).
From the point of view of the convergence of solutions, the main difference between these and our results (with or without 
boundary) is that we obtain the convergence in the same space where solutions live, whereas in 
\cite{AlazardIncompWall, ChengLowMach, ChengImprovedInc, KM1, KM2, Ma, MS1, S} solutions are obtained in $H^s$
 but convergence is established only in $H^{s-\de}$, for some 
$\de > 0$. We believe that the main reason for this difference is our use of Lagrangian coordinates, while 
the aforementioned authors work mostly in Eulerian coordinates.
We avoid a full discussion of all similarities and differences between these works and ours, but point out that 
some of these authors employ a more general equation of state than we do, allowing the inclusion of entropy in the system.
We also mention the related works \cite{MS2, UkaiIncompressible}. A somewhat intermediate situation between
the study of the equations in $\RR^n$ and in a bounded domain can be found in \cite{Secchi}, where the author considers
the half-plane in two dimensions and carries out a detailed analysis of the time of existence of solutions. 
In fact, in two dimensions the limiting system (the incompressible system) is globally well-posed (in time). Thus,
it is natural to ask whether the compressible system, being close to the incompressible one, is well-posed for large time.
This is the question addressed in \cite{Secchi}.

\section{Estimates for $f$ and $\dot{f}$ \label{estimates}}

Now we shall consider $f$, the solution of (\ref{eq_f_convected}), and find estimates for $\parallel f (t)\parallel_{s+1}$ and
$\parallel \dot{f}(t)\parallel_s$. These estimates will be similar to the
 estimates of \cite{E2} \S 12, but the method
of estimation is simpler because it avoids operator theoretic techniques. Also we will get more information about the 
density function. For example our estimate of $\parallel f \parallel_{s+1}$ gives an 
estimate of the $H^{s+1}$-norm of $\log\rho$ and hence of $\rho$. With this we will show that $\rho$ is actually 
smoother than the velocity $u$.

For simplicity of exposition we shall restrict ourselves to the case $s=3$ and $n=2$ or $3$. The method 
can be used for arbitrary $s$ and $n$, (provided $s>\frac{n}{2}+1$), but the expressions are more cumbersome. Also, the above 
case seems to be the one of primary importance.

We shall assume that we are given a family of functions $\{ \rho_k \}$ which satisfies Assumption \ref{assumption_on_rho} and 
that for each $k$ we have $u_k$ and $\rho_k$ which satisfy (\ref{eq_comp_u})-(\ref{eq_comp_rho}) and the boundary 
condition (\ref{bc_comp}) on some interval $[0,T]$. We further assume that $f_k = \log \rho_k \in H^4(\Om)$, and that
\begin{gather}
\parallel u_{k}(0) \parallel_3 \leq a_2,~
\parallel \de u_{k}(0) \parallel_3 \leq \frac{a_2}{\sqrt{k}} \text{ and }
\parallel f_{k}(0) \parallel_{4} \leq \frac{a_2}{k}
\label{hypothesis_at_0}
\end{gather}
as is assumed in theorem \ref{first_big_theo}.

Also we assume that there exists a constant $a_3$  such that
\begin{assump}
\begin{gather}
 \parallel u_k(t) \parallel_3 \leq a_3\,\, \text{ for all } \,\, 
  t \in [0,T], \label{extra_assump_u} 
\end{gather}
\label{extra_assumption_u_statement}
\end{assump}
\noindent and another  constant and $a_4$ such that
\begin{assump} 
 \begin{subequations}
 \begin{align}
 & \parallel f_k(t) \parallel_4 \leq \frac{a_4}{\sqrt{k}}, \label{extra_assump_f} \\
& \parallel \dot{f}_k(t) \parallel_3 \leq a_4, \label{extra_assump_f_dot}\\
& \parallel \ddot{f}_k(t) \parallel_2 \leq a_4\sqrt{k}, \label{extra_assump_f_ddot}\\
& \parallel \dddot{f}_k(t) \parallel_1 \leq a_4 k, \label{extra_assump_f_dddot}
\end{align}
 \end{subequations}
for $t \in [0,T]$.
\label{assump_on_f}
\end{assump}
As we shall see, this latter assumption can eventually be discarded. With these assumptions we will derive estimates which hold
for sufficiently large $k$. We will in fact find a constant $a_5$ such that for large $k$
\begin{subequations}
 \begin{align}
&  \parallel f_k(t) \parallel_4 \leq \frac{a_5}{k}, \label{ineq_a_5_f} \\
& \parallel \dot{f}_k(t) \parallel_3 \leq \frac{a_5}{\sqrt{k}}, \label{ineq_a_5_f_dot} \\
& \parallel \ddot{f}_k(t) \parallel_2 \leq a_5 \label{ineq_a_5_f_ddot}, \\
& \parallel \dddot{f}_k(t) \parallel_1 \leq a_5  \sqrt{k} \label{ineq_a_5_f_dddot},
 \end{align}
\end{subequations}
for $t \in [0,T]$.

Our method will be to derive a type of energy estimate for $f$, and its first material derivatives.
Let
\begin{gather}
L= -\de c^2 \nabla,
\nonumber
\end{gather}
and 
\begin{gather}
F = u^i_j u^j_i ,
\nonumber 
\end{gather}
 so that (\ref{eq_f_convected}) be written as
\begin{gather}
 \ddot{f} = Lf + F .
\label{eq_f_estimates}
\end{gather}
Then let 
\begin{gather}
 E(t) = \int_\Om \Big (  |c\nabla \dddot{f}(t)|^2 + |L \ddot{f}(t)|^2 \Big ).
\label{def_energy}
\end{gather}
We consider an energy $E(t)$ with  three material derivatives since we are doing the analysis for the case $s$ equal to 3.     
$E(t)$ will give us estimates for $\parallel \dddot{f}(t) \parallel_1$ and $\parallel \ddot{f}(t) \parallel_2$, and using
these and (\ref{eq_f_estimates}) we can get estimates for $f$ and $\dot{f}$.

We now begin the estimate of $E(t)$.
We shall use 
\begin{gather}
 E(t) = E(0) + \int_0^t \frac{dE(s)}{ds} \, ds,
\label{E_fund_thm_calc}
\end{gather}
but in order to do so we must compute $\frac{dE}{dt}$. Differentiating $E$ we find
\begin{gather}
 \frac{1}{2}\frac{dE}{dt} = \int_\Om \Big ( \langle \partial_t(c\nabla \dddot{f}),c\nabla \dddot{f} \rangle 
+ (\partial_t L \ddot{f} ) (L \ddot{f}) \Big ),
\nonumber
\end{gather}
but 
\begin{gather}
 \partial_t(c\nabla \dddot{f} ) = (\partial_t c) \nabla \dddot{f} + c \nabla \partial_t \dddot{f} 
= (\partial_t c) \nabla \dddot{f} - c \nabla \nabla_u  \dddot{f} +  c \nabla \ddddot{f}
\nonumber
\end{gather}
and similarly
\begin{gather}
 \partial_t L\ddot{f} = -\de\partial_t(c^2) \nabla \ddot{f} - L \nabla_u \ddot{f} + L \dddot{f} .
\nonumber
\end{gather}
Therefore
\begin{gather}
 \frac{1}{2} \frac{dE}{dt} = 
\int_\Om \Big ( \langle c\nabla \dddot{f},c\nabla \ddddot{f} \rangle 
+ (L \dddot{f} ) (L \ddot{f}) \Big ) + R_1
\nonumber
\end{gather}
where 
\begin{gather}
 R_1 = \int_\Om \Big ( \langle (\partial_t c) \nabla \dddot{f} - c \nabla \nabla_u  \dddot{f}, c\nabla \dddot{f} \rangle
-(\de\partial_tc^2\nabla \ddot{f})L\ddot{f} - (L\nabla_u \ddot{f})L\ddot{f} \Big ).
\label{definition_R_1}
\end{gather}
Integrating by parts we find that
\begin{gather}
 \frac{1}{2} \frac{dE}{dt} = 
\int_\Om  \langle c\nabla \dddot{f},c\nabla \ddddot{f} - c \nabla L \ddot{f} \rangle 
+ \int_{\partial \Om} (c^2 \nabla_\nu \dddot{f})L\ddot{f} + R_1 \nonumber \\
= I_1 + B_1 + R_1 .
\nonumber
\end{gather}
We proceed to evaluate $I_1$. First we take two material derivatives of equation (\ref{eq_f_estimates}). Letting $L_1$ be the operator
defined by
\begin{gather}
 L_1 h = (Lh)\dot{} - L\dot{h}
\label{definition_L_1}
\end{gather}
and computing directly we find
\begin{gather}
 (Lf)\ddot{} = L\ddot{f} + L_1\dot{f} + (L_1 f)\dot{}
\nonumber
\end{gather}
so we get
\begin{gather}
 \ddddot{f} = L\ddot{f} + L_1\dot{f} + (L_1 f)\dot{} + \ddot{F} .
\label{f_4_dots_L_1}
\end{gather}
Therefore
\begin{gather}
I_1 = \int_\Om \langle c\nabla \dddot{f}, c\nabla(L_1\dot{f} + (L_1 f)\dot{} + \ddot{F} )\rangle .
\nonumber
\end{gather}
We now calculate $L_1$. Using (\ref{definition_L_1}) we find
\begin{gather}
 L_1 = -\de \partial_t(c^2) \nabla + [\nabla_u,L] = -\de (c^2)\dot{}\,\nabla - [\nabla_u, \de] c^2 \nabla -\de c^2 [\nabla_u,\nabla]
\label{computing_L_1}
\end{gather}
where $[~,~]$ denotes the commutator.

Of course the commutators on the 
rightmost side of (\ref{computing_L_1}) are first order operators whose coefficients are first 
derivatives of $u$. Therefore $L_1$ is a second order operator whose coefficients depend on (derivatives of) $c$ and $u$.

In order to make estimates with such operators we will need several formulas and also a convenient way to bound many 
different expressions. For the latter purpose we introduce a generic constant $K$. It will have different values in different 
expressions but in each case it will depend only on $\Om$ and $\{a_i | i=1,\dots,4\}$.

To begin our formulas we note that from (\ref{eq_comp_u}) and the definition of $c^2$
\begin{gather}
 \dot{u} = \partial_t u + \nabla_u u = -c^2 \nabla f .
\nonumber
\end{gather}
Applying $\partial_j$ to this we get
\begin{gather}
 (u^i_j)\dot{} = -(c^2f_i)_j - u^k_ju^i_k .
\label{u_i_j_dot}
\end{gather}
For any function $h$ a direct computation gives
\begin{gather} 
 (\partial_i h)\dot{} = \partial_i (\dot{h}) - u^j_i\partial_j h .
\label{direct_comp_any_h_der}
\end{gather}
Applying (\ref{direct_comp_any_h_der}) to $f$ we find
\begin{gather}
 \ddot{u}^i = -(c^2)\dot{} \, \partial_i f - c^2 \partial_i (\dot{f}) + c^2 u^j_i f_j,
\nonumber
\end{gather}
and 
\begin{gather}
 \dddot{u}^i = - (c^2)\ddot{}\,\partial_i f -2(c^2)\dot{} \, (\partial_i(\dot{f}) - u^j_i f_j )
- c^2 \big ( \partial_i (\ddot{f}) -2u^j_i(\dot{f})_j - \big ( (c^2f_i)_j + 2u_i^ku_k^j)f_j \big) .
\label{u_i_3_dots}
\end{gather}
Thus we have expressions for material derivatives of $u$. We shall need formulas for material derivatives of $c^2$ as well. 
From the definition $c^2 = p^\prime(\rho)$ and the equation $\dot{f} = \frac{\dot{\rho}}{\rho}$ we find
\begin{gather}
 (c^2)\dot{} = p^{\prime\prime}(\rho) \dot{\rho} = p^{\prime\prime}(\rho)\rho \dot{f}.
\label{c2dot}
\end{gather}
Thus
\begin{gather}
 (c^2)\ddot{} = p^{\prime\prime\prime}(\rho)(\rho \dot{f})^2 + p^{\prime\prime}(\rho)(\rho \ddot{f} + \rho^2\dot{f}^2) ,
\label{c2ddot}
\end{gather}
and
\begin{gather}
 (c^2)\ddot{}\,\,\dot{} = p^{\prime\prime\prime\prime}(\rho)(\rho \dot{f})^3 +
 3 p^{\prime\prime\prime}(\rho) (\rho \dot{f})(\rho\ddot{f} + \rho^2\dot{f}^2) +
p^{\prime\prime}(\rho)(\rho \dddot{f} + \rho \dot{f} \ddot{f} + 2 \rho^2 \dot{f}\ddot{f} + 2 \rho^3 \dot{f}^3).
\label{c2dddot}
\end{gather}
Since $p^\prime(1) = k$, $c^2 = k + \int_1^\rho p^{\prime\prime}(\la)d\la$. Also, since $\parallel f \parallel_4 \leq
\frac{a_4}{\sqrt{k}}$, $\rho = e^f$, and $|p^{\prime\prime}|\leq a_1k$ we find
\begin{gather}
 \parallel c^2 - k \parallel_4 \leq K\sqrt{k} .
\label{estimate_c2_k}
\end{gather}
Then using the formulas (\ref{c2dot})-(\ref{c2dddot}) we get the estimates
\begin{gather}
 \parallel (c^2)\dot{} \parallel_3 \leq Kk, \label{estimate_c2_dot} \\
\parallel (c^2)\ddot{} \parallel_2 \leq Kk(\sqrt{k} + \parallel \ddot{f} \parallel_2 ) \leq K k^\frac{3}{2}, \label{estimate_c2_ddot}  \\
\parallel (c^2)\ddot{}\,\,\dot{} \parallel_1 \leq Kk(\sqrt{k} + \parallel \dddot{f} \parallel_1 ) \leq K k^2. \label{estimate_c2_dddot} 
\end{gather}
Using these estimates and the formulas for derivatives of $u$ we get
\begin{gather}
 \parallel \dot{u} \parallel_3 \leq K k \parallel f \parallel_4 \leq K \sqrt{k}, \label{estimate_u_dot} \\
\parallel \ddot{u} \parallel_2 \leq K (\sqrt{k} + k\parallel \dot{f} \parallel_3 ) \leq K k, \label{estimate_u_ddot}  \\
\parallel \dddot{u} \parallel_1 \leq K (k + k\parallel \ddot{f} \parallel_2 
+ k^\frac{3}{2} \parallel f \parallel_4 ) \leq K k^\frac{3}{2}. \label{estimate_u_dddot} 
\end{gather}
Using (\ref{computing_L_1}) with (\ref{product_estimate}) and then with  (\ref{extra_assump_u}), 
(\ref{estimate_c2_k}), (\ref{estimate_c2_dot}) we get
\begin{gather}
 \parallel L_1 \dot{f} \parallel_1 
\leq K (\parallel u \parallel_3 \parallel c^2 \parallel_3 + \parallel (c^2)\dot{} \parallel_2 ) \parallel \dot{f} \parallel_3 
\leq K k \parallel \dot{f} \parallel_3 .
\label{first_step_I_1}
\end{gather}
This is the first step in the estimate of $I_1$. In order to estimate $(L_1 f)\dot{}$ we will need the following additional 
formula:
\begin{gather}
 ( [\nabla_u, \partial_j] h )\dot{} = ( (c^2 f_i)_j + u_j^k u_k^i) h_i - u_j^i (h_i)\dot{} 
= ( (c^2 f_i)_j + 2u_j^k u_k^i) h_i - u_j^i (\dot{h})_i
\label{additional_formula}
\end{gather}
which follows from  (\ref{u_i_j_dot}) and (\ref{direct_comp_any_h_der}). From
(\ref{computing_L_1}) we have
\begin{gather}
 L_1 f = -\de (c^2)\dot{}\, \nabla f - [\nabla_u,\de] c^2 \nabla f - \de c^2 [\nabla_u, \nabla] f .
\nonumber
\end{gather}
Using the above formulas we get
\begin{gather}
(\de(c^2)\dot{}\,\nabla f)\dot{} = \de(c^2)\ddot{} \, \nabla f + 
u_i^j((c^2)\dot{}f_i)_j + \de(c^2)\dot{}(\nabla\dot{f} - (\nabla u^j)f_j),
\label{big_various_1}
\end{gather}
and
\begin{gather}
 ([\nabla_u,\de]c^2\nabla f)\dot{} = ((c^2 f_i)_j + 2u^k_ju_k^i)(c^2f_j)_i - u^i_j((c^2 f_j)\dot{})_i .
\label{big_various_2}
\end{gather}
But
\begin{gather} 
 ((c^2f_j)\dot{})_i = ((c^2)\dot{} f_j + c^2 ((\dot{f})_j-u^k_jf_k))_i .
\label{big_various_3}
\end{gather}
Therefore
\begin{gather}
 ([\nabla_u,\de]c^2\nabla f)\dot{} = ((c^2f_i)_j + 2u_j^ku_k^i)(c^2f_j)_i - u_j^i((c^2)\dot{}f_j + c^2((\dot{f})_j - u^k_j f_k))_i .
\label{big_various_4}
\end{gather}
Also
\begin{gather}
 (\de c^2 [\nabla_u,\nabla]f)\dot{} = \de (c^2)\dot{}[\nabla_u,\nabla] f - 
(c^2(((c^2f_i)_j + 2u_j^ku_k^i)f_i - u_j^i(\dot{f})_i))_j - u_i^j(c^2u_i^kf_k)_j .
\label{big_various_5}
\end{gather}
From (\ref{big_various_1}), (\ref{big_various_4}) and (\ref{big_various_5}) we find
\begin{align}
\begin{split}
 \parallel (L_1 f)\dot{} \parallel_1 \leq & K \Big ( \parallel (c^2)\ddot{}\parallel_2 
  +
 \parallel (c^2)\dot{}\parallel_2 \parallel u \parallel_3 
 \\
 & 
 +  \parallel c^2\parallel_3( \parallel u \parallel_3 + 
\parallel u \parallel_3^2 +  \parallel c^2\parallel_3\parallel f \parallel_3) \Big)\parallel f 
\parallel_4 
\\
& 
+ K(\parallel (c^2)\dot{}\parallel_2 + \parallel c^2\parallel_2\parallel u \parallel_3 )\parallel \dot{f} \parallel_3 
\\
\leq & K ( k^\frac{3}{2} \parallel f \parallel_4 + k \parallel \dot{f} \parallel_3 ) .
\end{split}
\label{second_step_I_1}
\end{align}
To complete the estimate of $I_1$, we now estimate the norm of 
\begin{gather}
\ddot{F} = (u_j^i u^j_i)\,\ddot{} = -\big ( ( (c^2f_i)_j + u^k_j u^i_k)u^j_i 
+ ( (c^2f_j)_i
+  u_i^k u^j_k)u^i_j \big )\dot{} \,\,.
\nonumber
\end{gather}
Note that
\begin{gather}
 ( (c^2 f_i)_j )\dot{} = ( (c^2 f_i)\dot{})_j - u^k_j(c^2 f_i)_k,
\nonumber
\end{gather}
which by (\ref{direct_comp_any_h_der}) equals
\begin{gather}
( (c^2)\dot{}f_i + c^2( (\dot{f})_i - u_i^k f_k ) )_j - u^k_j (c^2f_i)_k .
\label{big_various_6}
\end{gather}
Using (\ref{big_various_6}) and (\ref{u_i_j_dot}) we find that
\begin{align}
\begin{split}
 \parallel \ddot{F} \parallel_1 \leq 
 & K \Big (
\parallel u \parallel_3^4 + \parallel u \parallel_3^2\parallel c^2 \parallel_3 \parallel f \parallel_3 + 
\parallel u \parallel_3 \parallel (c^2)\dot{} \parallel_2 \parallel f \parallel_3 
\\
& + 
\parallel u \parallel_3 \parallel c^2 \parallel_3 \parallel \dot{f} \parallel_3 \Big ) 
\\
\leq & K k( \parallel f \parallel_3 + \parallel \dot{f} \parallel_3 ) .
\end{split}
\label{estimate_on_F_ddot} 
\end{align}
Combining
(\ref{first_step_I_1}), (\ref{second_step_I_1}) and (\ref{estimate_on_F_ddot}) we find that
\begin{gather}
 |I_1| \leq K k \parallel \dddot{f} \parallel_1 ( k \parallel \dot{f} \parallel_3 + k^\frac{3}{2} \parallel f \parallel_4 + 1) .
\label{estimate_on_I_1}
\end{gather}
We now proceed to estimate $R_1$. By commuting $\nabla$ and $\nabla_u$ and integrating by parts we find
\begin{align}
\begin{split}
 \int_\Om \langle c \nabla \nabla_u \dddot{f}, c \nabla \dddot{f} \rangle =
 \int_\Om \langle c [\nabla,\nabla_u] \dddot{f}, c\nabla \dddot{f} \rangle 
 \\
 - \int_\Om \langle (\nabla_u c) \nabla \dddot{f}, c\nabla \dddot{f} \rangle + \frac{1}{2} \int_\Om \de(u) \langle
 c\nabla \dddot{f},c\nabla \dddot{f} \rangle .
 \end{split}
\label{computations_R_1_1}
\end{align}
Similarly
\begin{gather}
 \int_\Om (L \nabla_u \ddot{f}) L\ddot{f} = \int_\Om \Big (
-(\de(\nabla_u (c^2))\nabla \ddot{f} ) L\ddot{f} + \text{com} + \frac{1}{2} \de(u) (L\ddot{f})^2 \Big ),
\label{computations_R_1_2}
\end{gather}
where ``$\text{com}$" indicates second order terms in $\ddot{f}$ involving the commutators
$[\nabla,\nabla_u]$ and $[\de,\nabla_u]$. 

Plugging (\ref{computations_R_1_1}) and (\ref{computations_R_1_2}) in the expression (\ref{definition_R_1}) 
for $R_1$ we get
\begin{gather}
 R_1 = \int_\Om \langle (\partial_t c) \nabla \dddot{f},c\nabla \dddot{f} \rangle - 
\int_\Om \langle c[\nabla,\nabla_u]\dddot{f},c\nabla \dddot{f}\rangle + 
\int_\Om \langle (\nabla_u c) \nabla \dddot{f},c\nabla\dddot{f} \rangle
\nonumber
\\
 -\frac{1}{2} \int_\Om \de(u) \langle c \nabla \dddot{f},c\nabla \dddot{f} \rangle  \nonumber 
 -\int_\Om (\de(\partial_tc^2))\nabla \ddot{f} )L\ddot{f} - \int_\Om\big(-(\de(\nabla_u (c^2))\nabla \ddot{f} ) \big )L\ddot{f}
 \nonumber \\
  - \frac{1}{2} \int_\Om \de(u) (L\ddot{f})^2 -\text{com} .
\nonumber
\end{gather}
The first and third integrals combined yield $\int_\Om \langle \dot{c} \nabla \dddot{f},c\nabla \dddot{f} \rangle $. The 
fifth and sixth integrals can be rewritten as
\begin{gather}
-\int_\Om (\de(\partial_tc^2))\nabla \ddot{f} )L\ddot{f} - \int_\Om\big(-(\de(\nabla_u (c^2))\nabla \ddot{f} ) \big )L\ddot{f} 
\nonumber \\
= -\int_\Om (\de(\partial_tc^2))\nabla \ddot{f} )L\ddot{f} + \int_\Om (\de(\nabla_u (c^2))\nabla \ddot{f} )L\ddot{f} 
- \int_\Om (\de(\nabla_u (c^2))\nabla \ddot{f} )L\ddot{f} + \int_\Om (\de(\nabla_u (c^2))\nabla \ddot{f} )L\ddot{f}
\nonumber \\
 = - \int_\Om (\de(c^2)\dot{}\,\nabla \ddot{f} )L\ddot{f} + 2\int_\Om (\de(\nabla_u (c^2))\nabla \ddot{f} )L\ddot{f} .
\nonumber
\end{gather}
Therefore
\begin{gather}
 R_1 = \int_\Om \Big ( \langle \dot{c} \nabla \dddot{f}, c\nabla \dddot{f} \rangle - (\de(c^2)\dot{}\,\nabla \ddot{f})L\ddot{f}
- \frac{1}{2} \de(u) \big ( \langle c\nabla \dddot{f}, c\nabla\dddot{f} \rangle + (L\ddot{f})^2 \big ) \nonumber \\
 -\frac{1}{2}\langle c[\nabla,\nabla_u]\dddot{f}, c\nabla \dddot{f} \rangle - 
\text{com} + 2(\de(\nabla_u (c^2))\nabla \ddot{f} )L\ddot{f} \Big ) .
\nonumber
\end{gather}
Hence
\begin{align}
\begin{split}
 |R_1| \leq & K \Big ( \parallel (c^2)\dot{} \parallel_2\parallel \dddot{f} \parallel^2_1 + 
\parallel (c^2)\dot{}\parallel_2\parallel c^2\parallel_2\parallel \ddot{f} \parallel_2^2 
\\
& + 
\parallel u \parallel_3(\parallel c^2 \parallel_2\parallel \dddot{f} \parallel_1^2 + \parallel c^2 \parallel_2^2 \parallel \ddot{f}
 \parallel_2^2 ) \Big ) 
 \\
\leq&  K( k \parallel \dddot{f} \parallel_1^2 + k^2 \parallel \ddot{f} \parallel_2^2 ) .
\end{split}
\label{estimate_on_R_1}
\end{align}
To estimate $B_1$ we must use the boundary condition (\ref{bc_comp}), and compute its material derivatives. In fact we 
will not get a good estimate for $B_1$ itself, but we will for $\int_0^t B_1$. Since we will use (\ref{E_fund_thm_calc}) to estimate
$E$, this will be sufficient.

Applying (\ref{f_4_dots_L_1}) we find that
\begin{gather}
 B_1 = \int_{\partial\Om} (c^2\nabla_\nu \dddot{f} )(\ddddot{f} - L_1 \dot{f} - (L_1 f)\dot{} - \ddot{F}) \nonumber \\
=  \int_{\partial\Om} (c^2\nabla_\nu \dddot{f})\ddddot{f} - \int_{\partial \Om} (c^2 \nabla_\nu \dddot{f})(L_1\dot{f}
- (L_1 f)\dot{} - \ddot{F} ) = B_{11} + B_{1R} .
\nonumber
\end{gather}
To estimate these terms we must compute $c^2 \nabla_\nu \dddot{f}$, and we do this by using the boundary condition for $f$.

Computing as above we find
\begin{gather}
 \nabla_\nu \dddot{f} = (\nabla_\nu f)\,\ddot{}\,\,\dot{} + (\nabla_{[\nu,u]}f)\ddot{} + (\nabla_{[\nu,u]}\dot{f})\dot{}
+ \nabla_{[\nu,u]}\ddot{f} .
\label{big_normal_3_dots_f}
\end{gather}
But by (\ref{normal_der_f})
\begin{gather}
(\nabla_\nu f)\, \ddot{}\,\,\dot{} = (-\frac{1}{c^2}\langle \nu, \nabla_u u \rangle)\,\ddot{}\,\,\dot{}  .
\label{3_dots_of_normal_f}
\end{gather}
Since $\langle \nu, u \rangle = 0$ on $\partial \Om$,
\begin{gather}
 -\langle \nu, \nabla_u u \rangle = S_2(u,u),
\nonumber
\end{gather}
where $S_2$ is the second fundamental form of $\partial \Om$. We shall sometimes write simply $S_2(u)$. Taking a material 
derivative we find
\begin{gather}
 (S_2(u))\dot{} = 2 S_2(\dot{u},u) + S_3(u),
\label{S_2_dot_terms_S_2}
\end{gather}
where $S_3$ is a symmetric tri-linear form.

Continuing we find
\begin{gather}
 (S_2(u))\,\ddot{} = 2 S_2(\ddot{u},u) + 2S_2(\dot{u},\dot{u}) + 5 S_3(\dot{u},u,u) + S_4(u)
\nonumber
\end{gather}
where $S_4$ is a symmetric quadri-linear form. 

Finally
\begin{gather}
 (S_2(u))\,\ddot{}\,\,\dot{} = 2 S_2(\dddot{u},u) + S_R(u,\dot{u},\ddot{u})
\nonumber
\end{gather}
where $S_R$ is a polynomial in $u$, $\dot{u}$, and $\ddot{u}$ of degree five, 
and each term of $S_R$ contains at most 
three dot's. That is, in a given term of $S_R$, $\dot{u}$ and $\ddot{u}$ appear with power $i$ and $j$ where 
$i+2j \leq 3$.

From (\ref{3_dots_of_normal_f}) we get a formula for $(\nabla_\nu f) \ddot{}\,\,\dot{} $ involving
$\frac{1}{c^2}$ and $S_2(u)$ and the first three material derivatives of each of them.
Letting $q = \frac{1}{c^2}$ we write this as
\begin{gather}
 (\nabla_\nu f)\, \ddot{}\,\,\dot{} = 2qS_2(\dddot{u},u) + \dddot{q}S_2(u) + f_{R\partial}
\label{normal_f_3_dots_q}
\end{gather}
where $f_{R\partial}$ is a polynomial in $u$, $\dot{u}$, $\ddot{u}$, $q$, $\dot{q}$, and $\ddot{q}$ each of whose terms 
contains at most three dot's.

In order to estimate this, we must compute material derivatives of $q$, and we do this using
(\ref{estimate_c2_k})-(\ref{estimate_c2_dddot}). Direct computation yields
\begin{gather}
 \dot{q} = -q^2 (c^2)\dot{} = -q^2 p^{\prime\prime}(\rho)\rho\dot{f}, \label{q_dot} \\
\ddot{q} = -q^2 (c^2)\ddot{} -(q^2)\dot{} (c^2)\dot{} = -q^2 (c^2)\ddot{} + 2 q^3 ((c^2)\dot{}\,)^2, \label{q_ddot}  \\
\dddot{q} = -q^2 (c^2)\ddot{}\,\,\dot{}+ 6 q^3 (c^2)\dot{} (c^2)\ddot{} - 6 q^4 ((c^2)\dot{})^3\,\,. \label{q_dddot} 
\end{gather}
From (\ref{estimate_c2_k}) we find that $\parallel \frac{c^2}{k} - 1 \parallel_4 \leq \frac{K}{\sqrt{k}}$ and from this it 
follows that $\parallel q \parallel_4 \leq \frac{K}{k}$. Then from(\ref{estimate_c2_dot})-(\ref{estimate_c2_dddot})
and (\ref{q_dot})-(\ref{q_dddot}) we get
\begin{gather}
 \parallel \dot{q} \parallel_3 \leq \frac{K}{k}, \label{various_estimates_q_dot} \\
\parallel \ddot{q} \parallel_2 \leq \frac{K}{\sqrt{k}}, \label{various_estimates_q_ddot} \\
\parallel \dddot{q} \parallel_1 \leq K.
\label{various_estimates_q_dddot}
\end{gather}
From these formulas and the estimates (\ref{estimate_u_dot})-(\ref{estimate_u_dddot}) for $u$ we get
\begin{gather}
 \parallel (\nabla_\nu f)\,\ddot{}\,\,\dot{} \parallel_{\partial,\frac{1}{2}} \leq K(1+\parallel \ddot{f} \parallel_2
+ \sqrt{k} \parallel f \parallel_4 ) .
\label{boundary_f_dddot}
\end{gather}
Now we estimate the other terms of (\ref{big_normal_3_dots_f}). Clearly
\begin{gather}
 \parallel \nabla_{[\nu,u]} \ddot{f} \parallel_{\partial,\frac{1}{2}} \leq K \parallel \ddot{f} \parallel_2
\nonumber
\end{gather}
and using (\ref{additional_formula}) we get
\begin{gather}
 \parallel ( \nabla_{[\nu,u]} \dot{f})\dot{} \parallel_{\partial,\frac{1}{2}} \leq K(1 + k\parallel f \parallel_4 +
\parallel \ddot{f} \parallel_2 ) .
\label{boundary_f_ddot}
\end{gather}
Finally we estimate
\begin{gather}
 (\nabla_{[\nu,u]} f)\ddot{} = (\nabla_{[\nu,u]\,\dot{}} f + \nabla_{[\nu,u]} \dot{f} 
 - [u,\nu]^i u_i^jf_j )\dot{} \,\,.
\label{com_ddot}
\end{gather}
By taking a material derivative of (\ref{u_i_j_dot}), using (\ref{estimate_c2_dot}) and 
(\ref{estimate_u_dot})-(\ref{estimate_u_ddot}) we find
\begin{gather}
 \parallel \nabla_{[u,\nu]\,\ddot{}} f \parallel_{\partial,\frac{1}{2}} \leq K k \parallel f \parallel_4 .
\nonumber
\end{gather}
The other terms of (\ref{com_ddot}) are easily estimated and we get
\begin{gather}
\parallel (\nabla_{[u,\nu]}f)\ddot{} \parallel_{\partial,\frac{1}{2}} \leq K(k\parallel f \parallel_4 
+ \parallel \dot{f} \parallel_3 + \parallel \ddot{f} \parallel_2 ) .
\label{boundary_f_com}
\end{gather}
Combining (\ref{boundary_f_dddot})-(\ref{boundary_f_ddot}) and (\ref{boundary_f_com}) we get
\begin{gather}
 \parallel \nabla_\nu \dddot{f} \parallel_{\partial,\frac{1}{2}} \leq K (1 + \parallel \ddot{f} \parallel_2 +
\parallel \dot{f} \parallel_3 + k \parallel f \parallel_4 ) .
\label{estimate_boundary_f_3_dots}
\end{gather}
Therefore
\begin{gather}
 \n c^2 \nabla_\nu \dddot{f} \n_{\partial,\frac{1}{2}} \leq K ( k\parallel \ddot{f} \parallel_2 
+ k \parallel \dot{f} \parallel_3 +k^2 \parallel f \parallel_4 ) ,
\nonumber
\end{gather}
so from (\ref{first_step_I_1}), (\ref{second_step_I_1})  and (\ref{estimate_on_F_ddot})
we get
\begin{gather}
 |B_{1R}| \leq K(k \parallel \dot{f} \parallel_3 + k^\frac{3}{2} \parallel f \parallel_4 )( k \parallel \ddot{f} \parallel_2 +
k \parallel \dot{f} \parallel_3 + k^2 \parallel f \parallel_4 ) .
\label{estimate_on_B_1R}
\end{gather}
It remains to estimate $\int_0^tB_{11}$, and to do this we integrate by parts with respect to the material derivative, obtaining
\begin{gather}
 \int_0^t B_{11} = \int_{\partial \Om} (c^2 \nabla_\nu \dddot{f})\dddot{f} \Big |_0^t
  - \int_0^t \int_{\partial \Om}
(c^2 \nabla_\nu \dddot{f})\dot{} \,\dddot{f} 
+ \int_0^t \int_{\partial \Om} \de_\partial(u) (c^2\nabla_\nu \dddot{f})\dddot{f},
\label{integral_B_11}
\end{gather}
where $\de_\partial$ is the formal adjoint of the gradient on the manifold $\partial \Om$.
The last term of this expression is clearly bounded by
\begin{gather}
 K \int_0^t \parallel u \parallel_3\parallel c^2 \nabla_\nu \dddot{f} \parallel_{\partial,\frac{1}{2}} \parallel \dddot{f} \parallel_1 .
\label{estimate_error_integral_B_11}
\end{gather}
To compute the next to last term, we take the material derivative of $c^2 \nabla_\nu \dddot{f}$.
Using (\ref{big_normal_3_dots_f}) we get
\begin{gather}
 (c^2 \nabla_\nu \dddot{f})\dot{} = (c^2)\dot{}\,\nabla_\nu \dddot{f} + 
c^2 ( (\nabla_\nu f)\dot{}\,\dot{}\,\dot{}\,\dot{}+ (\nabla_{[\nu,u]}f)\ddot{}\,\,\dot{} +
(\nabla_{[\nu,u]}\dot{f})\ddot{} + (\nabla_{[\nu,u]}\ddot{f})\dot{} ).
\nonumber
\end{gather}
From (\ref{estimate_boundary_f_3_dots}) we find
\begin{gather}
 \parallel (c^2)\dot{} \, \nabla_\nu \dddot{f} \parallel_{\partial,\frac{1}{2}} 
\leq K (k + k\parallel \ddot{f} \parallel_2 + k \parallel \dot{f} \parallel_3 + k^2 \parallel f \parallel_4 )
\nonumber
\end{gather}
so since $\parallel\dot{f} \parallel_3$ is assumed to be bounded 
\begin{gather}
 \big | \int_{\partial \Om} (c^2)\dot{} (\nabla_\nu \dddot{f}) \dddot{f} \big | \leq 
K\parallel \dddot{f}\parallel_1 ( k + k \parallel \ddot{f} \parallel_2 + k^2 \parallel f \parallel_4) .
\label{estimate_boundary_int_c_2_dot_normal}
\end{gather}
From (\ref{normal_f_3_dots_q}) we find that
\begin{gather}
 (\nabla_\nu f)\dot{}\,\dot{}\,\dot{}\,\dot{} = 
2q ( S_2(\ddddot{u},u) + S_2(\dddot{u},\dot{u}) + S_3(u,\dddot{u},u) ) + 2\dot{q}S_2(\dddot{u},u)
+ \ddddot{q} S_2(u) + \dddot{q} (S_2(u))\dot{} + (f_{R\partial})\dot{}
\nonumber
\end{gather}
$(f_{R\partial})\dot{}$ involves at most third material derivatives of $q$ and $u$ and so do all the other terms except
$2qS_2(\ddddot{u},u)$ and $\ddddot{q} S_2(u)$. Collecting them as $f_{R_1\partial}$ we get
\begin{gather}
(\nabla_\nu f )\dot{}\,\dot{}\,\dot{}\,\dot{} = 2qS_2(\ddddot{u},u) +  \ddddot{q} S_2(u)+ f_{R_1\partial} .
\label{boundary_f_4_dots_q_S}
\end{gather}
However, estimating as before we find
\begin{gather}
 \p f_{R_1\partial} \p_{\partial,\frac{1}{2}} 
\leq K (\sqrt{k} + k \parallel f \parallel_4 + k \parallel \dot{f} \parallel_3 + \p \ddot{f} \p_2 )
\label{estimate_f_R1_boundary}
\end{gather}
so we need only worry about the first two terms on the rightside of (\ref{boundary_f_4_dots_q_S}). Taking a material derivative 
of (\ref{u_i_3_dots}) we find that
\begin{gather}
 \ddddot{u} = u_R -c^2 \nabla \dddot{f}
 \nonumber
\end{gather}
where $u_R$ is a polynomial in (up to) third material derivatives of $c^2$ and in $u_j^i$, $\nabla f$, 
$\nabla \dot{f}$, and $\nabla \ddot{f}$. Thus we get
\begin{gather}
 \p u_R \p_1 \leq K ( k \p \ddot{f} \p_2 + k^\frac{3}{2} \p \dot{f} \p_3 
 + k^2 \p f \p_4 ) .
\label{estimate_u_R_boundary}
\end{gather}
Similarly, differentiating (\ref{q_dddot}) we get
\begin{gather}
 \ddddot{q} = -q^2(c^2)\dot{}\,\dot{}\,\dot{}\,\dot{} + q_R,
\nonumber
\end{gather}
where $q_R$ is a polynomial in $q$ and up to third material derivatives of $c^2$. Using the 
formulas (\ref{estimate_c2_k})-(\ref{estimate_c2_dddot}) we get
\begin{gather}
 \p q_R \p_1 \leq K(\sqrt{k} + \frac{1}{k} \p \dddot{f} \p_1 ) .
\label{estimate_q_R_boundary}
\end{gather}
But differentiating  (\ref{c2dddot}), we find
\begin{gather}
 (c^2)\dot{}\,\dot{}\,\dot{}\,\dot{}\ = p^{\prime\prime}(\rho) \rho \ddddot{f} + C_R,
\nonumber
\end{gather}
where $C_R$ is a polynomial in $\rho$ and material derivatives of $f$ whose coefficients are up to fifth derivatives of the 
function $p(\rho)$.

Estimating as before we get
\begin{gather}
 \p C_R \p_1 \leq K ( k + k^\frac{3}{2} \p \ddot{f} \p_2 + k \p \dddot{f} \p_1 )
\nonumber
\end{gather}
so
\begin{gather}
 \p q^2 C_R \p_1 \leq K ( \frac{1}{k} + \frac{1}{k} \p \dddot{f} \p_1 + \frac{1}{\sqrt{k}} \p \ddot{f} \p_2 ) .
\label{estiamte_q_2C_R}
\end{gather}
Using (\ref{boundary_f_4_dots_q_S}) and the above, we can write
\begin{gather}
 (\nabla_\nu f )\dot{}\,\dot{}\,\dot{}\,\dot{} = 2qS_2(c^2\nabla \dddot{f},u) -
q^2 p^{\prime\prime}(\rho)\rho \ddddot{f} S_2(u) + f_{R_2\partial}
\label{normal_f_4_dots_2}
\end{gather}
where 
\begin{gather}
 f_{R_2\partial} = 2qS_2(u_R, u)  - q_R S_2(u) - q^2C_RS_2(u) + f_{R_1\partial} .
\nonumber
\end{gather}
Using (\ref{estimate_f_R1_boundary})-(\ref{estimate_q_R_boundary}) and (\ref{estiamte_q_2C_R}) we get
\begin{gather}
 \p f_{R_2\partial} \p_{\partial,\frac{1}{2}} \leq K( \sqrt{k} + \frac{1}{k} \p \dddot{f} \p_1
+ \p \ddot{f}\p_2 + k \p \dot{f} \p_3 + k \p f \p_4 ) .
\label{estimate_boundary_f_R_2}
\end{gather}
We now use (\ref{normal_f_4_dots_2}) to estimate
\begin{gather}
 \int_0^t \int_{\partial\Om} c^2 (\nabla_\nu f )\dot{}\,\dot{}\,\dot{}\,\dot{} \, \dddot{f} .
\nonumber
\end{gather}
On $\partial\Om$ we can write
\begin{gather}
 \nabla \dddot{f} = \nabla_\partial \dddot{f} + \nabla_\nu \dddot{f}
\label{decomp_f_4_dots_tan_normal}
\end{gather}
where $\nabla_\partial \dddot{f}$ is the gradient of $\dddot{f}_{|\partial \Om}$. Then
\begin{align}
\begin{split}
\int_0^t \int_{\partial\Om} (\nabla_\nu f )\dot{}\,\dot{}\,\dot{}\,\dot{} \, \dddot{f}
= &\int_0^t \int_{\partial \Om} \Big ( S_2(c^2 \nabla_\partial (\dddot{f})^2,u) + 2S_2(c^2\nabla_\nu \dddot{f},u)\dddot{f} \\
&-\frac{1}{2} q p^{\prime\prime}(\rho) \rho S_2(u) ((\dddot{f})^2)\dot{} + c^2 (f_{R_2\partial})\dddot{f} \Big ).
\label{int_0_t_normal_4_dots}
\end{split}
\end{align}
From (\ref{estimate_boundary_f_3_dots}) and (\ref{estimate_boundary_f_R_2}) we see that the second and forth
terms of (\ref{int_0_t_normal_4_dots}) are bounded by
\begin{gather}
 K\int_0^t (k^\frac{3}{2} + \p \dddot{f} \p_1 + k \p \ddot{f} \p_2 + k^2 \p \dot{f} \p_3 + k^2 \p f \p_4)\p \dddot{f} \p_1 .
\label{second_fourth}
\end{gather}
Since $\nabla_\partial$ involves derivative only along $\partial \Om$ we can integrate by parts to get
\begin{gather}
 \Big | \int_0^t \int_{\partial \Om}S_2(c^2 \nabla_\partial (\dddot{f})^2,u) \Big | \leq K \int_0^t k \p \dddot{f} \p_1^2 .
\label{big_estimate_int_0_t_1}
\end{gather}
Also integrating by parts with respect to material derivative we find that
\begin{align}
\begin{split}
 \Big | \int_0^t \int_{\partial \Om} q p^{\prime\prime}(\rho) \rho S_2(u) ( (\dddot{f})^2)\dot{} \Big |  \leq & 
\Big | \int_{\partial\Om} qp^{\prime\prime}(\rho) S_2(u) (\dddot{f})^2 \Big |_0^t\Big | 
\\
+
& K\int_0^t \p \dddot{f}\p_1^2 (1 + \p\dot{f}\p_3 + k \p f \p_4 ) .
\end{split}
\label{big_estimate_int_0_t_2}
\end{align}
But 
\begin{gather}
 \int_{\partial\Om} qp^{\prime\prime}(\rho) S_2(u) (\dddot{f})^2 \leq K \p \dddot{f} \p_1^2
\nonumber
\end{gather}
so combining
(\ref{second_fourth}), 
 (\ref{big_estimate_int_0_t_1}), (\ref{big_estimate_int_0_t_2}) we get
\begin{align}
\begin{split}
 \Big | \int_0^t \int_{\partial \Om} c^2 (\nabla_\nu f)\dot{}\,\dot{}\,\dot{}\,\dot{}\, \dddot{f} \big | 
\leq 
& K( \p \dddot{f}(0)\p_1^2 + \p \dddot{f}(t) \p_1^2 ) 
+
K \int_0^t \p \dddot{f} \p_1 \Big ( k^\frac{3}{2} + \p \dddot{f} \p_1 
\\
& + k \p \ddot{f} \p_2 + k^2 \p \dot{f} \p_3 + k^2 \p f \p_4 \Big) .
\end{split}
\label{estimate_boundary_int_c_2_normal} 
\end{align}
To complete our estimate of $\int_0^t \int_{\partial \Om} (c^2 \nabla_\nu \dddot{f} )\dot{} \, \dddot{f}$ we must estimate
\begin{gather}
\int_0^t \int_{\partial\Om} c^2 \Big \{  (\nabla_{[\nu,u]}f)\,\ddot{}\,\,\dot{} + 
(\nabla_{[\nu,u]}\dot{f})\,\ddot{} + (\nabla_{[\nu,u]}\ddot{f})\dot{} \dot{} \Big \} \dddot{f} .
\nonumber
\end{gather}
Computing the material derivative we find that the term in braces equals
\begin{gather}
 3\nabla_{[\nu,u]}\dddot{f} + \nabla_{[\nu,u]\,\ddot\,\,\dot{}}f + f_{R_3}
\nonumber
\end{gather}
where $f_{R_3}$ is a polynomial in $c^2$ and its first two material 
derivatives as well as $u_j^i$, $\nabla f$, $\nabla \dot{f}$, and $\nabla\ddot{f}$. 
Thus we get
\begin{gather}
 \p f_{R_3} \p_1 \leq K ( \sqrt{k} \p \ddot{f} \p_2 + k \p \dot{f} \p_3 + k \p f \p_4 ),
\nonumber
\end{gather}
so
\begin{gather}
 \Big | \int_0^t \int_{\partial \Om} c^2 f_{R_3} \dddot{f} \Big | \leq 
K \int_0^t ( \sqrt{k} \p \ddot{f} \p_2 + k \p \dot{f} \p_3 + k \p f \p_4)\p \dddot{f}\p_1 .
\label{estiamte_int_f_R_3}
\end{gather}
We decompose $\nabla_{[\nu,u]} \dddot{f}$ in to tangential and normal derivatives of $\dddot{f}$ (as we did in
(\ref{decomp_f_4_dots_tan_normal})), getting
\begin{gather}
\nabla_{[\nu,u]} \dddot{f} = ( \nabla_{[\nu,u]} \dddot{f} )_\partial + (\nabla_{[\nu,u]} \dddot{f})_\nu .
\nonumber
\end{gather}
Then integration by parts gives
\begin{gather}
\Big | \int_{\partial \Om} c^2 ( \nabla_{[\nu,u]} \dddot{f} )_\partial \dddot{f} \Big | \leq K k \p \dddot{f} \p_1^2  .
\label{estimate_int_term_partial}
\end{gather}
Also $(\nabla_{[\nu,u]} \dddot{f})_\nu = \langle [\nu,u],\nu \rangle \nabla_\nu \dddot{f}$ so 
from (\ref{estimate_boundary_f_3_dots}) we get
\begin{gather}
\Big | \int_{\partial \Om} c^2 ( \nabla_{[\nu,u]} \dddot{f} )_\nu \dddot{f} \Big | \leq
K (k + k \n \ddot{f} \n_2 + k \n \dot{f} \n_3 + k^2 \n f \n_4 )\n \dddot{f} \n_1 .
\label{estimate_int_normal}
\end{gather}
We now deal with the remaining term $\nabla_{[\nu,u]\,\ddot{}\,\,\dot{}} f$. Computing, we find that
\begin{gather}
 [\nu,u]\,\ddot{}\,\,\dot{} = -c^2 \nabla \nabla_\nu \ddot{f} + u_{R_2}
\nonumber
\end{gather}
where $u_{R_2}$ obeys the inequality
\begin{gather}
 \n u_{R_2} \n_{\partial,\frac{1}{2}} \leq K (k^\frac{3}{2} \n f \n_4 + k \n \dot{f} \n_3) .
\label{estiamte_u_R_2}
\end{gather}
Using the tangential and normal decomposition on $\partial \Om$ we get
\begin{gather}
 -c^2 \nabla\nabla_\nu \ddot{f} = -c^2 \nabla_\partial \nabla_\nu \ddot{f} - c^2 \nabla_\nu \nabla_\nu \ddot{f} .
\nonumber
\end{gather}
But 
\begin{gather}
 c^2 \nabla_\nu \nabla_\nu \ddot{f} = L\ddot{f} + c^2 D_\tau \nabla_\nu \ddot{f} - L_\partial \ddot{f} + L_R \ddot{f}
\nonumber
\end{gather}
where $D_\tau$ is a first order operator involving only tangential derivatives, 
$L_\partial = -\de_\partial c^2 \nabla_\partial$ ($\nabla_\partial$ being the gradient of the 
manifold $\partial \Om$ and $\de_\partial$ its formal adjoint), and $L_R$ is a first order operator whose 
coefficients involve $c^2$ and its first derivatives.

Thus we find
\begin{align}
\begin{split}
 \int_{\partial \Om}
 &  c^2 (\nabla_{[\nu,u]\,\ddot{}\,\,\dot{}}f)\dddot{f} = 
\int_{\partial \Om} \Big ( c^2 \nabla_{u_{R_2}} f - c^4(\nabla_\partial\nabla_\nu \ddot{f} ) (\nabla_\partial f) \\
&
-c^2(L\ddot{f} + c^2 D_\tau \nabla_\nu \ddot{f} - L_\partial \ddot{f} + L_R \ddot{f} ) \nabla_\nu f \Big ) \dddot{f} \,\, .
\end{split}
\nonumber
\end{align}
Using (\ref{estiamte_u_R_2}) we find that
\begin{gather}
\Big | \int_{\partial \Om} ( c^2 \nabla_{u_{R_2}} f - (c^2L_R\ddot{f})\nabla_\nu f)\dddot{f} \Big |  \leq 
K \n \dddot{f} \n_1 ( k^\frac{3}{2}\n f \n_4 + k \n \dot{f} \n_3 + k \n \ddot{f} \n_2 ) .
\nonumber
\end{gather}
Since 
\begin{gather}
 \nabla_\nu \ddot{f} = (\nabla_\nu f)\ddot{} + (\nabla_{[\nu,u]} f)\dot{} + \nabla_{[\nu,u]}\dot{f}
\nonumber
\end{gather}
and 
\begin{gather}
 (\nabla_\nu f)\ddot{}  = -(qS_2(u))\ddot{}
\nonumber
\end{gather}
we also get the estimate
\begin{align}
\begin{split}
 \Big | \int_{\partial\Om} c^4  \big ( (\nabla_\partial \nabla_\nu \ddot{f} )\nabla_\partial f + 
(D_\tau \nabla_\nu \ddot{f})\nabla_\nu f \big ) \dddot{f} \Big | 
& \leq 
K \n \dddot{f} \n_1 ( k^\frac{3}{2} + k^2\n f \n_4 + k^\frac{3}{2} \n \dot{f} \n_3 ) \\
& \leq K k^\frac{3}{2} \n \dddot{f} \n_1 .
\end{split}
\nonumber
\end{align}
It remains to estimate the terms
\begin{gather}
 \int_0^t \int_{\partial \Om} c^2(L_\partial \ddot{f})\dddot{f} \,\, \text{  and  } \,\, 
 \int_0^t \int_{\partial \Om}
c^2 (L\ddot{f}) \dddot{f} .
\nonumber
\end{gather}
The first of these equals
\begin{gather}
 -\int_0^t \int_{\partial\Om} c^2 \langle \nabla_\partial \ddot{f},\nabla_\partial c^2 \dddot{f} \rangle ,
\nonumber
\end{gather}
which equals
\begin{gather}
 -\frac{1}{2} \int_0^t \int_{\partial\Om} c^4 (| \nabla_\partial \ddot{f}|^2)\dot{} 
 + \int_0^t R_2, 
\nonumber
\end{gather}
where 
\begin{gather}
 |R_2| \leq K k^2 \n \ddot{f} \n_2^2.
\nonumber
\end{gather}
Using  (\ref{f_4_dots_L_1}) we find that the second term equals
\begin{gather}
 \int_0^t \int_{\partial \Om} c^2(\ddddot{f} - L_1 \dot{f} - (L_1 f)\dot{} -\ddot{F} )\dddot{f} \, ,
\nonumber
\end{gather}
which equals
\begin{gather}
 \frac{1}{2} \int_0^t \int_{\partial\Om} c^2( (\dddot{f})^2 )\dot{} + \int_0^t R_3,
\nonumber
\end{gather}
and from (\ref{first_step_I_1}), (\ref{second_step_I_1}) and (\ref{estimate_on_F_ddot}) we find that 
\begin{gather}
 |R_3| \leq K \n \dddot{f} \n_1 (k \n\dot{f}\n_3 + k^\frac{3}{2} \n f \n_4) .
\nonumber
\end{gather}
Therefore
\begin{gather}
 \int_0^t \int_{\partial \Om} c^2 (\nabla_{[\nu,u]\,\ddot{}\,\,\dot{}} f) \dddot{f} = 
-\frac{1}{2}\int_{\partial\Om} \big ( c^4 |\nabla_\partial \ddot{f} |^2 + c^2 |\dddot{f}|^2 \big ) \Big |_0^t + \int_0^t R_4
\label{remaining_term}
\end{gather}
where 
\begin{gather}
 |R_4| \leq  K (k^\frac{3}{2} \n \dddot{f} \n_1 + k \n \dddot{f} \n_1^2 + k^2 \n \ddot{f} \n_2^2 ) .
\nonumber
\end{gather}
We let $P(t)$ stand for the term 
$\frac{1}{2}\int_{\partial\Om} \big ( c^4 |\nabla_\partial \ddot{f} |^2 + c^2 |\dddot{f}|^2 \big )$
so that  (\ref{remaining_term}) becomes
\begin{gather}
 \int_0^t \int_{\partial \Om} c^2 (\nabla_{[\nu,u]\,\ddot{}\,\,\dot{}} f) \dddot{f} = P(0) - P(t) + \int_0^t R_4 .
\label{int_in_terms_P}
\end{gather}
Combining (\ref{estiamte_int_f_R_3}), (\ref{estimate_int_term_partial}), (\ref{estimate_int_normal}) and
(\ref{int_in_terms_P}), and using our assumed bounds on $\n f \n_4$, $\n \dot{f} \n_3$, and $\n \ddot{f} \n_2$
we find
\begin{gather}
\int_0^t \int_{\partial\Om} c^2 \Big \{  (\nabla_{[\nu,u]}f)\,\ddot{}\,\,\dot{} + 
(\nabla_{[\nu,u]}\dot{f})\,\ddot{} + (\nabla_{[\nu,u]}\ddot{f}) \dot{} \Big \} \dddot{f} =
P(0) - P(t) + \int_0^t R_5
\label{remaining_big_terms_P}
\end{gather}
where
\begin{gather}
 |R_5| \leq K ( k \n \dddot{f} \n_1^2 + k^\frac{3}{2} \n \dddot{f} \n_1 + k^2 \n \ddot{f} \n_2^2 ) .
\nonumber
\end{gather}
Then combining (\ref{estimate_boundary_int_c_2_dot_normal}), (\ref{estimate_boundary_int_c_2_normal}) and 
(\ref{remaining_big_terms_P}) we find
\begin{gather}
 P(t) + \int_0^t\int_{\partial\Om} (c^2 \nabla_\nu \dddot{f} )\,\dot{} \, \dddot{f} \leq
P(0) + K(\n \dddot{f}(0)\n_1^2 + \n\dddot{f}(t)\n_1^2 ) \nonumber \\
+  K \int_0^t(k^\frac{3}{2} \n \dddot{f}\n_1 + k \n \dddot{f} \n_1^2 + k^2 \n \ddot{f} \n_2^2 + k^2 \n \dddot{f} \n_1 \n \dot{f} \n_3) .
\nonumber
\end{gather}
Using (\ref{integral_B_11}), and (\ref{estimate_error_integral_B_11}) and the fact that
\begin{gather}
 k^2\n \dddot{f} \n_1 \n \dot{f} \n_3 \leq \frac{1}{2}(k \n \dddot{f} \n_1^2 + k^3 \n \dot{f} \n_3^2)
\nonumber
\end{gather}
we get
\begin{align}
\begin{split}
 P(t)  + \int_0^t  B_{11}  \leq & \, P(0) + K(Q(t) + Q(0))   \\
 & + 
K \int_0^t( k^\frac{3}{2} \n \dddot{f}\n_1 + k \n \dddot{f} \n_1^2 + k^2 \n \ddot{f} \n_2^2 + k^3 \n \dot{f} \n_3^2  ),
\end{split}
\label{estimate_on_integral_B_11}
\end{align}
where $Q$ is a function of time defined by
\begin{gather}
 Q = \n \dddot{f} \n_1^2 + \n \dddot{f} \n_1(k + k \n \ddot{f} \n_2 + k \n \dot{f} \n_3 + k^2 \n f \n_4 ) .
\nonumber
\end{gather}
Clearly this $Q$ obeys the inequality
\begin{gather}
 Q \leq K (k^\frac{1}{2} \n \dddot{f}\n_1^2 + k^\frac{3}{2} + k^\frac{3}{2} \n \ddot{f} \n_2^2 +
k^\frac{3}{2} \n \dot{f} \n_1^2 + k^\frac{7}{2} \n f \n^2_4 ) .
\label{ineq_Q}
\end{gather}
Finally using (\ref{estimate_on_I_1}), (\ref{estimate_on_R_1}), (\ref{estimate_on_B_1R}) and 
(\ref{estimate_on_integral_B_11}) we get
\begin{align}
\begin{split}
 E(t) + P(t) \leq & \, E(0) + P(0) + 
K(Q(t) + Q(0)) \\
&  + 
K \int_0^t( k \n \dddot{f}\n_1^2 + k^2 \n \ddot{f} \n_2^2 + k^3 \n \dot{f} \n_3^2 + k^4 \n f \n_4^2 + k^2 ).
\end{split}
\label{estimate_growth_E}
\end{align}
This is our estimate for the growth of $E(t)$.

We now proceed to show that $E$ gives a bound for the norms of $f$ and its material derivatives. First we 
note that since $c^2$ is in $H^4$, $L : H^s \rar H^{s-2}$ ($2 \leq s \leq 5$) is a bounded operator with null 
space and co-null space equal to the constant functions. Because of this we will need the decomposition
\begin{gather}
 H^s = H_0^s\oplus\operatorname{Con}
\label{decomp_H_s}
\end{gather}
given by
\begin{gather}
 g = g_1 + g_2
\nonumber
\end{gather}
where $g_2$ is the constant function
\begin{gather}
 g_2 = \frac{\int_\Om g}{\int_\Om 1 } .
\nonumber
\end{gather}
Of course, $H^s_0$ is the space of $H^s$ functions whose integral is zero and $\operatorname{Con}$ stands for the constants. 
We shall use this decomposition on $f$ and its material derivatives. To simplify calculations we shall assume that $\int_\Om 1 = 1$.

First we note that
\begin{gather}
\int_\Om \dot{f} = \int_\Om \de(u) = 0 
\label{f_dot_2_zero}
\end{gather}
so
\begin{gather}
(\dot{f})_2 = 0 \text{ and } \dot{f} = (\dot{f})_1  .
\label{dot_f_equal_dot_f_1}
\end{gather}
Note also that 
\begin{gather}
 f_2(t) = f_2(0)  + \int_0^t \partial_t f_2 = f_2(0) + \int_0^t \int_\Om \partial_t f=f_2(0) + \int_0^t\int_\Om \dot{f}
 - \int_0^t \int_\Om \nabla_u f.
\label{expression_f_2}
\end{gather}
Therefore
\begin{gather}
 |f_2(t)| \leq |f_2(0)| + \int_0^t \Big |\int_\Om \nabla_u f \Big | \leq |f_2(0)| + K  \int_0^t \n f \n_0.
\label{int_by_parts_argument}
\end{gather}
To obtain this last inequality we integrated by parts and used the boundary condition,
\begin{align}
\begin{split}
\int_\Om \nabla_u f & = -\int_\Om \dive (u) f + \int_{\partial \Om} \langle u, \nu \rangle f \\
& =  -\int_\Om \dive(u) f,
\end{split}
\nonumber
\end{align}
and then invoked the Cauchy-Schwarz inequality,
\begin{align}
\begin{split}
\left| \int_\Om \dive (u) f \right| \leq \p \dive u \p_0 \p f \p_0 \leq   \p u \p_1 \p f \p_0 ,
\end{split}
\nonumber
\end{align}
absorbing the norm of $u$ (which is bounded because of Assumption  \ref{extra_assump_u}) into the constant $K$. 

But for any function $g$, and any $s\geq 0$
\begin{gather}
 \n g \n_s \leq K ( \n g_1 \n_s + \n g_2 \n_s ) .
\nonumber
\end{gather}
Hence 
\begin{gather}
 |f_2(t)| \leq |f_2(0)| + K  \int_0^t  (\n f_2 \n_0 + \n f_1 \n_0 ) .
\label{ineq_f_2_before_iteration}
\end{gather}
Let
\begin{gather}
 \n g \n_{s\max} = \max\{ \n g(\tau) \n_s | ~0\leq \tau \leq t\} .
\nonumber
\end{gather}
Then iterating inequality (\ref{ineq_f_2_before_iteration}) we get
\begin{gather}
 |f_2(t)| \leq e^{Kt}|f_2(0)| + \n f_1 \n_{0\max}(e^{Kt} - 1) .
\label{f_2_t_iterated}
\end{gather}
From this  and (\ref{dot_f_equal_dot_f_1}) we find
\begin{gather}
 \n f(t) \n_s \leq K ( |f_2(0)| + \n f_1 \n_{s\max} ) .
\label{estimate_f_f_2_zero_f_1}
\end{gather}
We now proceed to get similar inequalities for the higher material derivatives of $f$.
\begin{gather}
 |(\ddot{f})_2| = \Big | \int_\Om (\partial_t \dot{f} + \nabla_u \dot{f} ) \Big | \leq K \n \dot{f} \n_0 ,
\nonumber
\end{gather}
where we have integrated by parts as we did in the derivation of (\ref{int_by_parts_argument}).
Therefore
\begin{gather}
 \n \ddot{f} \n_s \leq K ( \n \ddot{f}_1 \n_s + \n \dot{f} \n_0 ) .
\label{estimate_f_ddot_s}
\end{gather}
Also
\begin{gather}
 (\dddot{f})_2 = \int_\Om(\partial_t^2 \dot{f} + 2 \partial_t \nabla_u \dot{f} + \nabla_u \nabla_u \dot{f} ) .
\nonumber
\end{gather}
But
\begin{gather}
 \int_\Om \partial_t^2 \dot{f} = 0 \,\, \text{ and } \,\,  \Big | \int_\Om \nabla_u \nabla_u \dot{f} \Big | \leq K \n \dot{f} \n_3 .
\nonumber
\end{gather}
Furthermore
\begin{gather}
 2\int_\Om \partial_t \nabla_u \dot{f} = 2 \int_\Om \de(u) \partial_t \dot{f} = 2 \int_\Om \dot{f}\partial_t \dot{f}=
2\int_\Om \dot{f}(\ddot{f} - \nabla_u f) .
\nonumber
\end{gather}
And since $\n \dot{f} \n_3$ is assumed to be bounded
\begin{gather} 
\Big | \int_\Om \dot{f}(\ddot{f} - \nabla_u f) \Big | \leq K(\n \ddot{f} \n_2 + \n \dot{f} \n_3 ) .
\nonumber
\end{gather}
It follows that
\begin{gather}
 | (\dddot{f})_2 | \leq K (\n \ddot{f} \n_2 + \n \dot{f} \n_3 ),
\nonumber
\end{gather}
so
\begin{gather}
 \n \dddot{f} \n_1 \leq K ( \n \ddot{f} \n_2 + \n \dot{f} \n_3 + \n (\dddot{f})_1 \n_1 ) .
\label{estimate_f_dddot_s}
\end{gather}
Combining (\ref{dot_f_equal_dot_f_1}), (\ref{estimate_f_f_2_zero_f_1}), (\ref{estimate_f_ddot_s}) and
(\ref{estimate_f_dddot_s}) we find that the norms of the first components of $f$ and its material derivatives
(in the decomposition (\ref{decomp_H_s})) bound the norms of $f$ and its derivatives. Specifically we get, after 
multiplication by various powers of $k$,
\begin{gather}
 k^2 \n f\n_4 + k^\frac{3}{2} \n \dot{f} \n_3 + k \n \ddot{f} \n_2 + k^\frac{1}{2} \n \dddot{f} \n_1 \nonumber \\
\leq
K\big ( k^2\n f_1\n_{4\max} + k^\frac{3}{2} \n (\dot{f})_1 \n_3 + k \n (\ddot{f})_1 \n_2 +
k^\frac{1}{2} \n (\dddot{f})_1 \n_1 + k^2 |f_2(0)| \big ) .
\nonumber
\end{gather}
From this inequality we know that it is sufficient to estimate the $H_0^s$ components. This is equivalent to working with the 
functions modulo additive constants. Therefore, for the rest of this section we shall work modulo constants; that is, we will 
simply equate $\n g \n_s$ and $\n g_1 \n_s$.

To get bounds for $f$ and its derivatives we will need some results about elliptic boundary value problems. First we consider 
the Neumann problem for the Laplacian. We let $N: H^s(\Om) \rar H^{s-\frac{3}{2}}(\partial \Om)$ ($s > \frac{3}{2}$) be the 
operator defined by:
$Ng = R(\nabla_\nu g)$, where $R: H^{s-1}(\Om) \rar H^{s-\frac{3}{2}}(\partial \Om)$ is simply the restriction of a 
function to $\partial \Om$.

It is well known that $N$ is continuous and surjective and admits a continuous right inverse $G_\Delta:
H^{s-\frac{3}{2}}(\partial \Om) \rar H^s(\Om)$ where $G_\Delta$ solves the Neumann problem, i.e., $G_\Delta h = g$ is the solution of
\begin{gather}
\begin{cases}
  \Delta g = 0, \\
  R \nabla_\nu g = h .
  \end{cases}
\nonumber
 \end{gather}
Also it is well known that $\Delta: H_\nu^s(\Om) \rar H^{s-2}(\Om)$ ($s\geq 2$) is bijective (modulo additive constants) where
\begin{gather}
 H_\nu^s(\Om) = \{ g \in H^s(\Om) ~|~ R\nabla_\nu g = 0 \} .
\nonumber
\end{gather}
We denote its inverse by
\begin{gather}
 \Delta^{-1}: H^{s-2} \rar H_\nu^s .
\nonumber
\end{gather}
From this we conclude that for any $g \in H^s(\Om)$, $s\geq 2$, we get
\begin{gather}
 g = \Delta^{-1}\Delta g + G_\Delta Ng 
\nonumber
\end{gather}
and it follows that
\begin{gather}
 \n g \n_s \leq K ( \n \Delta g \n_{s-2} + \n Ng \n_{\partial,s-\frac{3}{2}} ).
\label{basic_elliptic_estimate}
\end{gather}
Now we shall
 derive an inequality like (\ref{basic_elliptic_estimate}) using $L$ instead of $\Delta$. 
Since $\n c^2 \n_4 \leq Kk$ we know that $L: H^s \rar H^{s-2}$ has operator norm 
bounded by $Kk$ for $2 \leq s \leq 5$. Also from (\ref{estimate_c2_k}) we find
 that $\frac{1}{k}L - \Delta = \de(\frac{c^2}{k} - 1)\nabla: H^s \rar H^{s-2}$ has operator 
norm bounded by $\frac{K}{\sqrt{k}}$. Therefore $\frac{1}{k} L \Delta^{-1} - I : H^{s-2} \rar H^{s-2}$ has norm 
bounded by $\frac{K}{\sqrt{k}}$, where $I$ means the identity operator. But since one can invert operators near $I$ by a 
Neumann series we find that for any $\la > 1$, if $\sqrt{k} > \la K$, then $\frac{1}{k} L \Delta^{-1}$ has an 
inverse whose norm is less than $\frac{\la}{\la-1}$. Also $L: H_\nu^s \rar H^{s-2}$ equals $k(\frac{1}{k}L\Delta^{-1})\Delta$, 
so for large $k$,
\begin{gather}
 L:H^s_\nu \rar H^{s-2}
\nonumber
\end{gather}
is bijective and $L^{-1} = \frac{1}{k} \Delta^{-1} (\frac{1}{k} L \Delta^{-1})^{-1}$ has norm bounded by $\frac{K}{k}$.

Furthermore we can define 
$G_L: H^{s-\frac{3}{2}} (\partial \Om) \rar H^s(\Om)$ ($2\leq s \leq 5$) which 
plays a role analogous to $G_\Delta$. We let $G_L = G_\Delta - L^{-1}LG_\Delta$. Then one sees easily that $G_Lh$ is a 
solution to the boundary value problem:
\begin{gather}
\begin{cases}
 Lg = 0, \\
 R \nabla_\nu g = h.
 \end{cases}
\nonumber
\end{gather}
Also $G_L: H^{s-\frac{3}{2}} (\partial \Om) \rar H^s(\Om)$ clearly has a norm bounded independent of $k$, i.e., 
\begin{gather}
 \n G_L \n \leq \n G_\Delta \n + \n L L^{-1} G_\Delta \n \leq \n G_\Delta \n (1 + K \frac{1}{k} k) .
\nonumber
\end{gather}
As with the Laplacian, we get the equation
\begin{gather}
 g = L^{-1} L g + G_L N g
\label{decomp_g_L}
\end{gather}
and from this we get the inequality
\begin{gather}
 \n g \n_s \leq K (\frac{1}{k} \n Lg\n_{s-2} + \n Ng \n_{\partial, s-\frac{3}{2}} ).
\label{basic_L_estimate}
\end{gather}
This inequality will be our main tool used to estimate norms of $f$ and its derivatives. We begin with the norm of $f$ itself. From 
our basic equations we find that
\begin{gather}
L f = \ddot{f} - F \,\, \text{ and }\,\,  Nf = qS_2(u).
\nonumber
\end{gather}
Therefore
\begin{gather}
 \n f \n_4 \leq \frac{K}{k}(\n \ddot{f} \n_2 + 1).
\label{bound_f_L}
\end{gather}
Also differentiating these equations we get
\begin{gather}
 L \dot{f} = \dddot{f} - (L_1\dot{f} + \dot{F} ) \,\, \text{ and } \,\,  N\dot{f} = (qS_2)(u)\dot{} + \nabla_{[\nu,u]}f.
\nonumber
\end{gather}
But $\dot{F} = (u^i_j)\dot{}u^j_i + u^i_j(u^j_i)\dot{}$ so from (\ref{u_i_j_dot}) we get
\begin{gather}
 \n \dot{F} \n_2 \leq K(k\n f \n_4 + 1 ) \leq K\sqrt{k}
\nonumber
\end{gather}
and from the formula (\ref{computing_L_1}) for $L_1$ we find
\begin{gather}
 \n L_1 f \n_1 \leq K k\n f \n_4 \leq K \sqrt{k}.
\nonumber
\end{gather}
From (\ref{S_2_dot_terms_S_2}) and (\ref{various_estimates_q_dot})
we find 
\begin{gather}
 \n (qS_2(u))\dot{} \n_{\partial,\frac{3}{2}} \leq K (\frac{1}{k} + \n f \n_4 ) \leq \frac{K}{\sqrt{k}}
\nonumber
\end{gather}
and of course
\begin{gather}
 \n \nabla_{[\nu,u]} f \n_{\partial,\frac{3}{2}} \leq K \n f \n_4 \leq \frac{K}{\sqrt{k}}.
\nonumber
\end{gather}
Combining these four inequalities and using (\ref{basic_L_estimate}), we get
\begin{gather}
 \n \dot{f} \n_3 \leq K(\frac{1}{k}\n \dddot{f} \n_1 + \frac{1}{\sqrt{k}}).
\label{bound_f_dot_L}
\end{gather}
From the definition of $E$, we know $\n L\ddot{f} \n_0 \leq \sqrt{E}$, so to bound $\ddot{f}$ we need only look at 
boundary data. But
\begin{gather}
 N\ddot{f} = (qS_2(u))\,\ddot{} + (\nabla_{[\nu,u]}f)\dot{} + \nabla_{[u,\nu]} \dot{f}.
\nonumber
\end{gather}
Since $\n \dot{f} \n_3$ is assumed bounded and 
$\n f \n_4 \leq \frac{K}{\sqrt{k}}$, we get $\n N\ddot{f} \n_{\partial,\frac{1}{2}} \leq K$. Therefore
\begin{gather}
 \n \ddot{f} \n_2 \leq K(\frac{1}{k}\sqrt{E} + 1 ).
\label{bound_f_ddot_L}
\end{gather}
Finally we have the obvious inequality
\begin{gather}
 \n \dddot{f} \n_1 \leq K \sqrt{\frac{E}{k}}, 
\label{bound_f_dddot_L}
\end{gather}
so we have bounds for $f$ and its first three material derivatives. Combining 
(\ref{bound_f_L}) and (\ref{bound_f_ddot_L}) we find
\begin{gather}
 \n f \n_4 \leq K (\frac{1}{k} + \frac{1}{k^2} \sqrt{E} )
\label{bound_f_E}
\end{gather}
and combining (\ref{bound_f_dot_L}) and (\ref{bound_f_dddot_L}) we get
\begin{gather}
 \n \dot{f} \n_3 \leq K (\frac{1}{\sqrt{k}} + \frac{1}{k^{\frac{3}{2}}} \sqrt{E} ).
\label{bound_f_dot_E}
\end{gather}
Thus combining (\ref{bound_f_ddot_L})-(\ref{bound_f_dot_E}) we get all the bounds in terms of $E$ itself.

Now we are ready to get bounds on the growth of $f$ and its derivatives.

We let 
\begin{gather}
E_1 = k^4 \n f \n_4^2 + k^3 \n \dot{f} \n_3^2 + k^2 \n \ddot{f} \n_2^2 + k \n \dddot{f} \n_1^2 + k^2.
\nonumber
\end{gather}
Then (\ref{bound_f_ddot_L})-(\ref{bound_f_dot_E}) tell us that
\begin{gather}
 E_1 \leq K(E + k^2).
\nonumber
\end{gather}
Also from (\ref{ineq_Q}) we find $Q \leq \frac{K}{\sqrt{k}} E_1$ and from 
(\ref{estimate_growth_E}) we get
\begin{gather}
 E(t) + P(t) \leq E(0) + P(0) +\frac{K}{\sqrt{k}} (E_1(0)+E_1(t)) + K\int_0^t E_1(s) ds .
\label{E_plus_P_positive}
\end{gather}
But since $P(t)$ is positive and since $E(0) + P(0) \leq K E_1(0)$ using
(\ref{E_plus_P_positive}) we get
\begin{gather}
 E_1(t) \leq K E_1(0) + \frac{K}{\sqrt{k}}(E_1(0) + E_1(t) ) + K \int_0^t E_1(s)ds .
\nonumber
\end{gather}
Then if $\sqrt{k} > K$ by iterating this inequality we can get a constant $K_1$ depending only on $K$ such that
\begin{gather}
 E_1(t) \leq K_1 E_1(0) e^{K_1 t} .
\nonumber
\end{gather}
We assumed that $\n f(0) \n_4 \leq \frac{K}{k}$, $\n \dot{f}(0) \n_3 \leq \frac{K}{\sqrt{k}}$ and this implies that
$\n \ddot{f}(0) \n \leq K$ (using (\ref{eq_f_estimates})) and that $\n \dddot{f}(0)\n_1 \leq K\sqrt{k}$. These assumptions 
tell us that $E_1(0) \leq K k^2 $.
This together with the assumptions made at the beginning of the section (i.e. the bound (\ref{extra_assump_u}) on $u$ and 
Assumption \ref{assump_on_f} on and $f$ and material derivatives of $f$) tell us that for large $k$
\begin{gather}
 E_1(t) \leq K k^2 K_1 e^{K_1 t}.
\label{final_ineq_E_1}
\end{gather}
But the assumptions (\ref{extra_assump_f})-(\ref{extra_assump_f_dddot}) are equivalent to the inequality:
$E_1(t) \leq 4a^2_4 k^3 + k^2$. Therefore the assumptions $E_1(0) \leq K k^2$, (\ref{extra_assump_u}), 
and $E_1(t) \leq 4a_4^2k^3 + k^2$, imply $E_1(t) \leq K K_1 k^2 e^{K_1t}$. Let us now fix $T>0$. Then if $k$ is 
sufficiently large, for all $t \in [0,T]$, we have $K K_1 k^2 e^{K_1t} < 4a_4^2 k^3 + k^2$. Therefore the 
inequalities $E_1(0) \leq K k^2$, (\ref{extra_assump_u}), and $E_1(t) \leq 4a_2 k^3 + k^2$ imply $E_1(t) \leq KK_1 k^2 e^{K_1 t}$. 
Since $E_1$ is a continuous function, this tells us that if (\ref{hypothesis_at_0}) and (\ref{extra_assump_u}) hold, 
then for large $k$, $E_1(t) \leq K K_1 k^2 e^{K_1 t}$ for all $t \in [0,T]$. That is Assumption \ref{assump_on_f} is 
no longer necessary. The inequalities 
(\ref{hypothesis_at_0}) and (\ref{extra_assump_u}) imply that for large $k$, (\ref{final_ineq_E_1}) holds. 
But since $E_1$ bounds $f$ and its material derivatives, (\ref{final_ineq_E_1}) implies
\begin{gather}
 \n f \n_4 \leq \frac{K}{k},~\n \dot{f} \n_3 \leq \frac{K}{\sqrt{k}},~\n \ddot{f} \n_2 \leq K \text{ and } 
\n \dddot{f} \n_1 \leq K \sqrt{k} .
\nonumber
\end{gather}
This $K$ is the desired $a_5$ of (\ref{ineq_a_5_f})-(\ref{ineq_a_5_f_dddot}).

Thus we have shown that the assumptions (\ref{hypothesis_at_0}) and (\ref{extra_assump_u}) imply that there is a 
constant $a_5$ depending only on $\Om$, $T$, $a_2$ and $a_3$ such that (\ref{ineq_a_5_f})-(\ref{ineq_a_5_f_dddot}) hold for large $k$.

\begin{rema} The argument given to prove (\ref{ineq_a_5_f})-(\ref{ineq_a_5_f_dddot}) is similar to that used for Proposition 12.13 of \cite{E2}  and for  Lemma 4.6 of \cite{E4}.
\end{rema} 

\section{Proofs of Main Theorems \label{proofs}}

We begin this section with the proof of Theorem \ref{first_big_theo}. It will be a 
combination of the 
results of \cite{E3} with the estimates of section \ref{estimates}. As in section \ref{estimates}, we shall
 restrict 
ourselves to the case $s=3$. The proof for general $s$ is essentially the same.

We first consider for fixed $k$ the initial data $u_{0k}$, $\rho_{0k}$ and assume that $k$ is large 
enough so that $\rho_{0k}$ is near $1$, and so that for each $x\in \Om$, $|u_{0k}(x)|^2 < p^\prime(\rho_{0k}(x))$. This 
can be done because of assumption
2a) 
of Theorem \ref{first_big_theo} which implies that the $C^0$-norm of $u_{0k}$ is bounded 
independent of $k$ and because of 2c) of that theorem which says that
\begin{gather}
 \n \log \rho_{0k} \n_4 \leq \frac{a_2}{k} .
\nonumber
\end{gather}
Then from \cite{E3} we know that there exists an interval $[0,T]$ and $H^3$ functions $u$, and $\rho$ 
defined on $[0,t]\times \Om$ which satisfy the compressible motion problem with initial data $u_{0k}$, $\rho_{0k}$. Thus we 
need only show that $\rho(t) \in H^4$ and $\de u(t) \in H^3$.

In order to show this we must first approximate $u_{0k}$, $\rho_{0k}$ by $H^4$ functions which also satisfy the 
compatibility conditions. That is we want a sequence of $H^4$ functions $(u_n,\rho_n)$ such that\\

1) $u_n \rar u_{0k}$ in $H^3$ and $\de u_n \rar \de u_{0k}$ in $H^3$ as $n\rar \infty$. \\

2) $\rho_n \rar \rho_{0k}$ in $H^4$ as $n\rar \infty$. \\

3) For each $n$, the pair $(u_n,\rho_n)$ satisfies the compatibility conditions up to order $3$. \\

\begin{rema}
We notice that here we require $(u_n,\rho_n)$ to satisfy the compatibility conditions up to order $3$ 
for a solution in $H^4$, in contrast with the statement of Theorem \ref{first_big_theo} where compatibility conditions up to order
$s$ are required for a solution with velocity in $H^s$. The reason for this is as follows.

As seen, the compatibility conditions involve the time-derivative of the solution at time zero.
In general, for a solution $u$ and $\rho$ in $H^s$ 
with initial data in $H^s$
(i.e., we are now talking about the usual case
where $u$ and $\rho$ are equally differentiable and no extra differentiability 
of $\rho_0$ nor $\dive (u_0 ) \in H^s$ are assumed)
it is natural to expect compatibility conditions
up to order $s-1$. In fact, as mentioned in the paper, the $k^{\text{th}}$ order compatibility 
condition involves $k$ derivatives of $\rho_0$ and $k-1$ derivatives of $u_0$
and $\dive ( u_0 )$. Therefore, for $u_0 \in H^s$, without assuming 
$\dive (u_0 ) \in H^s$, we can only expect to have $s-1$ compatibility conditions,
otherwise more than $s$ derivatives of $u_0$ would be involved.
This is what is assumed, for instance, in \cite{E4} and \cite{S}.
Our goal here with the sequence $(u_n,\rho_n)$ is simply to construct
an $H^4$ solution (for the  $H^4$ initial data $(u_n,\rho_n)$), which then, in light of the estimates
on $f$, can be used to show the extra regularity of the original $\rho$ solution.
Therefore, in view of what was just above, this initial data $(u_n(0), \rho_n(0))$, being in $H^4$,
 has to satisfy the compatibility 
conditions only up to order $3$.

Notice that there are further constraints on the initial data $(u_n(0), \rho_n(0))$, namely,
$\dive ( u_n(0) )$ is also in $H^3$. But this is required in order for this initial data
to be close in $H^3$ to our prepared initial data $(u_0,\rho_0)$, and is not required
for the general fact that a $H^4$ initial data yields an $H^4$ solution. Since this general
fact only requires, as explained, compatibility conditions up to order $3$, that is all we require
for $(u_n(0), \rho_n(0))$.
\end{rema}

We now proceed to construct such a sequence.

First we decompose $u_{0k}$ into its divergence free and gradient parts:
\begin{gather}
 u_{0k} = w_{0k} + \nabla g_{0k}.
\nonumber
\end{gather}
We let $\tilde{w}_{n}$ be a sequence of $H^4$ divergence free vector fields which converge to $w_{0k}$ in $H^3(\Om,\RR^n)$.
Such a sequence is easily found; one simply approximates $w_{0k}$ by a sequence of $H^4$ vector 
fields and lets the $\{ \tilde{w}_n \}$ be $P$ applied to this sequence. Now 
let $\tilde{u}_n = \tilde{w}_n + \nabla g_{0k}$ and $\tilde{\rho}_n = \rho_{0k}$. Then $\de \tilde{u}_n = -\Delta g_{0k} = \de u_{0k}$, 
and $\Delta g_{0k} \in H^3$ implies $\nabla g_{0k} \in H^4$, so clearly the sequence $(\tilde{u}_n,\tilde{\rho}_n)$ satisfies the 
conditions 1) and 2). We now show that it can be modified so that it will satisfy 3) as well.

Of course $\tilde{u}_n$ satisfies the zeroth compatibility condition because
\begin{gather}
 \langle \tilde{u}_n, \nu \rangle = \langle \nabla g_{0k}, \nu \rangle = 0 \text{ on } \partial \Om
\nonumber
\end{gather}
We let $\varphi_i = \varphi_i(u,\rho)$ be the function defining the i$^{\text{ih}}$ compatibility 
condition; that is we define $\varphi_i$ so that the i$^{\text{th}}$ condition for initial data $(u,\rho)$ is 
\begin{gather}
 \varphi_i(u,\rho) = 0
\nonumber
\end{gather}
Then from section \ref{func_space} we know that (letting $f=\log \rho$)
\begin{gather}
 \varphi_1(u,\rho) = \langle \nabla_u u + c^2(\rho) \nabla f, \nu \rangle
\label{phi_1}
\end{gather}
To get formulas for $\varphi_i$, $i > 1$ is it convenient to use the quadratic form
\begin{gather}
 S_2(v_1,v_2) = -\langle v_1, \nabla_{v_2} \nu \rangle
\nonumber
\end{gather}
which we defined on section \ref{func_space}.

Then (\ref{second_ord_comp_cond}) tells us that 
\begin{gather}
 \varphi_2(u,\rho) = -2S_2(u, \nabla_u u + c^2 \nabla f) + p^{\prime\prime}(\rho) \rho (-\nabla_u f + \de u) \nabla_\nu f +
c^2 \nabla_\nu (-\nabla_u f + \de u).
\label{phi_2}
\end{gather}
To compute $\varphi_3$ we take a time derivative of (\ref{another_der_first_ord}) to get
\begin{gather}
 \varphi_3(u,\rho) = -2S_2(\partial_t u, \partial_t u) - 2S_2(\partial_{tt} u, u ) 
+ \partial_t^2(c^2) \nabla_\nu f + 2\partial_t(c^2) \nabla_\nu \partial_t f + c^2 \nabla_\nu \partial_t^2 f.
\label{phi_3}
\end{gather}
From our usual computations we get
\begin{gather}
 \partial_t u = -(\nabla_u u + c^2 \nabla f) ,
\nonumber \\
\partial_t f = -\nabla_u f + \de u ,
\nonumber \\
\partial_t (c^2) = p^{\prime\prime}(\rho) \rho \partial_t f ,
\nonumber \\
\partial_t^2 u = \nabla_{\partial_t u } u + \nabla_u \partial_t u - \partial_t(c^2)\nabla f - c^2 \nabla \partial_t f,
\nonumber \\
\partial_t^2 f = - \nabla_{\partial_t u} f - \nabla_u \partial_t f + \de \partial_t u =
- \nabla_{\partial_t u} f - \nabla_u \partial_t f - \de \nabla_u u - \de c^2 \nabla f,
\nonumber
\end{gather}
and substituting these expressions for the time derivatives in (\ref{phi_3}) gives an expression for $\varphi_3(u,\rho)$. 
Other $\varphi$'s are computed in the same way, each one involving one more derivative of (\ref{another_der_first_ord}).

Now let $u \in H^3$ (with $\langle u , \nu \rangle = 0$), $\rho \in H^4$ and assume $\de u \in H^3$. Then since $\varphi_i$ 
involves $i$ derivatives of $\rho$ but only $i-1$ derivatives of $u$ and $\de u$, we find
\begin{gather}
\varphi_i(u,\rho) \in H^{3\half\, - \, i}(\partial \Om, \RR)
\nonumber
\end{gather}
Letting $u=w+\nabla g$ be the usual decomposition we can think of each $\varphi_i$ as a 
function of $w$, $g$, and $f$. In this way we get
\begin{gather}
 \varphi_i: P(H^3(\Om,\RR^n))\times \nabla (H_\nu^5(\Om, \RR)) \times H^4(\Om, \RR) \rar H^{3\half \,  - \, i}(\partial \Om, \RR)
~~(i=1,2,3)
\nonumber
\end{gather}
where $H_\nu^5(\Om,\RR) = \{ g \in H^5 ~|~ \nabla_\nu g = 0 \text{ on } \Om \}$. We shall call the domain of this map $X^{3,4,4}$. We 
use $H_\nu^5(\Om,\RR)$ to insure that $\langle u, \nu \rangle = 0$ on $\partial \Om$ and that $\de u = -\Delta g$ is in $H^3$. 
Then if we let $\widetilde{\Phi} = (\varphi_1,\varphi_2,\varphi_3)$ we get a smooth map
\begin{gather}
\widetilde{\Phi}: X^{3,4,4} \rar H^{2\half}(\partial \Om, \RR) \times H^{1\half}(\partial \Om, \RR) \times
H^{\frac{1}{2}}(\partial \Om, \RR)
\label{tilde_Phi}
\end{gather}
We let $Y$ denote the range of this map. Of course $\widetilde{\Phi}(u_{0k},\rho_{0k})=0$ so 
$\lim_{n\rar \infty} \widetilde{\Phi}(\tilde{u}_n,\tilde{\rho}_n) = 0$.

Now we show the existence of $(u_n,\rho_n)$ in three steps. First we 
let $$X^{4,4,4} = \{ (w,\nabla g, f) \in X^{3,4,4,} ~|~ w \in H^4 \}$$ and let $\Phi$ be 
the restriction of $\widetilde{\Phi}$ to $X^{4,4,4}$. Second we show that the derivative of the 
map $\Phi: X^{4,4,4} \rar Y$ at each point $(\tilde{u}_n,\tilde{\rho}_n)$ is a surjection. Third we 
use an implicit function theorem type of argument to show that for large $n$ there 
exists $(u_n,\rho_n)$ near $(\tilde{u}_n,\tilde{\rho}_n)$ such that $\Phi(u_n,\rho_n)=0$. This 
sequence $\{(u_n,\rho_n)\}$ will satisfy the required conditions 1), 2) and 3).

We proceed with the second step, computing the derivative of $\Phi$. From 
(\ref{phi_1}) we see that (letting $\tilde{f} = \log\tilde{\rho}$).
\begin{gather}
D_{(\tilde{u},\tilde{\rho})} \varphi_1(w,\nabla g, f) = -2S_2(\tilde{u}, w + \nabla g) 
+ p^{\prime\prime}(\tilde{\rho}) f \nabla_\nu \tilde{f} + c^2(\tilde{\rho}) \nabla_\nu f
\nonumber
\end{gather}
where $D_{(\tilde{u},\tilde{\rho})} \varphi_1(w,\nabla g, f)$ means the 
derivative of $\varphi_1$ at $(\tilde{u},\tilde{\rho})$ in direction $(w,\nabla g, f)$.

Assuming that $\n \tilde{f}\n_4 \leq \frac{K}{k}$ we find that 
\begin{gather}
 D_{(\tilde{u},\tilde{\rho})} \varphi_1(w,\nabla g, f)=A_{11}(w,\nabla g, f) + c^2(\tilde{\rho})\nabla_\nu f
\label{D_phi_1}
\end{gather}
where $A_{11}: X^{4,4,4} \rar H^{2\frac{1}{2}}(\partial \Om, \RR)$ is a bounded 
linear operator depending on $(\tilde{u},\tilde{\rho})$ but with operator norm bounded by some $K$ independent of $k$.

Similarly from (\ref{phi_2}) we find
\begin{gather}
D_{(\tilde{u},\tilde{\rho})} \varphi_2(w,\nabla g, f) = A_{21}(w,\nabla g) + k A_{22} f + c^2(\tilde{\rho}) \nabla_\nu \Delta g,
\label{D_phi_2}
\end{gather}
and from (\ref{phi_3})
\begin{gather}
 D_{(\tilde{u},\tilde{\rho})} \varphi_3(w,\nabla g, f) = A_{31}(w,\nabla g) + kA_{32}(\nabla g) + k A_{33} f - c^2(\tilde{\rho}) 
\nabla_\nu Lf
\label{D_phi_3}
\end{gather}
where $L = -\de c^2(\tilde{\rho}) \nabla$, and $A_{21},~A_{22}: X^{4,4,4} \rar H^{1\half}(\partial \Om,\RR)$
are bounded operators with bounds independent of $k$, as are
$A_{31},~A_{32},~A_{33}: X^{4,4,4} \rar H^{\frac{1}{2}}(\partial \Om,\RR)$. Furthermore since the $\varphi_i$'s involve no more 
than two derivatives of $\tilde{u}$ and $\de \tilde{u}$, the bounds on the operators $A_{ij}$ are 
uniform for all $(\tilde{u},\tilde{\rho})$ which are near $(u_{0k},\rho_{0k})$ in $X^{3,4,4}$. 
Formulas (\ref{D_phi_1})-(\ref{D_phi_3}) give an expression for the linear operator $D_{(\tilde{u},\tilde{\rho})} \Phi: X^{4,4,4} \rar Y$. 
From the above we know that it is uniformly bounded for all $(\tilde{u},\tilde{\rho})$ in an $X^{3,4,4}$-neighborhood 
of $(u_{0k},\rho_{0k})$. 

Computing in the same way we get a formula for $D^2_{(\tilde{u},\tilde{\rho})}\Phi$, the second 
derivative of $\Phi$ at $(\tilde{u},\tilde{\rho})$ which is a bilinear map from $X^{4,4,4} \times X^{4,4,4} \rar Y$. Using the 
same reasoning as above, we find that $D^2_{(\tilde{u},\tilde{\rho})}\Phi$ is bounded as a bilinear  operator and again 
we see that the bound is uniform for $(\tilde{u},\tilde{\rho})$ near $(u_{0k},\rho_{0k})$ in $X^{3,4,4}$.

Now we shall use (\ref{D_phi_1})-(\ref{D_phi_3}) to find a right inverse to $D_{(\tilde{u},\tilde{\rho})} \Phi$. This 
amounts to finding a solution $(w,\nabla g, \nabla f)$ to the linear equations
\begin{gather}
 A_{11}(w,\nabla g, f) + c^2\nabla_\nu f = h_1 \label{sys_1} \\
A_{21}(w,\nabla g) + k A_{22} f + c^2 \nabla_\nu \Delta g = h_2 \label{sys_2} \\
A_{31}(w,\nabla g) + k A_{32}(\nabla g)+kA_{33} f - c^2 \nabla_\nu L f = h_3 \label{sys_3}
\end{gather}
where $h_i \in H^{3\half \, - \, i}(\partial \Om,\RR)$, $i=1,2,3$.

When $k$ is large this system can be solved by a Neumann series technique. To begin,
 we let $w=0$ so $u=\nabla g$. Then we let
\begin{gather}
 f_1 = G_L(\frac{1}{c^2} h_1) - L^{-1}G_L(\frac{1}{c^2} h_3) 
\nonumber \\
g_1 = \Delta^{-1} G(\frac{1}{c^2}h_2 - \frac{k}{c^2}A_{22}f_1 )
\nonumber
\end{gather}
where $G$, $G_L$, $\Delta^{-1}$ and $L^{-1}$ are as in section \ref{estimates}.

Then
\begin{gather}
 \frac{1}{c^2}A_{11}(0,\nabla g_1,f_1) + \nabla_\nu f_1 = \frac{1}{c^2}h_1 + \frac{1}{c^2} A_{11}(0,\nabla g_1, f_1) ,
\nonumber \\
\frac{1}{c^2}A_{21}(0,\nabla g_1) + \frac{k}{c^2} A_{22}f_1 + \nabla_\nu \Delta g_1 
= \frac{1}{c^2} h_2 + \frac{1}{c^2}A_{21}(0,\nabla g_1) ,
\nonumber \\
\frac{1}{c^2}A_{31}(0,\nabla g_1) + \frac{k}{c^2} A_{32}(\nabla g_1) + \frac{k}{c^2}A_{33}f_1 - \nabla_\nu Lf_1 =
\frac{1}{c^2}(h_3 + A_{31}(0,\nabla g_1) + k A_{32}(\nabla g_1) + k A_{33} f_1 ) \nonumber \\
= \frac{1}{c^2}(h_3 + R_1),
\nonumber
\end{gather}
where $R_1$ is defined by the last equality.
Therefore let
\begin{gather}
 f_2 = f_1 + G_L\frac{1}{c^2}A_{11}(0,\nabla g_1,f_1) + L^{-1}G_L \frac{1}{c^2} (R_1)
\nonumber
\end{gather}
and let
\begin{gather}
g_2 = g_1 + \Delta^{-1} G(\frac{1}{c^2} A_{21}(0,\nabla g_1) - \frac{k}{c^2} A_{22}(f_2 - f_1) ).
\nonumber
\end{gather}
continuing we get a sequence $\{ g_\ell, f_\ell \}$ and it is clear that 
\begin{gather}
 \n g_\ell - g_{\ell - 1} \n_5 + \n f_\ell - f_{\ell -1 } \n_4 \leq \frac{K}{k^\ell} .
\nonumber
\end{gather}
Therefore for $k$ large we get the limits $g_\ell \rar g$ and $f_\ell \rar f$ and $(0,\nabla g, f)$ satisfies the system
(\ref{sys_1})-(\ref{sys_3}). Furthermore it is clear that
\begin{gather}
 \n f \n_4 + \n \nabla g \n_4 \leq K( \n h_1 \n_{\partial,2\half} + 
 \n h_2 \n_{\partial,1\half} +
\n h_3 \n_{\partial,\frac{1}{2}} )
\nonumber
\end{gather}
where $K$ depends only on the norms of the operators $A_{ij}$. Therefore we have found a right 
inverse to $D_{(\tilde{u},\tilde{\rho})}\Phi$ which is uniformly bounded for $(\tilde{u},\tilde{\rho})$ near $(u_{k0},\rho_{k0})$.

Now we proceed to our third step, the implicit function theorem argument, which will show the existence of the 
sequence $(u_n,\rho_n)$. Our proof follows from the proof of the implicit function theorem. Let $A_n: Y \rar X^{4,4,4}$ be the 
right inverse to $D_{(\tilde{u}_n,\tilde{\rho}_n)}\Phi$ which we constructed above, and let
\begin{gather}
\psi_n: Y \rar Y
\nonumber
\end{gather}
be defined by
\begin{gather}
\psi_n(y) = \Phi(A_n(y) + (\tilde{u}_n,\tilde{\rho}_n) ) - \Phi(\tilde{u}_n,\tilde{\rho}_n) .
\nonumber
\end{gather}
Clearly $\psi_n$ takes zero to zero and its derivative at zero is the identity map from $Y$ to $Y$. Also,
since $D^2\Phi$, and $A_n$ are bounded so is $D^2\psi_n$. Assuming $D^2\psi_n: Y \times Y \rar Y$ has bound $K$, a standard 
estimate tells us that for $\ep < \frac{1}{4K}$, $I-\psi_n: B_{2\ep}(0) \rar B_{\ep}(0)$ is a contraction with Lipshitz 
constant $\frac{1}{2}.$ ($B_r(0)$ is the ball about zero of radius $r$ in $Y.$)

Now given $z \in B_\ep(0)$, let $\theta: Y \rar Y$ be defined by 
\begin{gather}
\theta(y) = y - \psi_n(y) + z = (I-\psi_n)(y) + z.
\nonumber
\end{gather}
Then $\theta: B_{2\ep}(0) \rar B_{2\ep}(0)$ is also a contraction and therefore has a unique fixed point, call it $y_n$.

Given any positive $\ep$ which is less than $\frac{1}{4K}$, we can find $n$ so that $\n \Phi(\tilde{u}_n, \tilde{\rho}_n) \n_Y < \ep$. 
Then let $z=\Phi(\tilde{u}_n,\tilde{\rho}_n)$, so that $y_n$ satisfies:
\begin{gather}
 y_n = y_n - \psi_n(y_n) + \Phi(\tilde{u}_n,\tilde{\rho}_n) = y_n - \Phi( A_n(y_n) + (\tilde{u}_n,\tilde{\rho}_n) ) .
\nonumber
\end{gather}
Thus $\Phi(A_n(y_n) + (\tilde{u}_n,\tilde{\rho}_n) ) = 0$ so $(u_n,\rho_n) = A_n(y_n) + (\tilde{u}_n,\tilde{\rho}_n)$ 
satisfies the compatibility conditions. 
Since $\n y_n \n < 2\ep$, $\n (u_n, \rho_n) - (\tilde{u}_n, \tilde{\rho}_n )\n_{X^{4,4,4}} < 2\ep$. 
Therefore $(u_n,\rho_n)$ converges to $(u_{0k},\rho_{0k})$ in $X^{3,4,4}$, and hence is the desired sequence.

Now that we have the sequence $(u_n,\rho_n)$, Theorem \ref{first_big_theo} follows from the estimates of section \ref{estimates}. 
From \cite{E3} we know that for each pair $(u_n,\rho_n)$ we get curves $(u_n(t),\rho_n(t))$ in $H^4$ defined on some 
interval $[0,T_n)$; we take $T_n$ to be as large as possible.

However since $(u_n,\rho_n) \rar (u_{0k},\rho_{0k})$ in $H^3$ we know that if $n$ is 
large, $\{ (u_n(t), \rho_n(t)) \}$ and $(u(t), \rho(t) )$ all exist as $H^3$ functions on $[0,T]$, and $(u_n(t),\rho_n(t))$ 
converges to $(u(t),\rho(t))$ in $H^3$. Therefore there exist constants $a_2$ and $a_3$ as required in section \ref{estimates} so that 
(\ref{hypothesis_at_0}) and (\ref{extra_assump_u}) hold for the functions $u_n(t)$ and $f_n(t) = \log\rho_n(t)$, and since
$(u_n,\rho_n) \rar (u,\rho)$ in $H^3$ one can choose one set of constants for all $n$. Hence we find a constant $a_5$ such that 
(\ref{ineq_a_5_f})-(\ref{ineq_a_5_f_dddot}) hold for all $f_n$ also. But it is well known (see \cite{K2} for example) that 
if $h_n \rar h$ in $H^s$ and
$\{h_n\}$ is a bounded sequence in $H^{s+1}$, then in fact $h \in H^{s+1}$. Therefore $f_n \rar f = \log \rho$ in $H^3$ implies
$f \in H^4$, and $\dot{f}_n \rar \dot{f}$ implies $\dot{f} \in H^3$. Hence we have $\rho \in H^4$ and $\dot{f} = \de u \in H^3$. 
This completes the proof of theorem \ref{first_big_theo}.

To prove Theorem \ref{second_big_theo} we first state an equivalent theorem and prove the latter using the 
estimates of section \ref{estimates} together with the proof of theorem 5.5 of \cite{E2}.

In order to state the equivalent theorem we will use the map $\zeta(t): \Om \rar \Om$ to describe the compressible 
fluid position at time $t$, and $\eta(t): \Om \rar \Om$ to describe the incompressible position. We think 
of both $\zeta$ and $\eta$ as curves in $\mathscr{D}^s$, with $\zeta(0) = \eta(0) =\operatorname{id}$, the identity 
diffeomorphism. If we use `` $\dot{}$ ''  to denote time derivatives of $\eta$ and $\zeta$ we have (as in section \ref{intro})
\begin{gather}
 \dot{\eta}(t)(x) = v(t,\eta(t)(x)) = v(t)(\eta(t)(x)) 
\label{eta_dot_v}
\end{gather}
so
\begin{gather}
 v(t)(x) = \dot{\eta}(t)\circ(\eta(t))^{-1}(x)
\nonumber
\end{gather}
and
\begin{gather}
 \dot{\zeta}(t)(x) = u(t,\zeta(t)(x)) = u(t)(\zeta(t)(x)) 
\label{zeta_dot_u}
\end{gather}
so
\begin{gather}
 u(t)(x) = \dot{\zeta}(t)\circ(\zeta(t))^{-1}(x)
\label{u_zeta_dot_compose_zeta_inverse}
\end{gather}
where $(\eta(t))^{-1}$ and $(\zeta(t))^{-1}$ are the inverses of the diffeomorphisms $\eta(t),~\zeta(t): \Om \rar \Om$ and
$\circ$ means composition. Therefore
\begin{gather}
 \ddot{\eta}(t)(x) = \partial_t v(t,\eta(t)(x)) + \nabla_v v(t,\eta(t)(x)) = \dot{v}(t,\eta(t)(x)) = \dot{v}(t)(\eta(t)(x))
\nonumber
\end{gather}
and similarly
\begin{gather}
\ddot{\zeta}(t)(x) = \partial_t u(t,\zeta(t)(x)) + \nabla_u u(t,\zeta(t)(x)) = \dot{u}(t,\zeta(t)(x)) = \dot{u}(t)(\zeta(t)(x))
\nonumber
\end{gather}
Hence (\ref{eq_v_no_p}) is equivalent to the equation
\begin{gather}
 \ddot{\eta}(t)(x) = (Q(\nabla_v v))(\eta(t)(x))
\label{eq_eta}
\end{gather}
and (\ref{eq_comp_u}) is equivalent to 
\begin{gather}
 \ddot{\zeta}(t)(x) = -(\frac{1}{\rho} \nabla p)(\zeta(t)(x)) = - (c^2 \nabla f)(\zeta(t)(x))
\label{eq_zeta}
\end{gather}
Since (\ref{eq_zeta}) involves $\rho$ or $f$ we must find a relation between one of them and $\zeta$. If we 
denote by $J(\zeta(t))$ the Jacobian determinant of $\zeta(t)$, a direct computation gives
\begin{gather}  
\frac{\partial}{\partial t } (J(\zeta(t))) (t,x) = J(\zeta(t))(t,x)(\operatorname{div}(u)(t,\zeta(t)(x))) .
\nonumber
\end{gather}
Thus we find that if we let
\begin{gather}
 h(t,x) = - \log(J(\zeta(t)))((\zeta(t)^{-1})(x))
\nonumber
\end{gather}
then $h$ satisfies
\begin{gather}
 \dot{h} = -\de u .
\nonumber
\end{gather}
Therefore since $f$ satisfies (\ref{f_dot_div_u}) with initial 
condition $f(0,x)=f_{0k}(x)$ and since $h(0,x) = \log (J(\operatorname{id})) = 0$ we must have
\begin{gather}
 f(t,x) = f_{0k}(x) + h(t,x)
\label{f_terms_h}
\end{gather}
Combining (\ref{eq_zeta}) and (\ref{f_terms_h}) we find that $(u,f)$ satisfy (\ref{eq_comp_u}) and 
(\ref{f_dot_div_u}) with initial condition $f(0,x)=f_{0k}(x)$ if and only if $\zeta$ satisfies
\begin{gather}
 \ddot{\zeta}(t)(x) = -(c^2\nabla (f_{0k}+h))(t,\zeta(t)(x))
\label{eq_zeta_f_h_c}
\end{gather}
where $h$ is defined in terms of $\zeta$ as above and $c^2$ depends on $\rho$ or $f$ which in turn depends on $h$. 
Also we note that if $\zeta$ and $\eta$ are curves in $\mathscr{D}^s$, and $u$ and $v$ are defined in terms of $\zeta$ 
and $\eta$ using
(\ref{eta_dot_v})-(\ref{zeta_dot_u}), then we find that $u$ and $v$ must satisfy the boundary conditions
\begin{gather}
 \langle u, \nu \rangle = \langle v, \nu \rangle = 0 \text{ on } \partial \Om.
\nonumber
\end{gather}
Now we state the equivalent theorem

\begin{theo}
Assume that $u_{0k}$, $\rho_{0k}$, $p_k$, and $v_0$ are as in Theorem \ref{second_big_theo}. Then there exist an interval
$[0,T]$ and a unique smooth curve $\eta:[0,T] \rar \mathscr{D}^s$ satisfying (\ref{eq_eta}) such 
that $\eta(0) = \operatorname{id}$, the identity diffeomorphism, and $\dot{\eta}(0) = v_0$. Also for each $k$, there 
exists a unique $C^2$ curve
$\zeta_k(t)$ in $\mathscr{D}^s$ defined on an interval $[0,T(k)]$ satisfying (\ref{eq_zeta_f_h_c}), and such that
$\zeta_k(0) = \operatorname{id}$, $\dot{\zeta}_k(0) = u_{0k}$.

Furthermore if $T(k)$ is maximal, then $T(k) > T$ for large $k$, and as $k\rar \infty$, $\zeta_k(t) \rar \eta(t)$ as a $C^1$ 
curve in $\mathscr{D}^s$. In addition
\begin{gather}
 J(\zeta_k(t))\circ \zeta_k(t)^{-1} \rar 1 \text{ in } H^{s+1}
\nonumber
\end{gather}
\label{equiv_second_big_theo}
\end{theo}

It is clear that Theorem \ref{equiv_second_big_theo} implies \ref{second_big_theo}; one simply 
let $u_k(t) = \dot{\zeta}_k(t)\circ(\zeta(t))^{-1}$, $v(t)=\eta(t)\circ(\eta(t))^{-1}$. Then $\zeta_k \rar \eta$ in $C^1$ implie that 
$u_k \rar v$ in $C^0$. Also since
\begin{gather}
 \rho_k(t) = \frac{\rho_{0k}}{J(\zeta_k(t))\circ(\zeta_k(t)^{-1})}
\nonumber
\end{gather}
we find $\rho_k \rar 1$ in $H^{s+1}$.

We now proceed with the proof of Theorem \ref{equiv_second_big_theo}. It is essentially the same as the proof of 
Theorem 5.5 of \cite{E2} so we shall go through it rather briefly expecting the reader to refer to \cite{E2} for details.

First we define $Z: \mathscr{D}^3 \times H^3(\Om,\RR^n) \rar H^3(\Om,\RR^n)$ by 
\begin{gather}
 Z(\xi,\al) = (Q(\nabla_{\al\circ\xi^{-1}}P(\al\circ\xi^{-1})))\circ \xi
\nonumber
\end{gather}
as in \cite{E2} section 13. Then since $v=\dot{\eta}\circ\eta^{-1}$ and $P(v) = v$ we find that (\ref{eq_eta}) can be written 
\begin{gather}
 \ddot{\eta} = Z(\eta,\dot{\eta}) .
\nonumber
\end{gather}
Solving this with initial condition $v_0$ is of course the same as solving
\begin{gather}
 \dot{\eta}(t) = v_0 + \int_0^t Z(\eta(s),\dot{\eta}(s))ds .
\label{int_eq_eta}
\end{gather}
We will now find an equation similar to (\ref{int_eq_eta}) for $\zeta = \zeta_k$. Note that
\begin{gather}
\dot{\zeta} = u \circ \zeta = w \circ\zeta + \nabla g \circ \zeta
\nonumber
\end{gather}
where $u = w + \nabla g$ is the decomposition of $u$ into divergence free and gradient
parts.

But from (\ref{eq_w}) we find that 
\begin{gather}
 \frac{\partial w}{\partial t} + \nabla_u w = Q(\nabla_u w) - P(\nabla_w \nabla g)
\nonumber
\end{gather}
and therefore
\begin{gather}
 (w\circ\zeta)\dot{} = (Q(\nabla_u w) - P(\nabla_w \nabla g))\circ \zeta .
\nonumber
\end{gather}
Define $R: \mathscr{D}^3 \times H^3(\Om,\RR^n) \rar H^3(\Om,\RR^n)$ by
\begin{gather}
R(\xi,\al) = (P(\nabla_{\al\circ\xi^{-1}}Q( \al\circ\xi^{-1})))\circ \xi
\nonumber
\end{gather}
as in \cite{E2} section 13. Then
\begin{gather}
 (w\circ\zeta)\dot{} = Z(\zeta,\dot{\zeta}) - R(\zeta,\dot{\zeta})
\nonumber
\end{gather}
so
\begin{gather}
 \dot{\zeta} = P(u_{0k}) + \int_0^t \big ( Z(\zeta(s),\dot{\zeta}(s)) - R(\zeta(s),\dot{\zeta}(s) \big ) ds 
+(\nabla g(t))\circ \zeta(t) .
\label{int_eq_zeta}
\end{gather}
We shall show that this equation is close to (\ref{int_eq_eta}), and from this the theorem will follow.

We let $T_1(k) = \min\{ T, T(k) \}$ and assume that $(\zeta_k,\dot{\zeta}_k): [0,T_1(k)) \rar \mathscr{D}^3 \times H^3(\Om,\RR^n)$ is 
within $\ep$ of the curve $(\eta,\dot{\eta}): [0,T_1(k)) \rar \mathscr{D}^3 \times H^3(\Om,\RR^n)$ .
This gives a constant $a_3$ such that (\ref{extra_assump_u}) holds. The hypothesis of Theorem \ref{first_big_theo} implies 
that (\ref{hypothesis_at_0}) holds so if $k$ is large, the estimates of section \ref{estimates} are valid and we find $a_5$ such that
\begin{gather}
 \n f_k(t) \n_4 \leq \frac{a_5}{k} \text{ and } \n \dot{f}_k (t) \n_3 \leq \frac{a_5}{\sqrt{k}}
\nonumber
\end{gather}
But $\nabla g = \nabla \Delta^{-1} \dot{f}$ and $R(\zeta,\dot{\zeta}) = P(\nabla_w \nabla g)\circ \zeta$ so we get a constant
$a_6$ depending on $\ep$, such that
\begin{gather}
 \n \nabla g \circ \zeta \n_3 \leq \frac{a_6}{\sqrt{k}}
\nonumber
\end{gather}
and
\begin{gather}
 \n R(\zeta,\dot{\zeta} )\n_3 \leq \frac{a_6}{\sqrt{k}}
\nonumber
\end{gather}
Also as $k \rar \infty$, $u_{0k} \rar v_0$, so $P(u_{0k}) \rar v_0$ also. Thus we find
\begin{align}
\begin{split}
\n \dot{\eta}(t) - \dot{\zeta}_k (t) \n_3 \leq & 
\n P(u_{0k}) - v_0 \n_3 + (1+t)\frac{a_6}{\sqrt{k}}
\\
& + 
\int_0^t \n Z(\eta(s),\dot{\eta}(s)) - Z(\zeta_k(s),\dot{\zeta}_k(s)) \n_3 ds .
\end{split}
\label{difference_eta_dot_zeta_dot_k_iteration}
\end{align}
Iterating this inequality we find that for large $k$ we can find $\ep_2 < \ep$ such that if $t \in [0,T_1(k)]$
\begin{gather}
 \n (\zeta_k(t),\dot{\zeta}_k(t)) - (\eta(t),\dot{\eta}(t) ) \n_3 \leq \ep_2 .
\nonumber
\end{gather}
Furthermore, since $\n f \n_4 \leq \frac{a_5}{k}$ we can pick $a_6$ so that
$\n \rho_k -1 \n_4 \leq \frac{a_6}{k}$ also. Therefore we find that since $T(k)$ is maximal $T(k) > T_1$ so $T = T_1(k)$. 
Also as $k \rar \infty$, $\ep_2$ can be taken arbitrarily small. Hence $\zeta_k \rar \eta$ as a $C^1$ curve in $\mathscr{D}^3$ and
$\rho_k \rar 1$ in $H^4$, so $J(\zeta_k)\circ \zeta_k^{-1} \rar 1$ in $H^4$ also. $\zeta$ is a $C^2$ curve in $\mathscr{D}^3$ 
because $f \in H^4$ and by (\ref{eq_zeta}), $\ddot{\zeta} = -(c^2\nabla f) \circ \zeta \in H^3(\Om,\RR^n)$. This concludes the proof.

\begin{rema} It is most curious that $\rho(t)$ or $f(t)$ are in $H^4$ while $u(t)$ or $\zeta(t)$ are only in $H^3$. In the 
first place $\rho = J(\zeta)\circ \zeta^{-1}$ involves first derivatives of $\zeta$ so one would expect $\rho$ to be less rather 
than more differentiable. Secondly, $(u,\rho)$ satisfies (\ref{eq_comp_u})-(\ref{eq_comp_rho}) which is a quasi-linear symmetric 
hyperbolic system. The usual methods of solution of such systems would produce equally differentiable $u$ and $\rho$.
\end{rema}

\section{Differentiable Dependence on Initial Conditions \label{diff_dep} }

In this section we shall show that the solution of the compressible fluid problem depends differentiably on the initial 
conditions $u_{0k}$, $\rho_{0k}$. This solution, $u_k(t)$, $\rho_k(t)$ is given by Theorem \ref{first_big_theo}.

As we explained in section \ref{proofs}, the motion described by $u_k(t)$, $\rho_k(t)$ (which we sometimes call $u(t)$,$\rho(t)$) 
can be equivalently described by $\zeta (t):\Om \rar \Om$ where $\zeta(t)$ is a differentiable curve in $\mathscr{D}^s$, and
\begin{gather}
 \dot{\zeta}(t) = u(t)(\zeta(t)(x))
\label{eq_zeta_simp}
\end{gather}
We shall show that for fixed $t$, $\zeta(t)$ and $\dot{\zeta}(t)$ depend differentiably on $u_{0k}$, and $\rho_{0k}$.
First we note that $u_k(t)$, $\rho_k(t)$ are given by theorem \ref{first_big_theo} only 
when $u_{0k} \in H^s$, $\de u_{0k} \in H^s$, and $\rho_{0k} \in H^{s+1}$, and when $u_{0k}$ satisfies some inequalities and the 
compatibility conditions.

Thus (taking $s=3$), we let 
\begin{gather}
 \mathscr{I} = \big \{ (u,\rho) \in X^{3,4,4} ~ | ~ (u,\rho) 
\text{ satisfy the inequalities of Theorem } \ref{first_big_theo} \text{, and  } \widetilde{\Phi}(u,\rho) = 0 \big \}
\nonumber
\end{gather}
where $\widetilde{\Phi}: X^{3,4,4} \rar Y $ is defined in (\ref{tilde_Phi}). The set $\mathscr{I}$ is then the 
set of possible initial conditions for Theorem \ref{first_big_theo}.

Given $t$ we define $\psi_t(u,\rho) = (\zeta(t),\dot{\zeta}(t))$ where $\zeta(t)$ is the compressible fluid motion with 
initial data $(u,\rho)$. Then $\psi_t: U_t \rar \mathscr{D}^3 \times H^3(\Om, \RR^n)$ where $U_t$ is the set of initial 
conditions for which the fluid motion is defined for at least time $t$. From Theorem \ref{first_big_theo} it 
follows that $U_t$ is open in $\mathscr{I}$.

The goal of this section is to show the following:

\begin{theo}
$U_t$ is a submanifold of $X^{3,4,4}$ and $\psi_t: U_t \rar \mathscr{D}^3 \times H^3(\Om,\RR^n)$ is $C^1$. 
\label{submanifold_theo}
\end{theo}
\begin{proof}
Since $U_t$ is open in $\mathscr{I}$, to show that $U_t$ is a submanifold is suffices to check that $\mathscr{I}$ is a 
submanifold. This we proceed to do.

First let 
\begin{gather}
\mathscr{J} = \big \{ (u,\rho) \in X^{3,4,4} ~ | ~ (u,\rho) \text{ satisfy the inequalities of Theorem } \ref{first_big_theo} ~\big \}
\nonumber
\end{gather}
Clearly $\mathscr{J}$ is open in $X^{3,4,4}$. Then let $\widetilde{\Phi}: \mathscr{J} \rar Y$ be defined as in (\ref{tilde_Phi}). 
As we showed in section \ref{proofs}, $\widetilde{\Phi}$ is a smooth map and its derivative at any point $(u,\rho)$ is a 
surjective map from $X^{3,4,4}$ to $Y$. Hence $\mathscr{I} = \{ (u,\rho) \in \mathscr{J} ~|~ \widetilde{\Phi}(u,\rho) = 0 \}$ is a 
submanifold 
of $\mathscr{J}$ or $X^{3,4,4}$.

Now we must show that $\psi_t: U_t \rar \mathscr{D}^3 \times H^3(\Om, \RR^n)$ is differentiable, and to do so we will use 
equation  (\ref{int_eq_zeta}). Note that $\nabla g$ of that equation is equal to $\nabla \Delta^{-1} \dot{f}$ where $f$ is the 
solution of (\ref{eq_f_convected}). Of course $f$ and therefore $\nabla g$ depend on the initial data $(u,\rho)$ so we 
define for each $t$:
\begin{gather}
 \tl_t: U_t \rar H^3(\Om,\RR^n), \nonumber \\
 \tl_t(u,\rho) = \nabla g .
\nonumber
\end{gather}
We shall need the following proposition which will be proven in section \ref{diff_dep_quasi}.
\begin{prop}
 $\tl_t: U_t \rar H^3(\Om,\RR^n)$ is a $C^1$ map.
\label{diff_map_technical}
\end{prop}
\begin{rema} Since $\dot{f}$ is in $H^3$, $\nabla g$ is actually in $H^4$. However, the map $(u,\rho) \mapsto \nabla g$ is 
continuous in $H^4$, but probably not differentiable.
\end{rema}
Given Proposition \ref{diff_map_technical}, our proof that $\psi_t$ is differentiable will be a modification of the proof that the 
solution of an ordinary differential equation depends differentiably on initial conditions (see for example \cite{L}). In this 
case the equation will be (\ref{int_eq_zeta}).

We shall need the following.
\begin{lemma}
 Let $w \in H^4(\Om,\RR^n)$. Then the map $\zeta \mapsto w\circ\zeta$ is a $C^1$ map from $\mathscr{D}^3$ to $H^3(\Om,\RR^n)$.
\label{calculus_lemma}
\end{lemma}
\begin{proof}
This is just a calculus lemma, see \cite{E1}, \cite{BB}, or \cite{EM}
\end{proof}
To show that $\psi_t$ is differentiable, it suffices to show that $(u,\rho) \mapsto \dot{\zeta}$ is differentiable, because
\begin{gather}
 \zeta(t) = \operatorname{id} + \int_0^t \dot{\zeta}(s) ds.
\label{ftc_zeta}
\end{gather}
To show that $\dot{\zeta}$ depends differentiably on $(u,\rho)$ we rewrite (\ref{int_eq_zeta}), using 
(\ref{ftc_zeta}) to replace $\zeta$ with $\dot{\zeta}$.

First we rewrite the terms of (\ref{int_eq_zeta}), letting
\begin{gather}
 A(u,\rho) = Pu 
\nonumber \\
B(\zeta,\dot{\zeta}) = Z(\zeta,\dot{\zeta}) - R(\zeta,\dot{\zeta}) \nonumber \\
C(u,\rho,\zeta,t) = (\nabla g(t) )\circ \zeta(t) = \tl_t(u,\rho)\circ\zeta(t)
\nonumber
\end{gather}
Then (\ref{int_eq_zeta}) becomes
\begin{gather}
 \dot{\zeta}(t) = A(u,\rho) + \int_0^t B\left( \zeta(s),\dot{\zeta}(s)\right)ds + C(u,\rho,\zeta,t)
\nonumber
\end{gather}
or
\begin{gather}
 \dot{\zeta}(t) = A(u,\rho) + \int_0^t B\big(\operatorname{id} 
+ \int_0^s \dot{\zeta}(s^\prime)ds^\prime,\,\dot{\zeta}(s)\big)ds + C(u,\rho,\zeta,t) .
\label{long_int_dot_zeta}
\end{gather}
Now we use an argument from \cite{L}, to get the required differentiability. Let
\begin{gather}
X=C([0,t],H^3(\Om,\RR^n))
\nonumber
\end{gather}
 be the 
Banach space of continuous curves from $[0,t]$ to $H^3(\Om,\RR^n)$ with norm
\begin{gather}
 \n z \n = \sup_{0\leq s\leq t}\{ \n z(s) \n_3 \}.
\nonumber
\end{gather}
Let $T: U_t \times X \rar X$ be defined by
\begin{align}
\begin{split}
T(u,\rho,z)(s) & =  A(u,\rho) + \int_0^s B\big(\operatorname{id} 
+ \int_0^\ell z(t^\prime)dt^\prime,\,z(\ell)\big) d\ell 
\\
& + C \big(u,\rho,\operatorname{id} + \int_0^s z(\ell)d\ell,s \big) -z(s).
\end{split}
\label{def_T}
\end{align}
Checking carefully, we see that $z(t)$ is a solution of (\ref{long_int_dot_zeta}) 
if and only if $T(u,\rho,z)=0$. Thus 
given $(u,\rho)$ we find that
\begin{gather}
 T(u,\rho,\dot{\zeta}) = 0
\nonumber
\end{gather}
We now use the implicit function theorem to show that $\dot{\zeta}$ is a $C^1$ function of $(u,\rho)$. First note that the maps
$A$, $B$, and $C$ are all $C^1$: $A: X^{3,4,4} \rar H^3(\Om,\RR^n)$ is a continuous linear 
map; $B: \mathscr{D}^3\times H^3(\Om,\RR^n) \rar H^3(\Om,\RR^n)$ is $C^\infty$ since $Z$ and $R$ are $C^\infty$ as is 
shown in \cite{E2}; $C: U_t \times \mathscr{D}^3 \times [0,t] \rar H^3(\Om,\RR^n)$ is $C^1$ by 
Proposition \ref{diff_map_technical} and Lemma
\ref{calculus_lemma}. Therefore $T$ is also $C^1$.

Now we compute the partial derivative of $T$ with respect to its last (i.e, $z$) variable. Using
(\ref{def_T}) we find that the $z$-partial derivative of $T$ in direction $\xi$ is:
\begin{align}
\begin{split}
D_z T(u,\rho,z)(\xi)(s) =  &
\int_0^s DB\big(\operatorname{id} + \int_0^\ell z(t^\prime)dt^\prime,\,z(\ell)\big)\big(\int_0^\ell\xi(t^\prime)dt^\prime,\xi(\ell)\big) d\ell 
\\
& + DC\big(u,\rho,\operatorname{id} + \int_0^s z(\ell)d\ell,s\big)\big(0,0,0,\int_0^s \xi(\ell)d\ell\big) - \xi(s).
\end{split}
\label{partial_z_T}
\end{align}
Therefore we can write
\begin{gather}
 D_z T(u,\rho,z)\xi = \ep(u,\rho,z)\xi - \xi
\nonumber
\end{gather}
where $\ep(u,\rho,z): X \rar X$ is a linear map defined by (\ref{partial_z_T}).

Since $\ep(u,\rho,z): X \rar X$ involves $\int_0^t\xi$, we can pick $t>0$ small enough that the operator norm of $\ep(u,\rho,z)$ is 
less than $1$. In that case $D_zT(u,\rho,z): X \rar X$ is an invertible map.

Now fix a fluid motion $\zeta$ with initial data $(u,\rho)$, so $T(u,\rho, \dot{\zeta}) = 0$. Then if $D_z T(u,\rho, \dot{\zeta})$ is 
invertible, the implicit function theorem tells us that there is a neighborhood $V_t$ of $(u,\rho)$ in $U_t$ and a $C^1$ map
$\wedge: V_t \rar X$ such that for any $(\tilde{u},\tilde{\rho}) \in V_t$, $T(\tilde{u},\tilde{\rho},\wedge(\tilde{u},\tilde{\rho})) = 0$. Thus $\dot{\tilde{\zeta}} = \wedge(\tilde{u},\tilde{\rho})$ gives 
a fluid motion with initial data $(\tilde{u},\tilde{\rho})$. But if $t$ is small, $D_z T(u,\rho,\dot{\zeta})$ is 
invertible, so we get the $C^1$ map $\wedge$. Therefore since
$\psi_t(\tilde{u},\tilde{\rho})=\wedge(\tilde{u},\tilde{\rho})(t)$, $\psi_t: U_t \rar \mathscr{D}^3 \times H^3(\Om,\RR^n)$ is $C^1$ 
as well.

To show that $\psi_t$ is $C^1$ for large $t$, we note that for $0<t_1<t$
\begin{align}
\begin{split}
 \dot{\zeta}(t) =&\,  P( \dot{\zeta}(t_1)\circ\zeta^{-1}(t_1))\circ\zeta(t_1)  
  + 
\int_{t_1}^t B \big(\zeta(t_1) 
+ \int_{t_1}^s\dot{\zeta}(s^\prime)ds^\prime, \dot{\zeta}(s) \big)ds 
\\ &
+ 
C(u,\rho, \zeta(t_1) + \int_{t_1}^t \dot{\zeta}(s)ds,t)
\end{split}
\label{long_int_dot_zeta_t_1}
\end{align}
Using (\ref{long_int_dot_zeta_t_1}) instead of 
(\ref{long_int_dot_zeta}), and defining $T$ analogously, we find 
that $\dot{\zeta}(t)$ is a $C^1$ function of $(u,\rho, \zeta(t_1),\dot{\zeta}(t_1))$. Iterating 
this argument we get that $\dot{\zeta}(t)$ is a $C^1$ function of $(u,\rho, \dot{\zeta}(0)) = (u,\rho,u)$. 
It follows that $\psi_t$ is $C^1$ for all $t$, so Theorem \ref{submanifold_theo} is proven.
\end{proof}
\begin{rema}
 The map $(u_0,\rho_0) \mapsto u(t)$ defined from $U_t$ to $H^3(\Om,\RR^n)$ cannot be shown to be $C^1$ by this method. 
The use of Lagrange coordinates is essential. Indeed $(u_0,\rho_0) \mapsto u(t)$ is probably not $C^1$ or even H\"older continuous. 
See the counter example of \cite{K2} and also section \ref{diff_dep_quasi} of this paper.
\end{rema}

\section{Differentiable Dependence on Initial Conditions for Quasilinear Symmetric Hyperbolic Equations. Proof of Proposition \ref{diff_map_technical}\label{diff_dep_quasi}}

The specific goal of this section is the proof of Proposition \ref{diff_map_technical}, which says that the 
gradient part of a compressible fluid velocity depends differentiably on the initial conditions. However, since the proof is 
basically the same for a general quasi-linear symmetric hyperbolic system, we will discuss the general case as well.

Consider the initial value problem
\begin{gather}
 \begin{cases}
\partial_t u + \sum_{i=1}^n a_i(t,x,u) \partial_i u = f(t,x,u) \\
u(0,x) = u_0(x)  
 \end{cases}
\label{ivp}
\end{gather}
where $0\leq t \leq T$, $x \in \RR^n$, $f$ and $u$ take values in $\RR^m$ and $a_i(t,x,u)$ is a symmetric $m \times m$ matrix. 
To avoid extra technicalities, we will assume that $\{ a_i \}$ and $f$ are smooth.

It is well known (see \cite{K2}  for example) that for small $t$, this problem has a unique solution. We shall discuss this 
solution using the results and function spaces of \cite{K2}. However, in order to avoid imposing conditions at $|x|\rar \infty$, we 
shall assume that $a_i$, $f$, $u_0$ and $u$ are all periodic in $x$, or equivalently $x \in \TT^n = \RR^n/\ZZ^n$.

\begin{theo}
If $s > \frac{n}{2} + 1$, then for each $u_0 \in H^s(\TT^n)$ there exists a $T> 0$ and a unique 
solution $u(t,x)$ on $[0,T]\times \TT^n$
\label{existence_kato}
\end{theo}
\begin{proof}
See \cite{K2}.
\end{proof}
Letting $u(t,x)=u(t)(x)$, we consider $u$ as a curve of functions on $\TT^n$. This curve is continuous in $H^s(\TT^n)$, 
that is, it is an element of 
\begin{gather}
 CH^s([0,T]) = C^0([0,T],H^s(\TT^n)).
 \nonumber
\end{gather}
Also there exists a neighborhood $U$ of $u_0$ in $H^s(\TT^n)$, such that for each $v_0$ in $U$, there 
exists a solution $v$ of (\ref{ivp}) with $v \in CH^s([0,T])$ and $v(0)=v_0$. Furthermore $v$ depends continuously on $v_0$.

Counter examples in \cite{K2} show that $u(t) \in H^s(\TT^n)$ do not depend differentiably (or even H\"older continuously)
on $u_0$ however. We shall investigate this phenomenon and show that the dependence of $u$ on $u_0$ is $C^1$ if we 
reduce $s$ by one in the range. That is, the map from $u_0 \in H^s$ to $u \in CH^{s-1}([0,T])$ is $C^1$. Then using this 
idea, we will prove Proposition \ref{diff_map_technical}.

First we consider a $C^1$-curve of initial data $u_0^\la$ and let $u^\la(t)$ be a solution of (\ref{ivp}) with 
initial data $u^\la_0$. When $\la=0$ we suppress the superscript so that $u_0^0 = u_0$ and $u^0(t)=u(t)$. Let us assume that all 
derivatives exist and let
\begin{gather}
 z(t) = \partial_\la(u^\la(t))_{|\la=0}.
\nonumber
\end{gather}
Then differentiating (\ref{ivp}) we find:
\begin{gather}
\partial_t z + \sum_{i=1}^n a_i(t,x,u) \partial_i z + \sum_{i=1}^n \partial_u(a_i(t,x,u))z\partial_i u  = \partial_u f(t,x,u)z 
\label{ivp_z}
\end{gather}
and we rewrite this as
\begin{gather}
 \partial_t z + A(u) \nabla z + B(u,\nabla u) z = F(u) z .
\label{ivp_z_abs}
\end{gather}
This equation is linear symmetric-hyperbolic and such equations are also analyzed in \cite{K2}. The solution of 
(\ref{ivp_z}) with $H^s$ initial data is shown to be a continuous curve in $H^{s-1}$. It is not shown to be in $H^s$ because the 
operator $B(u,\nabla u)$ is multiplication by $H^{s-1}$-functions (since it involves $x$-derivatives of $u$). Thus 
$B(u,\nabla u): H^{s-1} \rar H^{s-1}$ but $B(u,\nabla u)(H^s)\nsubseteq H^s$.

From this we might suspect that $\partial_\la (u^\la)$ exists as a curve in $CH^{s-1}$, but not generally in $CH^s$. We 
shall show that this is in fact what happens.

To begin with, we present a simple example which will help us both with the symmetric hyperbolic system and with the proof of 
proposition \ref{diff_map_technical} (compare to the counter example of \cite{K2}). Let $n=m=1$, and let $u_0^\la$ be a $C^1$ 
curve in $H^s(\TT^1,\RR^1)$ ($s > \frac{3}{2}$) parameterized by $\la$. Then for each $\la$, let $u^\la$ be the solution of 
\begin{gather}
 \partial_t u^\la + u^\la \partial_x u^\la =0, ~~~u^\la(0) = u^\la_0.
\label{eq_model}
\end{gather}
(this is a one dimensional compressible fluid motion with $p(\rho)=0$). Such a solution is easily found: Let $\zeta^\la(t)$ be a 
curve in $\mathscr{D}^s$ defined by $\zeta^\la(t) = x+tu_0^\la(x)$, so $\dot{\zeta}^\la(t) = u_0^\la$. Then let $u^\la(t) = \dot{\zeta}^\la(t)\circ(\zeta^\la(t))^{-1} = u_0^\la \circ(\zeta^\la(t))^{-1}$ (compare with
(\ref{zeta_dot_u}) ). Clearly:
\begin{gather}
 u^\la(t) \circ \zeta^\la(t) = u^\la_0
\label{u_la_zeta_la_model}
\end{gather}
and differentiating (\ref{u_la_zeta_la_model}) with respect to $t$, we find that $u^\la$ satisfies
(\ref{eq_model}). Since $\zeta^\la(t)$ is a continuous curve in $\mathscr{D}^s$ (at least for $t$ near zero), 
$(\zeta^\la(t))^{-1}$ is also continuous in $\mathscr{D}^s$. Hence $u^\la(t)$ is continuous in $H^s(\TT,\RR)$. Thus if $T$ is 
small enough so that $\zeta^\la(t) \in \mathscr{D}^s$ for all $t$ between zero and $T$, then $u^\la \in CH^s([0,T])$. 
Also it is clear 
from the construction that $u^\la$ depends continuously on $\la$ and that we can in fact find numbers $T$ and $\wedge$ such 
that the map $\la \mapsto u^\la$ is continuous from $(-\wedge,\wedge)$ to $CH^s([0,T])$.

Now we consider differentiability with 
respect to $\la$. Let $z_0 = \partial_\la (u^\la_0)_{|\la=0}$ 
and $z = \partial_\la(u^\la)_{|\la=0}$. Then $\partial_\la(\zeta^\la(t))_{|\la=0} = tz$, 
so differentiating (\ref{u_la_zeta_la_model}) with respect to $\la$ and letting $\la=0$, we get
\begin{gather}
z(t)\circ \zeta(t) + (\partial_x u(t))\circ \zeta(t)\, tz_0 = z_0
\label{eq_z_eta_model}
\end{gather}
where we omit the ``$\la$'' when $\la=0$. Also applying $\partial_x$ to (\ref{u_la_zeta_la_model}) we get:
\begin{gather}
 ((\partial_x u^\la(t) )\circ \zeta^\la(t))\partial_x (\zeta^\la(t)) = \partial_x u_0^\la .
\nonumber
\end{gather}
But 
\begin{gather}
 \partial_x(\zeta^\la(t)) = 1 + t \partial_x u_0^\la
\nonumber
\end{gather}
so
\begin{gather}
 \partial_x(u^\la(t))\circ \zeta^\la(t) = \frac{\partial_x u_0^\la}{1+t\partial_x u^\la_0} .
\label{eq_partial_x_u_la_comp_zeta_la_model}
\end{gather}
Combining (\ref{eq_z_eta_model}) and (\ref{eq_partial_x_u_la_comp_zeta_la_model}) we find:
\begin{gather}
 z(t)\circ \zeta(t) = z_0\Big ( 1-\frac{t\partial_x u_0}{1+t\partial_x u_0} \Big ) = z_0 \Big ( \frac{1}{1+t\partial_x u_0} \Big ).
\label{exp_z_zeta_la_model}
\end{gather}
From (\ref{exp_z_zeta_la_model}) we see that $z$ is an element of $CH^{s-1}$ which depends continuously on $u_0$, and from 
this it follows that $u \in CH^{s-1}$ is a $C^1$-function of $u_0$

Differential dependence on $u_0$ can also be shown by another method which will prove useful later. Differentiating 
(\ref{eq_model}) with respect to $\la$ we find that $z$, if it exists, must satisfy:
\begin{gather}
 \partial_t z + u \partial_x z + (\partial_x u)z = 0 .
\label{eq_model_anot_1}
\end{gather}
This equation is of the same type as (\ref{ivp_z_abs}), so it can be solved for $z \in CH^{s-1}$ by the methods of \cite{K2}. Also 
the difference quotient $z^\la = \frac{1}{\la}(u^\la - u)$ satisfies
\begin{gather}
 \partial_t z^\la + u \partial_x z^\la + (\partial_x u^\la) z^\la = 0,~~~z^\la(0) = z^\la_0 = \frac{1}{\la}(u_0^\la - u_0)
\label{eq_model_anot_2}
\end{gather}
so $y = z^\la - z$ satisfies
\begin{gather}
 \partial_t y + u \partial_x y+ (\partial_xu)y = \la z^\la \partial_x z^\la,~~~y(0)=y_0=z^\la_0 - z_0 .
\label{eq_model_anot_3}
\end{gather}
The estimates of \cite{K2} show that $y\rar 0$ in $CH^{s-1}$ as $\la \rar 0$, and from this it follows that $u \in CH^{s-1}$ 
depends differentiably on $u_0$.

We will use this method to show differentiability of the solution of (\ref{ivp}), and also in the proof of 
Proposition \ref{diff_map_technical}.

One can also solve (\ref{eq_model_anot_1})-(\ref{eq_model_anot_3}) more directly as follows: (\ref{eq_model_anot_1}) is equivalent to
\begin{gather}
 (z(t)\circ \zeta(t))\dot{} + (\partial_x u(t)\circ \zeta(t))(z(t)\circ \zeta(t) ) = 0
\nonumber
\end{gather}
so
\begin{gather}
z(t)\circ \zeta(t) = \exp(-\int_0^t (\partial_x u(s) )\circ \zeta (s) ds) z_0 .
\label{eq_model_direct_1}
\end{gather}
Similarly from (\ref{eq_model_anot_2}) we get
\begin{gather}
 z^\la\circ \zeta(t) = \exp(-\int_0^t (\partial_x u^\la(s) )\circ \zeta(s) ds )z_0^\la .
\label{eq_model_direct_2}
\end{gather}
Also, (\ref{eq_model_anot_3}) is equivalent to:
\begin{gather}
 ( y(t)\circ \zeta(t) )\dot{} + ( (\partial_x u(t) )\circ \zeta(t) ) (y(t)\circ\zeta(t)) = \la(z^\la \partial_x z^\la)\circ \zeta(t) .
\label{eq_model_direct_3}
\end{gather}
Therefore 
\begin{align}
\begin{split}
 y(t)\circ \zeta(t) = & \, \exp(-\int_0^t (\partial_x u)(s) \circ \zeta(s) ds )y_0 \\
 & + \int_0^t 
\exp(-\int_s^t (\partial_x u)(\tau)\circ \zeta(\tau) d\tau)(z^\la(s)\la \partial_x z^\la(s) )\circ \zeta (s) ds 
\end{split}
\label{eq_model_direct_big}
\end{align}
As $\la \rar 0$, $u^\la \rar u$ in $CH^s$. Hence $\partial_x u^\la - \partial_x u = \la \partial_x z^\la \rar 0$ in $CH^{s-1}$. 
Also from (\ref{eq_model_direct_2}) we find that $z^\la$ is bounded in $CH^{s-1}$ uniformly in $\la$. Therefore the last term of 
(\ref{eq_model_direct_big}) goes to zero in $CH^{s-1}$ as $\la \rar 0$. But $z_0^\la \rar z_0$ by 
definition, so $y_0 \rar 0$. Hence by (\ref{eq_model_direct_big}) $y(t)\circ \zeta(t) \rar 0$ in $CH^{s-1}$ as $\la \rar 0$. An 
argument of this type will be used in the proof of Proposition \ref{diff_map_technical} also.

We now proceed to prove differentiability for the symmetric hyperbolic case.

\begin{prop}
 Let $u_0 \in U \subset H^s(\TT^n)$ and $u \in CH^s([0,T])$ be as in Theorem \ref{existence_kato}. Also let $v$ be the 
solution of (\ref{ivp}) for $v_0 \in U$ as in that theorem. Define $\Phi: U \rar CH^{s-1}([0,T])$ by $\Phi(v_0) = v$. 
Then $\Phi$ is $C^1$. In fact, $D\Phi(u_0)z_0$ is the solution of (\ref{ivp_z_abs}) with initial data $z_0$.
\label{Phi_C_1_sym_hyper}
\end{prop}

To prove this proposition we will need the following Lemma which is a variant of the $\Om$-Lemma (see e.g. \cite{E1} or \cite{P}).
\begin{lemma}
 Let $h=h(t,x,u)$ be a smooth function from $[0,T]\times \RR^n\times \RR^m$ to some $\RR^k$ 
and for $s^\prime > \frac{n}{2}$ let $\Om_h: CH^{s^\prime} \rar CH^{s^\prime}$ be 
defined by $\Om_h(u)(t)(x)= h(t,x,u(t)(x))$. Then $\Om_h$ is a smooth map. Also the derivative of $\Om_h$ at $u$ in the 
direction $z$ obeys the formula:
\begin{gather}
 (D(\Om_h)(u)z)(t)(x) = \partial_uh(t,x,u(t)(x))(z(t)(x))
\nonumber
\end{gather}
or more succinctly
\begin{gather}
 D(\Om_h)(u)z = \Om_{\partial_u h}(u)z.
\label{D_Om_variant_Om_lemma}
\end{gather}
\label{variant_Om_lemma}
\end{lemma}
\begin{proof}
 The proof is the same as the proof of the $\Om$-Lemma in \cite{E1} or \cite{P}, so we shall omit most of it. We shall 
however compute the first derivative of $\Om_h$.

Let $u^\la = u +\la z \in CH^{s^\prime}$. Then for each $t$ and $x$
\begin{gather}
 \Om_h(u^\la)(t)(x) - \Om_h(u)(t)(x) = \int_0^1 \partial_u h(t,x,u^{\tau \la}(t,x))\la z(t,x) d\tau .
\nonumber
\end{gather}
Therefore 
\begin{gather}
 \Om_h(u^\la)(t)(x) - \Om_h(u)(t)(x) = \int_0^1 \Om_{\partial_u h}(u^{\tau\la})\la z d\tau 
= \la \Big( \int_0^1 \Om_{\partial_u h}(u^{\tau \la})d\tau \Big) z .
\nonumber
\end{gather}
Hence 
\begin{gather}
 \frac{1}{\la} \big(\Om_h(u^\la) - \Om_h(u)\big) - \Om_{\partial_u h}(u)z = \Big(\int_0^1 (\Om_{\partial_u h}(u^{\tau \la}) 
- \Om_{\partial_u h}(u) )d\tau \Big)z .
\label{computing_der_Om_h}
\end{gather}
The right hand side of (\ref{computing_der_Om_h}) goes to zero in $CH^{s^\prime}$ as $\la \rar 0$, and 
(\ref{D_Om_variant_Om_lemma}) follows.
\end{proof}
\noindent \emph{Proof of Proposition \ref{Phi_C_1_sym_hyper}:} Let $u$ be the solution of (\ref{ivp}) with initial 
data $u_0$ and let $u^\la$ be the solution with initial data $u_0^\la = u_0 + \la z_0$. Also let $z$ be the 
solution of (\ref{ivp_z_abs}) with initial data $z_0$ and let $z^\la$ equal $\frac{1}{\la}(u^\la - u)$.

We shall show that $\lim_{\la\rar 0}z^\la = z$ where the limit is in $CH^{s-1}$. Since $u$ and $u^\la$ 
satisfy (\ref{ivp}), $z^\la$ satisfies:
\begin{gather}
 \partial_t z^{\la} + \sum_{i=1}^n a_i(t,x,u)\partial_i z^\la + 
 \sum_{i=1}^n\frac{1}{\la}(a_i(u^\la) - a_i(u))\partial_iu^\la =
\frac{1}{\la}(f(u^\la) - f(u))
\label{eq_z_la_proof}
\end{gather}
We want to show that the solution of (\ref{eq_z_la_proof})  is near $z$ and to do so we rewrite (\ref{eq_z_la_proof}) as:
\begin{gather}
\partial_t z^\la + A(u) \nabla z^\la + B(u,\nabla u) z^\la = F(u)z^\la + E
\label{eq_z_la_proof_abs}
\end{gather}
where $A$, $B$, and $F$ are as in (\ref{ivp_z_abs}) and 
\begin{gather}
E = \sum_{i=1}^n \Big (\partial_u a_i(u) z^\la \partial_i u - \frac{1}{\la}( a_i(u^\la) - a_i(u) )\partial_i u^\la \Big )
+\frac{1}{\la}(f(u^\la) - f(u) ) -\partial_u f(u) z^\la
\label{definition_E_proof}
\end{gather}
Subtracting (\ref{ivp_z_abs}) from (\ref{eq_z_la_proof_abs}) we find
\begin{gather}
\partial_t(z^\la - z)+  A(u) \nabla(z^\la - z) + B(u,\nabla u)(z^\la - z) = F(u)(z^\la - z) + E
\label{eq_z_difference_proof}
\end{gather}
and $z^\la(0) - z(0) = 0$. From standard energy estimates (or from Theorem I of \cite{K2}) we find that 
the solution of (\ref{eq_z_difference_proof}) obeys:
\begin{gather}
 \n z^\la(t) - z(t) \n_{s-1} \leq e^{Kt} \int_0^t \n E(\tau) \n_{s-1} d\tau
\label{energy_esimate_difference_z_la_z_proof}
\end{gather}
where $K$ depends only on $\max_{0\leq t \leq T} \{ \n u(t) \n_s \}$, the norm of $u$ in $CH^s$.

Thus to show that $z^\la(t) \rar z(t)$ in $H^{s-1}$ it suffices to estimate $\n E(t) \n_{s-1}$. To do this it will be 
convenient to use the notation $O(\la)$. We say that a function $h(\la,t)$ is $O(\la)$ if $\lim_{\la \rar 0} h(\la,t) = 0$ 
uniformly in $t \in [0,T]$.

From (\ref{definition_E_proof}) we find:
\begin{gather}
 \n E \n_{s-1} \leq \sum_{i=1}^n \n \frac{1}{\la}( a_i(u^\la) - a_i(u) ) - \partial_u a^i(u) z^\la \n_{s-1} \n\partial_i u \n_{s-1}
 \nonumber \\
+ \sum_{i=1}^n \frac{1}{\la}\n a_i(u^\la) - a_i(u) \n_{s-1} \n \partial_i u^\la - \partial_i u \n_{s-1} 
+ \n \frac{1}{\la}( f(u^\la) - f(u)) - \partial_u f(u) z^\la \n_{s-1} \nonumber \\
= I + II + III
\nonumber
\end{gather}
We note that $I$ is $O(\la)$ by Lemma \ref{variant_Om_lemma}.

We proceed to estimate $II$. $\n \partial_i u^\la - \partial_i u \n_{s-1}$ is $O(\la)$, and 
$\frac{1}{\la}\n a_i(u^\la) - a_i(u) \n_{s-1}$ is $O(\la)$ $+ \n \partial_u a^i(u)z^\la \n_{s-1}$. Therefore
\begin{gather}
 II \leq O(\la)(1+\n z^\la \n_{s-1} )
\nonumber
\end{gather}
But then since $\n z \n_{s-1}$ is independent of $\la$ we get:
\begin{gather}
 II \leq O(\la)(1+ \n z^\la - z \n_{s-1})
\nonumber
\end{gather}
From Lemma \ref{variant_Om_lemma}, we find that $III$ is $O(\la)$ also and thus we get
\begin{gather}
 \n E \n_{s-1} \leq O(\la) ( 1 + \n z^\la - z \n_{s-1} )
\label{estimate_E_s_1_O_la}
\end{gather}
Hence from (\ref{energy_esimate_difference_z_la_z_proof}) we get:
\begin{gather}
 \n z^\la(t) - z(t) \n_{s-1} \leq O(\la)e^{Kt}\int_0^t \n z^\la(r) - z(r) \n_{s-1} dr + O(\la)
\nonumber
\end{gather}
from this it follows, iterating the inequality, that $\n z^\la - z \n_{s-1}$ is $O(\la)$. Thus the 
derivative of $\Phi$ at $u_0$ in direction $z_0$ is $z$, the solution of (\ref{ivp_z_abs}). But this $z$ is clearly a 
continuous linear function of $z_0$ and from \cite{K2} we know that $z\in CH^{s-1}$ is a continuous function of $u \in CH^s$. 
Therefore $\Phi$ is $C^1$ and the proposition if proved. \hfill $\qed$\\ \medskip

Proposition \ref{diff_map_technical} does not follow as a special case of Proposition
\ref{Phi_C_1_sym_hyper} because $\Om$ is a domain with boundary and we must consider initial-boundary value problems. 
But we will see that much of the argument is the same.

Proposition \ref{diff_map_technical} says that $\tl_t$ is a $C^1$ map,
 where $\tl_t(u_0,\rho_0) = \nabla g(t)$, but 
since $\nabla g(t) = \nabla \Delta^{-1} \dot{f}$, is suffices to show that $(u_0,\rho_0)\mapsto \dot{f}$ 
is a $C^1$ map from $U_t$ to $H^2(\Om,\RR)$. We shall prove a slightly different proposition that clearly implies this.

Given $(u_0,\rho_0) \in U_t$ pick an interval $[0,T]$ such that $t\in [0,T]$ and the fluid motion $u(t)$,$\rho(t)$ with 
initial data $(u_0,\rho_0)$ is defined on $[0,T]$. As before we let $f_0 = \log \rho_0$ and $f(t) = \log \rho(t)$. Let
\begin{gather}
 Z = \big \{ x \in C^0([0,T],H^2(\Om,\RR^n))~|~\dot{x} \in \bigcap_{k=0}^2 C^k([0,T],H^{2-k}(\Om,\RR^n) ) \big \}
\nonumber \\
X_0^k =  \bigcap_{j=0}^k C^j([0,T],H^{k-j}(\Om,\RR) ) 
\nonumber \\
X_{\frac{1}{2},\partial}^{k+\frac{1}{2}} = \bigcap_{j=0}^k C^j([0,T],H^{k+\frac{1}{2}-j}(\partial \Om,\RR) )
\nonumber
\end{gather}
and let $\n\cdot \n_Z$, $\n \cdot \n_{X_0^k}$ and $\n \cdot \n_{X_{\frac{1}{2},\partial}^{k+\frac{1}{2}}}$ be the appropriate norms for 
these spaces (see below).

Let $\widetilde{\tl}(\tilde{u}_0,\widetilde{f}_0)=(\tilde{u},\tilde{f})$, where $\tilde{u}$, $\tilde{f}$ defines the 
fluid motion with initial data $\tilde{u}_0$, $\tilde{f}_0$. Then $\widetilde{\tl}$ is defined from a 
neighborhood $V$ of $(u_0,f_0)$ in $U_t$ to $Z \times X_0^3$.

\begin{prop} $\widetilde{\tl}: V \rar Z\times X_0^3$ is a $C^1$ map.
\label{tilde_tl_C_1}
\end{prop}
\begin{proof}
First we note that $(u,f)$ satisfies the system
\begin{subnumcases}{}
 \dot{u} = -c^2(f) \nabla f 
\label{system_u_f} \\
\ddot{f}= -\de c^2(f) \nabla f + u^i_j u^j_i 
\label{system_u_f_dot_u} 
\end{subnumcases}
with boundary condition
\begin{gather}
 c^2 \nabla_\nu f = - S_2(u,u) \text{ on } \partial \Om,
\label{bc_system_u_f}
\end{gather}
and initial condition
\begin{gather}
 f(0) = f_0,~~\dot{f}(0) = \de u_0,~~u(0) = u_0.
\label{ic_system_u_f}
\end{gather}
Proceeding as in the proof of Proposition \ref{Phi_C_1_sym_hyper}, we let $(u^\la,f_0^\la)$ be a $C^1$ curve of initial 
conditions in $V$ and let $(u^\la,f^\la)$ be the corresponding solutions of (\ref{system_u_f})-(\ref{system_u_f_dot_u}), (\ref{bc_system_u_f}), (\ref{ic_system_u_f}). If 
all derivatives exist, $(z_0,h_0) = \partial_\la(u_0^\la,f_0^\la)_{|\la=0}$ 
and $(z,h) = \partial_\la(u^\la,f^\la)_{|\la=0}$, then $(z,h)$ must satisfy
\begin{subnumcases}{}
\dot{z} + \nabla_z u = -(c^2)^\prime (f) h \nabla f - c^2 \nabla h \label{sys_z_h_1} \\
\ddot{h} + \nabla_z \dot{f} + (\nabla_z f )\dot{} = -\de c^2 \nabla h - 
\de (c^2)^\prime(f) h \nabla f + 2 z_j^i u_i^j \label{sys_z_h_2} 
\end{subnumcases}
with initial-boundary conditions
\begin{gather}
 (c^2)^\prime h \nabla_\nu f + c^2 \nabla_\nu h = -2S_2(u,z) \text{ on } \partial \Om, \label{bc_sys_z_h}  \\
z(0)=z_0,~ h(0)=h_0,~ \text{ and } \dot h(0)+\nabla_{z(0)} f_0 = \de z_0. \label{ic_sys_z_h} 
\end{gather}
Here the extra terms on the left of (\ref{sys_z_h_1})-(\ref{ic_sys_z_h}) come from the fact that `` $\dot{}$ '' depends on $\la$.
For example
\begin{gather}
 \partial_\la(\dot{u}^\la) = \partial_\la(\partial_t u^\la + \nabla_{u^\la} u^\la ) 
= \partial_t \partial_\la u^\la + \nabla_{u^\la} \partial_\la u^\la 
+ \nabla_{\partial_\la u} u^\la,
\nonumber
\end{gather}
so 
\begin{gather}
 \partial_\la(\dot{u}^\la)_{|\la=0} = \dot{z} + \nabla_z u.
\nonumber
\end{gather}
Also $c$ is a function of $f$ (see (\ref{definition_of_c})) and $(c^2)^\prime(f)$ means $\partial_f (c^2(f))$. Below we sometimes
write $c^2$ for $c^2(f)$.

The remainder of our proof will be like the proof of proposition \ref{Phi_C_1_sym_hyper} except that (\ref{ivp_z_abs}) will be 
replaced by (\ref{sys_z_h_1})-(\ref{ic_sys_z_h}).  We let $(z^\la, h^\la) = \frac{1}{\la} ( u^\la - u, f^\la - f)$. Then
$(z^\la, h^\la)$ satisfies
\begin{subnumcases}{}
 \dot{z}^\la + \nabla_{z^\la} u^\la = -c^2 \nabla h^\la - \frac{1}{\la}( c^2(f^\la) - c^2(f) )\nabla f^\la \label{sys_z_h_la_1} \\
\ddot{h}^\la + \nabla_{z^\la} \dot{f}^\la + (\nabla_{z^\la} f^\la )\dot{} + \la \nabla_{z^\la}\nabla_{z^\la} f = -\de c^2 \nabla h^\la \nonumber \\
\; \; \; \; \; \; \; \; \; \; \; \; -\frac{1}{\la} \de (c^2(f^\la) - c^2(f) ) \nabla f^\la + z^{\la i}_j(u^{\la j}_i+u^j_i) \label{sys_z_h_la_2}
\end{subnumcases}
with boundary condition
\begin{gather}
 \frac{1}{\la}( c^2 (f^\la) - c^2(f) ) \nabla_\nu f^\la + c^2 \nabla_\nu h^\la = -S_2(u^\la + u, z^\la) \text{ on } \partial \Om.
\label{bc_sys_z_h_la}
\end{gather}
Assuming that $(z,h)$ satisfies (\ref{sys_z_h_1})-(\ref{ic_sys_z_h}) we find that $(z^\la - z,h^\la - h)$ satisfies the following:
\begin{gather}
 (z^\la - z )\dot{} + \nabla_{z^\la -z } u = -(c^2)^\prime(h^\la - h)\nabla f -c^2 \nabla( h^\la - h) + E_1,
\label{sys_z_h_la_difference_1}
\end{gather}
where 
\begin{gather}
 E_1 = -\nabla_{u^\la - u} z^\la - (\frac{1}{\la}(c^2(f^\la)-c^2(f)) - (c^2)^\prime h^\la)\nabla f -(c^2(f^\la) - c^2(f) )\nabla h^\la,
\label{def_E_1}
\end{gather}
and
\begin{align}
\begin{split}
 (h^\la - h)\,\ddot{} + \nabla_{z^\la - z} \dot{f} + (\nabla_{z^\la -z} f)\dot{} = & \,
 -\de c^2 \nabla(h^\la - h ) 
 \\ &
- \de (c^2)^\prime (h^\la - h)\nabla f + 2u_j^i(z^\la - z)^j_i + E_2,
\end{split}
\label{sys_z_h_la_difference_2}
\end{align}
where
\begin{align}
\begin{split}
 E_2 = & \,-\la (\nabla_{z^\la}\dot{h}^\la + (\nabla_{z^\la} h^\la)\dot{} + \nabla_{z^\la}\nabla_{z^\la} f)
  - \de(\frac{1}{\la}(c^2(f^\la) -c^2(f)) - (c^2)^\prime h^\la)\nabla f 
  \\
  & - \de( c^2(f^\la) 
- c^2(f))\nabla h^\la + (u^\la-u)^i_j z^{\la j}_i,
\end{split}
\label{def_E_2}
\end{align}
with boundary condition
\begin{gather}
 c^2 \nabla_\nu (h^\la - h) + (c^2)^\prime (h^\la - h) \nabla_\nu f = 2S_2(u, z^\la - z) + E_\partial,
\label{bc_sys_z_h_la_difference}
\end{gather}
where 
\begin{gather}
 E_\partial = -S_2(u^\la - u, z^\la) - (\frac{1}{\la}(c^2(f^\la) - c^2(f) ) - (c^2)^\prime h^\la)\nabla_\nu f -
(c^2(f^\la) - c^2(f))\nabla_\nu h^\la.
\label{def_E_partial}
\end{gather}
Notice that (\ref{sys_z_h_la_difference_1})-(\ref{def_E_partial}) is the same system as
(\ref{sys_z_h_1})-(\ref{ic_sys_z_h}) except for the inhomogeneous terms $E_1$, $E_2$, $E_\partial$

We shall show that both these systems have solutions in $Z \times X_0^3$. Also since $(u_0^\la,f_0^\la)$ is a $C^1$ curve 
in $V$, $(z_0^\la - z_0, h^\la_0 - h_0)$ is $O(\la)$ in $X^{3,4,4}$. We shall also get bounds for $E_1$, $E_2$ 
and $E_\partial$ in terms of $O(\la)$. From this it will follow that $(z^\la - z, h^\la - h)$ is $O(\la)$ 
in $Z\times X^3_0$, so $\widetilde{\tl}$ is differentiable at $(u_0,f_0)$ and $D\widetilde{\tl}(u_0,f_0)(z_0,h_0) = (z,h)$. One 
routinely sees that $(z,h)$ depends continuously on $(u,f)$, so $\widetilde{\tl}$ is in fact $C^1$, which is what we want to prove.

To solve (\ref{sys_z_h_la_difference_1})-(\ref{bc_sys_z_h_la_difference}), we shall rewrite it as:
\begin{subnumcases}{}
 \dot{x} + \nabla_x u = B_1 r + A_1r +E_1, \label{sys_z_h_la_dif_abs_1}  \\
\ddot{r} + \nabla_x \dot{f} + (\nabla_x f)\dot{} = Lr + B_2r + A_2 x + E_2, \label{sys_z_h_la_dif_abs_2} \\
c^2\nabla_\nu r + B_\partial r = Sx + E_\partial \text{ on } \partial \Om, \label{bc_sys_z_h_la_dif_abs}
\end{subnumcases}
where $x$ replaces $z^\la - z$, $r$ replaces $h$, $L$ is as in (\ref{eq_f_estimates}), and the other terms are defined in the 
obvious way.

We first consider the case $E_1=E_2=E_\partial=0$ which is a homogeneous linear system in $(x,r)$. 
Also let $(x_0,r_0) \in X^{3,4,4}$ be the initial data which we assume satisfy the compatibility conditions.

Given $r \in X_0^3$, we can consider (\ref{sys_z_h_la_dif_abs_1}) as an inhomogeneous linear equation in $x$. We will 
solve it following the idea of our example (see (\ref{eq_model_anot_1})-(\ref{eq_model_direct_big})). We first assume $r=0$, so we have:
\begin{gather}
 \dot{x} + \nabla_x u = 0
\label{eq_sys_x_r_1}
\end{gather}
But this is equivalent to 
\begin{gather}
 \partial_t (x(t)\circ \zeta(t))=-(\nabla_{x(t)}u(t))\circ \zeta(t).
\label{eq_sys_x_r_equiv_1}
\end{gather}
For a fixed $t$, the operator $B_3(t): H^2(\Om,\RR^n) \rar H^2(\Om,\RR^n)$, defined 
by 
\begin{gather}
B_3(t) y = -(\nabla_{y\circ\zeta(t)^{-1}}u\circ\zeta(t)^{-1})\circ \zeta(t),
\nonumber
\end{gather}
 is 
bounded since $u,~\zeta \in H^3$ and since $u(t)$ and $\zeta(t)$ are continuous in $t$, so is $B_3(t)$ in 
the operator norm. Thus (\ref{eq_sys_x_r_equiv_1}), which is 
\begin{gather}
 \partial_t y = -B_3y
\nonumber
\end{gather}
is simply a linear ordinary differential equation in $H^2(\Om,\RR^2)$. Hence we can solve it for $y(t)$ and then we define
$x(t) = y(t)\circ(\zeta(t))^{-1}$. It is clear that this $x(t)$ satisfies (\ref{eq_sys_x_r_1}) and that it is in $Z$. We 
let $U(t,s): H^2(\Om,\RR^n) \rar H^2(\Om,\RR^n)$ be the operator which gives the solution of (\ref{eq_sys_x_r_1}). That is,
for any $t,~s \in [0,T]$,  $x(t) = U(t,s)x_0$ is the solution of (\ref{eq_sys_x_r_1}) with initial condition $x(s)=x_0$. By 
standard means we can show that $\n U(t,s) \n \leq e^{K|t-s|}$ where $K$ depends only on $\n u \n_3$.

Now we can solve (\ref{sys_z_h_la_dif_abs_1}) for general $r \in X_0^3$. We have the usual (Duhamel's) formula
\begin{gather}
 x(t) = U(t,0)x_0 + \int_0^tU(t,s)(A_1r(s) + B_1r(s))ds
\nonumber
\end{gather}
Since $r(s) \in H^3$, $A_1 r(s) \in H^2$, and $B_1 r(s) \in H^3$, so $x(t) \in H^2$. Also 
since $x$ satisfies (\ref{sys_z_h_la_dif_abs_1}) and $r \in X_0^3$ it follows that $x \in Z$. Furthermore we get:
\begin{gather}
 \n x \n_Z \leq K ( \n x_0 \n_2 + \n r\n_{X_0^3})
\label{estimate_x_basic}
\end{gather}
Give $x \in Z$, we can also solve (\ref{sys_z_h_la_dif_abs_2})-(\ref{bc_sys_z_h_la_dif_abs}) for $r \in X_0^3$. In fact 
(\ref{sys_z_h_la_dif_abs_2}) is a second order hyperbolic equation for $r$ with inhomogeneous 
term $-\nabla_x\dot{f} -(\nabla_x f)\dot{} + A_2x$ and with Neumann type boundary condition (\ref{bc_sys_z_h_la_dif_abs}). 
Such equations are considered in \cite{M}
where good estimates for their solutions are given. In particular if we consider
\begin{gather}
\begin{cases}
 \ddot{r} = Lr + B_2r + A_2 x + F ,
\\
c^2 \nabla_\nu r + B_\partial r = G \text{ on } \partial \Om,
\end{cases}
\label{eq_M}
\end{gather}
and assume that compatibility conditions hold up to order $k-1$ we get:
\begin{gather}
 \n r(t) \n_{X_0^k} \leq K( \n r(0) \n_{X_0^k} + \int_0^t ( \n F(s) \n_{X_0^{k-1}} 
+ \n G(s) \n_{X^{k-\frac{1}{2}}_{\frac{1}{2},\partial}} )ds
\label{estimate_r_t_X_0_k}
\end{gather}
where
\begin{gather}
 \n r(t) \n_{X_0^k} = \sum_{j=0}^k \n \partial_t^j r(t) \n_{k-j} 
\nonumber \\
\n G(t) \n_{X^{k-\frac{1}{2}}_{\frac{1}{2},\partial}} = \sum_{j=0}^{k-1} \n \partial_t^j G(t) \n_{\partial,k-j-\frac{1}{2}}
\nonumber
\end{gather}
In the cases of 
(\ref{sys_z_h_la_dif_abs_2})-(\ref{bc_sys_z_h_la_dif_abs}) the inhomogeneous term $A_2 x(t)$ is in $H^1$, but not in $H^2$, and 
the other terms are at least as smooth. Thus if we let
\begin{gather}
 \n x(t) \n_Z = \n x(t) \n_2 + \n \dot{x}(t) \n_2 + \n \partial_t \dot{x} \n_1 + \n \partial_t^2 \dot{x} \n_0
\label{norm_x_t_dot_partial}
\end{gather}
we get
\begin{gather}
 \n r(t) \n_{X_0^2} \leq K (\n r(0) \n_{X^2_0} + \int_0^t (\n x(s) \n_Z + \n S x(s) 
 \n_{X_{\frac{1}{2},\partial}^{1\half } } )ds.
\label{norm_r_t}
\end{gather}
Furthermore, since restriction of functions to $\partial \Om$ gives a 
continuous map from $H^s(\Om)$ to $H^{s-\frac{1}{2}}(\partial \Om)$ ($s > \frac{1}{2}$), the 
term $\n S x(s) \n_{X_{\frac{1}{2},\partial}^{1\half} }$ is bounded by $\n x(s) \n_Z$. Therefore we 
can omit it in (\ref{norm_r_t}).

We would like to get a bound for $\n r(t) \n_{X_0^3}$, and to do so we apply `` $\dot{}$ '' to 
(\ref{sys_z_h_la_dif_abs_2}) and (\ref{bc_sys_z_h_la_dif_abs}). This gives
\begin{gather}
 \dddot{r} + (\nabla_x \dot{f})\dot{} + (\nabla_x f)\,\ddot{} = L\dot{r} + L_1 r + B_2 \dot{r} + B_{21} r + (A_2 x)\dot{},
\label{eq_of_previous_form}
\end{gather}
and 
\begin{gather}
 c^2 \nabla_\nu \dot{r} + c^2 \nabla_{[u,\nu]} r + (c^2)\dot{} \,\nabla_\nu r + B_\partial \dot{r} + B_{\partial 1} r =
(Sx)\dot{} \text{ on } \partial \Om,
\nonumber
\end{gather}
where $L_1$ is defined in (\ref{definition_L_1}), $B_{21} r = (B_2 r)\dot{} - B_2 \dot{r}$ 
and $B_{\partial 1} r = (B_\partial r)\dot{} - B_\partial \dot{r}$. (\ref{eq_of_previous_form}) is an 
equation of form (\ref{eq_M}) for $\dot{r}$ with inhomogeneous terms
\begin{gather}
 F = -(\nabla_x \dot{f} )\dot{} - (\nabla_x f)\,\ddot{} + L_1 r + B_{21} r + (A_2 x)\dot{} 
\nonumber
\end{gather}
and
\begin{gather}
 G = -c^2\nabla_{[u,\nu]} r + (c^2)\dot{}\, \nabla_\nu r - B_{\partial 1} r + (Sx)\dot{} \,\,.
\nonumber
\end{gather}
Since $\n (A_2 x)\dot{} \n_{X_0^1} \leq K \n x \n_Z$ and other terms are sufficiently smooth, from 
(\ref{estimate_r_t_X_0_k}) we get
\begin{gather}
 \n \dot{r} (t) \n_{X_0^2} \leq K ( \n \dot{r}(0) \n_{X_0^2} + \int_0^t( \n x(s) \n_Z + \n r(s) \n_{X_0^3} ) ds.
\label{estimate_r_dot_t_X_0_3}
\end{gather}
Combining (\ref{estimate_r_dot_t_X_0_3}) and (\ref{norm_x_t_dot_partial}) with the equations 
(\ref{eq_of_previous_form}) and (\ref{sys_z_h_la_dif_abs_2}) we obtain
\begin{gather}
 \n r \n_{X^3_0} \leq K ( \n r(0) \n_{X_0^3} + \int_0^t \n x(s) \n_Z ds ).
\label{estimate_r}
\end{gather}
The inequalities (\ref{estimate_x_basic}) and (\ref{estimate_r}) show that if $T$ is small enough we can solve 
(\ref{sys_z_h_la_dif_abs_1})-(\ref{bc_sys_z_h_la_dif_abs}) by an iteration. Given initial data $(x_0,r_0)$ we 
let $x_1$ be a curve in $Z$ starting at $x_0$. Then let $r_1$ be the solution 
of (\ref{sys_z_h_la_dif_abs_2})-(\ref{bc_sys_z_h_la_dif_abs})
with initial data $r_1(0) = r_0$, $\dot{r}_1(0) = \de x_0 - \nabla_{x_0} f_0$ and with $x_1$ in place of $x$. Let
\begin{gather}
 x_2(t) = U(t,0)x_0 + \int_0^t(U(t,s)(A_1 r_1(s) + B_1 r_1(s) )ds
\nonumber
\end{gather}
and let $r_2$ be the solution of (\ref{sys_z_h_la_dif_abs_2})-(\ref{bc_sys_z_h_la_dif_abs}) with the same initial 
data but with $x_2$ in place of $x$. Continue inductively to get sequences $\{ x_n \} \subset Z$ and $\{ r_n \} \subset X_0^3$. 
Clearly
\begin{gather}
 (x_n - x_{n-1})\dot{} + \nabla_{x_n - x_{n-1}} u = A_1(r_{n-1} - r_{n-2}) + B_1 (r_{n-1} - r_{n-2})
\nonumber
\end{gather}
and $x_n(0) - x_{n-1}(0) = 0$. Hence by (\ref{estimate_x_basic}) we get
\begin{gather}
 \n x_n - x_{n-1} \n_Z \leq K \n r_{n-1} - r_{n-2} \n_{X_0^3}.
\nonumber
\end{gather}
Similarly from (\ref{estimate_r}) we find
\begin{gather}
 \n r_n - r_{n-1} \n_{X_0^3} \leq K \int_0^t \n x_n(s) - x_{n-1}(s) \n_Z ds.
\nonumber
\end{gather}
From these two inequalities if follows that if $K^2T < 1$, $\{ (x_n,r_n) \}$ is a Cauchy sequence  in $Z \times X_0^3$. Its limit, 
which we call $(z,h)$, is clearly a solution of (\ref{sys_z_h_la_dif_abs_1})-(\ref{bc_sys_z_h_la_dif_abs}) or equivalently of 
(\ref{sys_z_h_1})-(\ref{ic_sys_z_h}). Since the system is linear we can piece together short time solutions to get a solution on 
any interval. Thus we can drop the restriction $K^2 T < 1$, and get a solution on any interval $[0,T]$ on which $u$ and $f$ are defined.

Having solved the homogeneous system (\ref{sys_z_h_la_dif_abs_1})-(\ref{bc_sys_z_h_la_dif_abs}) we can solve the 
inhomogeneous system by the usual application of Duhamel's formula. Then using the inequalities 
(\ref{estimate_x_basic}) and (\ref{estimate_r}) with inhomogeneous terms included we find that the solution of 
(\ref{sys_z_h_la_dif_abs_1})-(\ref{bc_sys_z_h_la_dif_abs}) obeys
\begin{gather}
 \n x \n_Z \leq K (\n x(0) \n_2 + \n r \n_{X_0^3} + \n E_1 \n_{X_0^2} )
\nonumber
\end{gather}
and 
\begin{gather}
 \n r \n_{X_0^3} \leq K \big ( \n r(0) \n_{X_0^3} + \int_0^T ( \n x(t) \n_Z 
+ \n E_2(t) \n_{X_0^2} + \n E_\partial (t) \n_{X_{\frac{1}{2},\partial}^{2\half }} ) dt \big )
\nonumber
\end{gather}
Iterating these inequalities we find a new $K$ such that
\begin{gather}
 \n x \n_Z + \n r \n_{X_0^3} \leq K \big ( \n x(0) \n_2 + \n r(0) \n_{X_0^3} + \n E_1 \n_{X_0^2} + \n E_2 \n_{X_0^2} 
+ \n E_\partial \n_{X_{\frac{1}{2},\partial}^{2\half }} \big )
\label{estimate_x_and_r_together}
\end{gather}

Now we are ready to show that $(z^\la - z, h^\la - h)$ is $O(\la)$ in $Z \times X_0^3$. We simply consider the system
(\ref{sys_z_h_la_difference_1})-(\ref{def_E_partial}) with initial 
conditions $z^\la_0 - z_0$, $h^\la_0 - h_0$ and $\de(z_0^\la - z_0)-\nabla_{z^\la_0 - z_0}f_0$. 
These initial conditions are $O(\la)$ in $H^3\times H^4 \times H^3$ because $(u_0^\la,f_0^\la)$ 
is a $C^1$-curve in $X^{3,4,4}$. Hence we need only show that the last three terms of 
(\ref{estimate_x_and_r_together}) are $O(\la)$. This one does in the same way as in the 
estimate (\ref{estimate_E_s_1_O_la}) of $\n E \n_{s-1}$. We omit the details.
\end{proof}

\section{Asymptotic Approximation to Compressible Fluid Motion \label{assymp} }

In this section we show how the equation (\ref{int_eq_zeta}) can be used to find approximate solutions to the 
compressible fluid motion problem. We will construct a sequence $\{ \zeta_n \}$ of curves in $\mathscr{D}^3$ which will 
approach the actual fluid motion $\zeta(t)$. The methods used to estimate the $\{ \zeta_n \}$ are some of the methods of 
sections \ref{estimates} 
through \ref{diff_dep_quasi}. Hence our presentation will be brief.

First we note that the estimate (\ref{difference_eta_dot_zeta_dot_k_iteration}) shows that incompressible motion is itself an 
approximation of order $\frac{1}{\sqrt{k}}$. We let $(u_{0k},\rho_{0k})$ be as in Theorem \ref{first_big_theo}, 
let $v_0 = Pu_{0k}$ and let $\eta$ and $\zeta$ be incompressible and compressible fluid motions as in 
Theorem \ref{equiv_second_big_theo}. Then by
(\ref{difference_eta_dot_zeta_dot_k_iteration})
\begin{gather}
 \n \dot{\eta}(t) - \dot{\zeta}(t) \n_3 \leq K \frac{(1 + t)}{\sqrt{k}} 
+ \int_0^t Z(\eta(\tau),\dot{\eta}(\tau)) - Z(\zeta(\tau),\dot{\zeta}(\tau) ) \n_3 d\tau.
\nonumber
\end{gather}
Iterating this inequality we get
\begin{gather}
 \n (\eta(t),\dot{\eta}(t)) - (\zeta(t),\dot{\zeta}(t)) \n_3 \leq \frac{K}{\sqrt{k}}.
\label{app_estimate}
\end{gather}
To get better approximations we will construct a sequence of curves $\zeta_n(t)$ such that 
\begin{gather}
 \n (\zeta_n(t),\dot{\zeta}_n(t) ) - (\zeta_{n-1}(t),\dot{\zeta}_{n-1}(t) )\n_3 \leq K k^{-\frac{n}{2}}
\label{app_better}
\end{gather}
and 
\begin{gather}
 \n (\zeta(t),\dot{\zeta}(t) ) - (\zeta_{n}(t),\dot{\zeta}_{n}(t) )\n_3 \leq K k^{-\frac{n+1}{2}}.
\label{app_better_2}
\end{gather}
To begin we simply let $\zeta_0(t) = \eta(t)$ be the incompressible motion. Then we define the $\zeta_n$ inductively as follows:

Let $f_0(t) = \log\rho_{0k}(t)$. Then let $f_n$ be the solution of
\begin{gather}
\ddot{f}_n = -\de c^2(f_n) \nabla f_n + (u_{n-1})_j^i(u_{n-1})^j_i  \, ,\label{app_def_f_n} \\
c^2(f_n) \nabla_\nu f_n = -S_2(u_{n-1},u_{n-1}) \text{ on } \partial\Om,
 \label{app_bc_def_f_n} \\
f_n(0) = \log \rho_{0k},~ \dot{f}(0) = \de u_{0k}.
\nonumber
\end{gather}
where $u_{n-1} = \dot{\zeta}_{n-1}\circ (\zeta_{n-1})^{-1}$ as in (\ref{zeta_dot_u})-(\ref{u_zeta_dot_compose_zeta_inverse}) 
and `` $\dot{}$ '' means $\partial_t + \nabla_{u_{n-1}}$. Let $\nabla g_n(t) = \nabla \Delta^{-1} \dot{f}_n$ and let $\zeta_n$ be 
the solution of
\begin{gather}
\begin{cases}
 \dot{\zeta}_n(t) = Pu_{0k} + \int_0^t \tilde{Z}(\zeta_n(s),\dot{\zeta}_n(s))ds 
+ \nabla g_n(t) \circ \zeta_n(t),  \\
\zeta_n(0) = \operatorname{id},
\end{cases}
\label{app_zeta_tilde_Z}
\end{gather}
where $\tilde{Z}$ is defined by: $\tilde{Z}(\xi,\al) = Z(\xi,\al) - R(\xi,\al)$ (cf. section \ref{proofs}).

Since $\nabla g_n(t) \in H^4$, the map $\xi \mapsto \nabla g_n(t)\circ \xi$ is $C^1$ as a map from 
$\mathscr{D}^3$ to $H^3(\Om,\RR^n)$. Also $\tilde{Z}$ is a smooth map from $\mathscr{D}^3 \times H^3 \rar H^3$. 
Therefore from standard estimates involving the Lipshitz constants of $\tilde{Z}$ and $\Om_{\nabla g(t)}$ we find
\begin{gather}
 \n (\zeta_n(t),\dot{\zeta}_n(t)) - (\zeta_{n-1}(t),\dot{\zeta}_{n-1}(t)) \n_3 \leq K \n \nabla g_n(t) - \nabla g_{n-1}(t) \n_3
\label{app_zeta_g}
\end{gather}
However $\n \nabla g_n(t) - \nabla g_{n-1}(t) \n_3 \leq K \n \dot{f}_n - \dot{f}_{n-1} \n_2$ and using estimates like 
those of sections \ref{estimates} and \ref{diff_dep_quasi} we find that
\begin{gather}
 \n \dot{f}_n(t) - \dot{f}_{n-1}(t) \n_2 \leq \frac{K}{\sqrt{k}} \n u_{n-1} - u_{n-2} \n_3.
\label{app_f_u}
\end{gather}
Combining  (\ref{app_zeta_g}) and (\ref{app_f_u}) we get
\begin{gather}
 \n (\zeta_n,\dot{\zeta}_n) - (\zeta_{n-1},\dot{\zeta}_{n-1}) \n_3 \leq \frac{K}{\sqrt{k}}
\n (\zeta_{n-1},\dot{\zeta}_{n-1}) - (\zeta_{n-2},\dot{\zeta}_{n-2}) \n_3 ~~(n\geq 2).
\label{app_zeta_n_minus_zeta_n-1}
\end{gather}
Also $\zeta_0 = \eta$ so 
\begin{gather}
 \dot{\zeta}_0(t) = P u_{0k} + \int_0^t\tilde{Z}(\zeta_0(\tau),\dot{\zeta}_0(\tau))d\tau
\nonumber
\end{gather}
and
\begin{gather}
 \dot{\zeta}_1(t) = Pu_0 + \int_0^t\tilde{Z}(\zeta_1(\tau),\dot{\zeta}_1(\tau))d\tau + \nabla g_1(t)\circ \zeta_1(t).
\nonumber
\end{gather}
But since $f_1$ satisfies (\ref{app_def_f_n})-(\ref{app_bc_def_f_n}) we know $\n \dot{f}_1 \n\leq \frac{K}{\sqrt{k}}$ so 
$\n \nabla g_1(t) \n_3 \leq \frac{K}{\sqrt{k}}$. Hence it follows as before that 
\begin{gather}
 \n (\zeta_1,\dot{\zeta}_1) - (\zeta_0,\dot{\zeta}_0) \n_3 \leq \frac{K}{\sqrt{k}}
\label{app_zeta_1_minus_zeta_0}
\end{gather}
(\ref{app_zeta_n_minus_zeta_n-1}) and (\ref{app_zeta_1_minus_zeta_0}) together imply (\ref{app_better}).

Now we prove (\ref{app_better_2}). If $n=0$ (\ref{app_better_2}) is simply (\ref{app_estimate}). For $n>0$, we note that
$\zeta$ satisfies (\ref{app_zeta_tilde_Z}) with $g$ replacing $g_n$. Thus we get as before
\begin{gather}
 \n ( \zeta,\dot{\zeta}) - (\zeta_n,\dot{\zeta}_n) \n_3 \leq K \n \nabla g - \nabla g_n \n_3 \leq K \n \dot{f} - \dot{f}_n \n_2.
\label{app_zeta_zeta_n_2nd_a}
\end{gather}
Also
\begin{gather}
 \n \dot{f} - \dot{f}_n\n_2 \leq \frac{K}{\sqrt{k}}\n u - u_{n-1} \n_3 \leq \frac{K}{\sqrt{k}}
\n  (\zeta,\dot{\zeta}) - (\zeta_{n-1},\dot{\zeta}_{n-1}) \n_3 .
\label{app_f_f_n_2nd_a}
\end{gather}
(\ref{app_better_2}) follows from (\ref{app_zeta_zeta_n_2nd_a}) and (\ref{app_f_f_n_2nd_a}) by induction.
\begin{rema}
 Equation (\ref{eq_f_convected}) which is (\ref{app_def_f_n}) without the subscripts ``$n$'' and ``$n-1$'' is the 
equation of sound on a fluid moving with velocity $u$. Acoustical engineers sometimes find an approximate 
solution to this equation by solving 
(\ref{app_def_f_n}) with $n=1$ (cf. \cite{H}).
\end{rema}

\appendix

\section{Function spaces and some auxiliary results}

In this appendix we review some basic constructions and known results.

Given $g$ and $h$, $C^k$ functions from $\Om$ to $\RR$, we define the inner product
\begin{gather}
 (g,h)_k = \int_\Om \sum_{\ell=0}^k \langle \nabla^\ell g,\nabla^\ell h\rangle
\nonumber
\end{gather}
where $\nabla^\ell g$ is the vector valued function consisting of all $\ell^{\text{th}}$ order partial derivatives of $g$. This 
inner product induces a norm (which we call $\parallel ~\parallel_k$) on $C^k(\Om,\RR)$, the set of all $C^k$ 
functions on $\Om$. $H^k(\Om,\RR)$ is defined to be the completion of $C^k(\Om,\RR)$ in this 
norm. $C^k(\Om,\RR^m)$ and $H^k(\Om,\RR^m)$ will be analogous spaces for functions with values in $\RR^m$. Sometimes we 
write only $H^k$. We shall also use the norm
\begin{gather}
\parallel g \parallel_{C^k} = \sup_{x\in\Om} \Big ( \sum_{\ell=0}^k |\nabla^\ell g(x)| \Big )
\nonumber
\end{gather}
the usual $C^k$ norm on $C^k(\Om,\RR)$. It is well known that in this norm $C^k(\Om,\RR)$ is complete.

By the Sobolev embedding theorem, $H^s(\Om,\RR^m) \subset C^k(\Om, \RR^m)$ if $s > k+ \frac{n}{2}$ and for any
$w \in H^s(\Om,\RR^m)$:
\begin{gather}
 \parallel w \parallel_{C^k} \leq K \parallel w \parallel_s
\nonumber
\end{gather}
where $K$ depends only on $k,~s$ and $\Om$. (see \cite{T}, chapter 4.)  Also, if $s > \frac{n}{2}$, $h \in H^s$ and $g \in H^{s_1}$, with $s\geq s_1 \geq 0$, then
$gh \in H^{s_1}$ and we have the product estimate
\begin{gather}
 \parallel gh \parallel_{s_1} \leq K \parallel g \parallel_{s_1} \parallel h \parallel_s
\label{product_estimate}
\end{gather}
where $K$ depends on $s_1$, $s$ and $\Om$. (see \cite{T}, chapter 13.)

Equation (\ref{product_estimate}) will be used throughout the paper to estimate
the several products involved. 

Since $\partial \Om$ is a compact manifold we can also define $C^k(\partial \Om, \RR)$, the space of $C^k$-functions 
on $\partial \Om$, and completing this space with respect to an appropriate inner product we get the 
Hilbert space $H^k(\partial \Om, \RR)$. Analogously one constructs 
$H^k(\partial \Om, \RR^m)$.

Using Fourier series we can construct $H^s(\Om)$ and $H^s(\partial \Om)$ for all non-negative real numbers $s$ (cf. \cite{P}), 
and with this we get a restriction inequality as follows. 
Let $R: C^s(\Om) \rar C^s(\partial \Om)$ be defined by restricting each $C^s$ function to its values on $\partial \Om$. As is 
shown in \cite{P}, if $s > \frac{1}{2}$, $R$ extends to a bounded linear map $R: H^s(\Om) \rar H^{s-\frac{1}{2}}(\partial \Om)$.
Denote by  $\parallel ~ \parallel_{\partial, s}$ the norm of $H^s(\partial \Om)$. Then we get the inequality
\begin{gather}
 \parallel Rh \parallel_{\partial,s-\frac{1}{2}} \leq K \parallel h \parallel_s,\,\, s > \frac{1}{2},
\label{restrictions_inequality}
\end{gather}
where $K$ depends on $s$ and $\Om$. (\ref{restrictions_inequality}) will
be used throughout the paper to estimate the restrictions to $\partial \Om$.

If $s > \frac{n}{2} + 1$ we define $  \mathscr{D}^s(\Om) \equiv \mathscr{D}^s$ to be the set of bijective maps from $\Om$ to itself which together 
with their inverses are in $H^s(\Om,\RR^n)$. It is easy to prove the following:
\begin{gather}
\mathscr{D}^s = \Big \{ \zeta \in H^s(\Om,\RR^n) ~\big | ~ \zeta: \Om \rar \Om \text{ is bijective, and }
J(\zeta) \text{ is nowhere zero } \Big  \}
\nonumber
\end{gather}
where $J$ is the Jacobian. Also $\mathscr{D}^s$ can be shown to be a subgroup of the group of $C^1$ diffeomorphisms 
of $\Om$ (see \cite{EM} or \cite{E1}).

\end{document}